\documentclass[aos,noinfoline]{imsart_nj}

\RequirePackage{amsthm,amsmath,amsfonts,amssymb}
\RequirePackage[numbers]{natbib}
\RequirePackage[colorlinks,citecolor=blue,urlcolor=blue]{hyperref}
\RequirePackage{graphicx}

\startlocaldefs
\numberwithin{equation}{section}
\theoremstyle{plain}

\newtheorem{theorem}{Theorem}[section]
\newtheorem{lemma}[theorem]{Lemma}
\newtheorem{corollary}[theorem]{Corollary}
\newtheorem{proposition}[theorem]{Proposition}
\theoremstyle{remark}

\newtheorem{remark}[theorem]{Remark}

\endlocaldefs

\newenvironment{proofof}{\noindent\sc Proof of}{
    \hspace*{\fill} $\Box$ \vspace{2ex} }

\newcommand{\R}{\mathbb{R}}
\newcommand{\N}{\mathbb{N}}

\newcommand{\Z}{\mathbb{Z}}

\renewcommand{\AA}{\mathcal{A}}

\newcommand{\FF}{\mathcal{F}}
\newcommand{\ZZ}{\mathcal{Z}}
\newcommand{\UU}{\mathcal{U}}
\newcommand{\RR}{\mathcal{R}}
\newcommand{\XX}{\mathcal{X}}

\newcommand{\VC}{\mathcal{V}}

\newcommand{\Dt}{\tilde D}
\newcommand \St{S'}

\newcommand{\bfx}{\mathbf{x}}
\newcommand{\bfX}{\mathbf{X}}
\newcommand{\bfz}{\mathbf{0}}
\newcommand{\bfy}{\mathbf{y}}

\newcommand{\Ind}[1]{{\boldsymbol{1}_{\textstyle\{#1\}}}}
\newcommand{\Indd}[1]{{\boldsymbol{1}_{\{#1\}}}}
\newcommand{\ind}[1]{{\boldsymbol{1}_{\textstyle #1}}}
\newcommand{\indd}[1]{{\boldsymbol{1}_{#1}}}

\renewcommand{\theenumi}{\roman{enumi}}

\def\eps{\varepsilon}

\def\Min(#1,#2){#1\wedge #2}
\def\Max(#1,#2){#1\vee #2}
\newcommand{\ceil}[1]{\lceil #1 \rceil}
\newcommand{\ceill}[1]{\left\lceil #1 \right\rceil}
\newcommand{\floor}[1]{\lfloor #1 \rfloor}
\newcommand{\floorr}[1]{\left\lfloor #1 \right\rfloor}

\def\Var{{\rm Var}}
\def\Cov{{\rm Cov}}

\def\e{{\rm e}}

\def\sphere{{\mathcal S}^{d-1}}

\def\unji{{u_{n,j}^{(i)}}}
\def\unjik{{u_{n,j}^{(i,k)}}}
\def\Snt{{S^{(n)}_t}}
\def\Snrj{{S^{(n)}_{r_j}}}

\allowdisplaybreaks

\begin{document}
\begin{frontmatter}
\title{Statistical Inference on a Changing Extreme Value Dependence Structure}
\runtitle{Changing Extremal Dependence}

\begin{aug}
\author[A]{\fnms{Holger} \snm{Drees}\ead[label=e1]{holger.drees@uni-hamburg.de}}
\address[A]{Department of Mathematics,
University or Hamburg,
\printead{e1}}
\end{aug}

\begin{abstract}
  We analyze the extreme value dependence of independent, not necessarily identically distributed multivariate regularly varying random vectors. More specifically, we propose estimators of the spectral measure locally at some time point and of the spectral measures integrated over time. The uniform asymptotic normality of these estimators is proved under suitable nonparametric smoothness and regularity assumptions. We then use the process convergence of the integrated spectral measure to devise consistent tests for the null hypothesis that the spectral measure does not change over time. 
\end{abstract}
\end{frontmatter}

\section{Introduction}

When analyzing the extreme value behavior of serial data, one usually assumes that the underlying time series is stationary or, somewhat weaker, that its extreme value behavior does not change over time, an assumption which is not always guaranteed to be fulfilled. The most obvious examples of data with potentially changing extreme value behavior are environmental time series over a long time horizon, but also the stationarity of returns of a financial investment should not be taken for granted, if the economic environment, the regulatory framework or the trading technology develop.

Statistical inference of the extreme value behavior of univariate random variables that depend nonparametrically on covariates (like time) has been investigated by \cite{GGL10}, \cite{DGG13} and \cite{GGS14}, among others. The basic idea is to apply localized versions of standard extreme value estimators (constructed for iid data) to observations with covariates near the point of interest. Because the neighborhood must shrink towards this point to ensure consistency and, at the same time, only extreme observations can be used, for any fixed value of the covariate only a very small fraction of the data influence the extreme value estimators, which thus converge at a slow rate. To avoid this problem, one may, of course, model the extreme value behavior as a parametric function of the covariates. This approach has often be used in extreme quantile regression; see, e.g., \cite{WangLi13} and \cite{ADDGUC20}. Since usually there is no physical justification of such a parametric relationship, tests for these model assumptions are of crucial importance.
To the best of our knowledge, the first test of such a hypothesis has been suggested and analyzed by \cite{dHZ21}, who estimated the local extreme value indices of a nonstationary sequence of independent heavy-tailed random variables locally by a Hill type estimator and then used integrals of these local estimators to devise tests for a constant extreme value index.

Similar, but even more serious problems arise if one wants to analyze the extreme value dependence of a nonstationary sequence of independent random vectors. If one assumes the random vectors to be multivariate regularly varying, then the extreme value dependence is captured by the so-called spectral or angular measures; see Section 2 for a precise definition. Since these measures constitute a nonparametric class, even for iid data the extreme value analysis of the dependence is substantially more involved than the marginal analysis. Therefore, any localized version of nonparametric estimators of the spectral measure will require large data sets to perform well. A localized estimator of the so-called Pickands dependence function, which describes the spectral measure after a suitable marginal standardization, has been analyzed by \cite{EBGG18} in a general setting. In \cite{dCD14} a baseline spectral measure is modeled nonparametrically, while a parametric model is used for the temporal development of the dependence structure. Both papers analyze  the asymptotic behavior of their estimator  under the restrictive assumption that the respective model describes the dependence structure for all observations, so that there is no need to use only extremes. In addition, \cite{dCD14} assumes that the marginal distribution is known. There are also papers which propose some models for a changing dependence and procedures for fitting such a model without investigating their performance mathematically; see, e.g., \cite{CCCW18} and
\cite{Hoga21}. Moreover, some publications like \cite{GG15} examine the asymptotic behavior of estimators of the dependence structure locally around a fixed point in a regression setting, but do not consider the global behavior over the full time interval, so that tests for a changing dependence are unfeasible.

Here we want to analyze the asymptotic behavior both of a localized version of a well-known nonparametric estimator of the spectral measure and of an integrated version thereof in a setting with nonstationary independent multivariate regularly varying observations. Unlike in previous approaches, we do not restrict ourselves to the estimation of certain cdf's which determine the spectral measure, but we prove uniform convergence over quite general families of sets which form a Vapnik-\v{C}ervonenkis (VC) class. The major application we have in mind are tests for (parametric) assumptions about trends in the dependence structure.

In Section 2, we first introduce the setting and the estimators of the spectral measure and the integrated spectral measure. Then the uniform asymptotic normality both of the local estimators at a fixed time point and of the integrated estimators over families of sets and (for the integrated version) over time is established under natural smoothness assumptions. Similarly as in \cite{dHZ21}, it turns out that the estimated spectral measure integrated over time converges at a much faster rate than the local estimator. In Section 3 we discuss how to employ these limit results to test for a constant dependence structure. The finite sample performance of these tests is examined in Section 4. In Section 5 we apply our tests to a well-known actuarial data set, while in Section 6 possible extensions of our results are discussed.
While the proofs of our main results are postponed until the last section, the proofs of the results of Section 2.1 and some technical lemmas are given in the Supplement \cite{Drees22}.

\subsection*{Notation} Throughout the paper, all random elements are defined on some probability space $(\Omega,\FF,P)$. The distribution of some random element $Z$ is denoted by $P^Z$, its conditional distribution given some event $B$ with $P(B)>0$ by $P^{Z|B}$. The maximum and minimum of two reals $x$ and $y$ are denoted by $\max(x,y)=x\vee y$ and $\min(x,y)=x\wedge y$, respectively. Moreover, we define $\floor{x}:=\max\{l\in\Z|l\le x\}$ and $\ceil{x}:=\min\{l\in\Z|l\ge x\}$. The symmetric difference between two sets $A$ and $B$ is denoted by $A\triangle B:= (A\setminus B)\cup (B\setminus A)$.

\section{Estimators of the (integrated) spectral measure}

Throughout,
$X_t$, $t\in[0,1]$, will denote independent, but not necessarily identically distributed $\R^d$-valued random vectors. For some norm $\|\cdot\|$ on $\R^d$, let
$$ R_t := \|X_t\|, \qquad \Theta_t := \frac {X_t}{\|X_t\|}. $$
We assume that, for all $t\in[0,1]$, $X_t$ is regularly varying, i.e.\ there exist $\alpha_t>0$ and a so-called  {\em spectral (probability) measure} $S_t$ on the unit sphere  $\sphere=\{x\in\R^d\mid \|x\|=1\}$ such that
\begin{equation} \label{eq:multRV}
  \lim_{u\to\infty} \frac{P\{R_t>ru, \Theta_t\in A\}}{P\{R_t>u\}} = r^{-\alpha_t} S_t(A)
\end{equation}
for all $r>0$ and all Borel sets $A\subset\sphere$ with $S_t(\partial A)=0$. (If this condition is fulfilled for some norm, then it holds for any other norm, too, but of course the spectral measure depends on the specific choice of the norm.) While the {\em index of regular variation} $\alpha_t$ describes the marginal tail behavior, the spectral measure captures the extremal dependence structure between the components of the vector $X_t$.

Suppose that $X_{i/n}$, $1\le i\le n$, are observed. (The equidistant observation scheme is not essential, but mainly due to notational convenience; see Section \ref{sect:outlook} for details.) We assume that the distribution of $X_t$ varies smoothly with $t$ (at least in extreme regions) to ensure that observations near a fixed time point $t$ contain sufficient information about the extreme value behavior of $X_t$.
In \cite{dHZ21} it is discussed how to estimate $\alpha_t$ (or its reciprocal value) and integrals thereof. Here we focus on estimators of the spectral measure $S_t$.

Our first goal is to estimate $S_t(A)$ jointly for all sets $A$ belonging to a suitable family $\AA$ of Borel sets in $\sphere$ and some fixed $t\in [0,1]$. To this end, we apply a standard nonparametric estimator of the spectral measure to observations in a neighborhood of $t$.

Fix some sequence of bandwidths $h_n\to 0$ such that $n h_n\to \infty$ and denote by
$$I_{n,t} := \{i\in\{1,\ldots,n\}\mid i/n \in (t- h_n, t+h_n]\}$$
the set of indices of those observations that lie in a $h_n$-neighborhood of $t$. Since $S_t$ is defined as a limit distribution as the norm of the vector tends to infinity, we may only use the largest observations among the $X_i$, $i\in I_{n,t}$, to estimate $S_t$. To this end, fix some intermediate sequence $k_n\to \infty$ so that $k_n=o(nh_n)$ and let $\hat u_{n,t}$ be the $k_n+1$ largest order statistic among $R_{i/n}$, $i\in I_{n,t}$. Our estimator of $S_t(A)$  is then defined by
\begin{equation*}
  \hat S_{n,t}(A) := \frac 1{k_n} \sum_{i\in I_{n,t}} \Ind{R_{i/n}>\hat u_{n,t}, \Theta_{i/n}\in A},
\end{equation*}
 with $A\in\AA$. It will turn out that under suitable conditions $\hat S_{n,t}(A)$ converge to $S_t(A)$ at the rate $k_n^{-1/2}$ uniformly over $A\in\AA$. Here one must choose $h_n$ small enough such that the distributions of $X_{i/n}$, $i\in I_{n,t}$, are sufficiently close to the one of $X_t$. Moreover, $k_n$ must be sufficiently small in comparison with $2nh_n$ so that only extreme observations enter the estimators. Hence usually the rate of convergence of this estimator will be quite slow, which is of course due to the purely nonparametric setting we work with.

Consequently, a test whether the spectral measure does not change over time based on the above local estimators will detect only rather large changes. As a naive alternative, one may think of comparing estimators of an ``average spectral measure'' over different intervals $[t_1,t_2]$ defined by the projections of all observations $X_{i/n}$ with $i/n\in[t_1,t_2]$ and radius exceeding some threshold onto the unit sphere. However, this approach only makes sense if the distribution of the radius does not change over time. If, for instance, the radii of observations in the first half of a time interval have substantially lighter tails than in the second half, then asymptotically only observations in the second half will be used in the estimator and changes in the dependence structure in the first half will not be detected.

To obtain a more reliable picture how the spectral measure evolves over time (in particular to construct tests of a stationary dependence), we thus consider estimators of the {\em integrated spectral measure}
$$IS_t(A):=\int_0^t S_r(A)\, dr$$
 for all $t\in[0,1]$ and $A\in\AA$.  To be more specific, we integrate our estimator of the spectral measure after a suitable  discretization of  the time parameter, which  ensures that the estimator can be written as a sum with independent summands:
\begin{equation*}
  \widehat{IS}_{n,t}(A) := \int_0^t \hat S_{n,(2\ceil{r/(2h_n)}-1)h_n}(A)\, dr
\end{equation*}
with $\hat S_{n,r}:=\hat S_{n,1}$ for $r>1$. The crucial difference to the aforementioned naive approach is that now at different time points different thresholds are used to define which observations are considered ``large''.

There is no obvious interpretation of the integrated spectral measure, and to the best of our knowledge it has not been considered before. However, in nonparametric regression, it is a time-tested approach to examine integrated versions of functions of interest, in order to devise better statistical tests; see e.g.\ \cite{Delgado93} or \cite{GMC13}, p.\ 368 and 380, and the literature cited therein.

For simplicity, we assume throughout the paper that, for some $u_0>0$, all marginal cdf's $F_t$ of $R_t$, $t\in[0,1]$, are continuous on $(u_0,\infty)$. Moreover, w.l.o.g.\ we assume that $\AA$ comprises $\sphere$.

\subsection{Asymptotic behavior of $\hat S_{n,t}$ for fixed $t\in[0,1]$}
\label{sect:fixedt}

Fix some $t\in[0,1]$  and some intermediate sequence $k_n=o(nh_n)$, and define
\begin{align*}
   u_{n,t}  := \inf \{u>0 \mid \nu_{n,t}(u,\infty)\le k_n\} \quad \text{with} \quad \nu_{n,t} & := \sum_{i\in I_{n,t}} P^{R_{i/n}}.
\end{align*}
Since $u_{n,t}\to\infty$ and $F_{i/n}$ is continuous on $(u_0,\infty)$,  one has $\nu_{n,t}(u_{n,t},\infty)=k_n$ for sufficiently large $n$.

We assume that there exists some $\eps>0$ such that the following conditions hold:

{\bf (RV1)}\;  \parbox[t]{12cm}{$\displaystyle \sup_{s\in[1-\eps,1+\eps]} \bigg| \frac{\nu_{n,t}(su_{n,t},\infty)}{k_n}-s^{-\alpha_t}\bigg| \to 0.$}

{\bf (RV2)}\; \parbox[t]{12cm}{$\displaystyle \sup_{s\in[1-\eps,1+\eps],A\in\AA^*} k_n^{1/2}\bigg| \frac{\sum_{i\in I_{n,t}} P\{R_{i/n}>s u_{n,t}, \Theta_{i/n}\in A\}}{\nu_{n,t}(su_{n,t},\infty)}-S_t(A)\bigg| \to 0$\\
  with $\AA^* :=\{A_1 \triangle A_2,A_1\cap A_2\mid A_1, A_2\in\AA\}.$}

{\bf (A)}\; \parbox[t]{13cm}{The family $\AA$ of subsets of $\sphere$ forms a VC class with VC-index $\VC$ and it is totally bounded w.r.t.\ the semi-metric $\rho_t(A,B) := S_t(A\triangle B)$, $A,B\in\AA$. Moreover, for all $n\in\N$, $r\in [0,1]$ and $i\in I_{n,r}$, the processes $[1-\eps,1+\eps]\times \AA\ni (s,A)\mapsto \Indd{R_{i/n}>su_{n,r}, \Theta_{i/n}\in A}$ are separable.}

Roughly speaking, the first two assumptions ensure that, in the average, the extreme value behavior of $X_{i/n}$ for $i\in I_{n,t}$ is similar to the one of $X_t$ and that the  approximation suggested by \eqref{eq:multRV} is sufficiently accurate for the thresholds under consideration. They are fulfilled if
$$  \frac{P\{R_r>s u_{n,t}\}}{P\{R_t>s u_{n,t}\}} \to 1, \quad
  P(\Theta_r\in A\mid R_r>su_{n,t})= S_t(A) + o(k_n^{-1/2})
$$
uniformly for $r\in[t-h_n,t+h_n]\cap[0,1]$, $A\in\AA^*$ and $s\in[1-\eps,1+\eps]$.

 Condition (A) restricts the complexity of the family $\AA$; see e.g.\ Section 2.6 of \cite{vdVW96} for an introduction to VC theory. Note that under Condition (A) the extended family $\AA^*$ is a VC class, too (\cite{vdVW96}, Lemma 2.6.17). A typical example of a family $\AA$ fulfilling Condition (A) is $\{[\bfz,\bfx]\cap\sphere \mid \bfx\in[0,\infty)^d\}$ if $\|\cdot\|$ is the $p$-norm for some $p\in[1,\infty)$. While the VC-property is well known (see e.g.\ Ex.\ 2.6.1 of \cite{vdVW96}), the total boundedness of $\AA$ w.r.t.\ $S_t$ follows from the fact that $\bfy\mapsto S_t([\bfz,pr^{-1}(\bfy)]\cap\sphere)$ (with $pr:\sphere\to\R^{d-1}$ denoting the projection onto the first $d-1$ coordinates) defines a cdf of some measure $G$ on $[0,1]^{d-1}$. To each $\eps>0$, one can decompose $[0,1]^{d-1}$ into  finitely many points of mass and strips of the form $[0,1]^{l-1}\times I\times [0,1]^{d-l-1}$ for some interval $I$, such that every strip has mass less than $\eps/d$ outside the given points of mass. It is then easily seen that each of the rectangles obtained by intersecting $d-1$ of these strips (and excluding the points of mass) form a set with diameter less than $\eps$ w.r.t.\ the semi-metric $\rho_G(\bfx,\bfy):=G\big([\bf0,\bfx]\triangle [\bf0,\bfy]\big)$. The total boundedness property of $\AA$ now follows from the above construction of $G$.

\begin{theorem} \label{theo:fixedt}
  If the Conditions (RV1), (RV2) and (A) are fulfilled then, for all $t\in [0,1]$, $k_n^{1/2}(\hat S_{n,t}(A)-S_t(A))_{A\in\AA}$ converges weakly to a centered Gaussian process $Z_t$ with covariance function $c_t(A,B):= \Cov(Z_t(A),$ $Z_t(B))= S_t(A\cap B)-S_t(A)S_t(B)$.
\end{theorem}
The proof is given in the Supplement \cite{Drees22}.

\subsection{Asymptotic behavior of $\widehat{IS}_{n,\cdot}$}
\label{sect:integratedSM}

First note that $\hat S_{n,r}$ and $\hat S_{n,t}$ are independent if $|t-r|\ge 2h_n$. Due to the discretization of the time index, the estimator of the integrated spectral measure can thus be written as a weighted sum of independent terms, which are in general not identically distributed:
$$ \widehat{IS}_{n,t}(A) = 2h_n \sum_{j=1}^{\floor{t/(2h_n)}} \hat S_{n,r_j} (A)+ \big(t-2h_n\floor{t/(2h_n)}\big)  \hat S_{n,r_{\floor{t/(2h_n)}+1}}(A),
$$
with
$$r_j := (2j-1)h_n.$$
 However, the approximation to $\hat S_{n,t}$ established in the preceding subsection is too crude to derive a non-trivial limit of the estimator of the integrated spectral measure. In fact, the techniques used to establish the asymptotics of the estimator of the spectral measure for a fixed time point are not applicable in the present context. Instead, we first analyze the pseudo-estimator
\begin{align*}
  \widetilde{IS}_{n,t}(A) & := \int_0^t \tilde S_{n,(2\ceil{r/(2h_n)}-1)h_n}(A)\, dr \nonumber\\
   & = 2h_n \sum_{j=1}^{\floor{t/(2h_n)}} \tilde S_{n,r_j}(A) + \big(t-2h_n\floor{t/(2h_n)}\big)  \tilde S_{n,r_{\floor{t/(2h_n)}+1}} (A)
\end{align*}
with
\begin{align}
 \tilde S_{n,r}(A) & := \frac 1{N_r} \sum_{i\in I_{n,r}} \Indd{R_{i/n}>u_{n,r}, \Theta_{i/n}\in A},\nonumber \\
 N_r & := \sum_{i\in I_{n,r}} \Indd{R_{i/n}>u_{n,r}}, \label{eq:Nrdef}
\end{align}
where the order statistic $\hat u_{n,t}$ has been replaced by the unknown quantile $u_{n,t}$. Here and in what follows, we use the convention $0/0:=0$. In a second step, we show that the difference to the actual estimator is asymptotically negligible.

To this end, we need the following conditions to hold for some $\eps>0$ and some sequences $q_n,q_n^*$ and $q_n'$ tending to 0:

{\bf (US)}\; \parbox[t]{12.5cm}{$\displaystyle \sup_{|r-t|\le h_n, t\in[0,1],A\in\AA^* } \big|P(\Theta_{r}\in A \mid R_r> u_{n,t})-P(\Theta_{t}\in A \mid R_t> u_{n,t})\big| =O(q_n) $}

{\bf (US*)}\; \parbox[t]{12cm}{$\displaystyle \sup_{|r-t|\le h_n, t\in[0,1],s\in[1-\eps,1+\eps],A\in\AA^* } \big|P(\Theta_{r}\in A \mid R_r> u_{n,t})$ \\ \hspace*{6cm} $\displaystyle  -P(\Theta_{r}\in A \mid R_r> su_{n,t})\big| = O(q_n^*) $}

{\bf (B)}\; \parbox[t]{12cm}{$\displaystyle \sup_{t\in[0,1],A\in\AA^* } \big|P(\Theta_{t}\in A \mid R_t> u_{n,t})-S_t(A)\big| = O(q_n') $}

{\bf (IS)}\; \parbox[t]{12cm}{The function $r\mapsto S_r(A)$ is continuous for all $A\in\AA^*$ and
 $$\displaystyle \sup_{A\in\AA^*}\Big(\frac{k_n}{h_n}\Big)^{1/2}\int_0^1 \big|S_{(2\ceil{r/(2h_n)}-1)h_n}(A) - S_r(A)\big|\, dr\to 0.$$}

{\bf (A$^*$)}\; \parbox[t]{12.8cm}{The family $\AA$ of subsets of $\sphere$ forms a VC-class with VC-index $\VC$ and it is totally bounded w.r.t.\ the semi-metric $\rho_I(A,B) := \int_0^1 S_r(A\triangle B)\, dr$, $A,B\in\AA$. Moreover, for all $n\in\N$, $r\in [0,1]$ and $i\in I_{n,r}$, the processes $\AA\ni A\mapsto \Indd{R_{i/n}>u_{n,r}, \Theta_{i/n}\in A}$ are separable.}

{\bf (R)}\; \parbox[t]{13cm}{$k_nh_n\to 0$,\; $\log^3 h_n=o(k_n)$, \; $\max(q_n,q_n^*)= o\big((k_n\log(k_n/h_n))^{-1}\big)$, and 
$\max(q_n,q_n',q_n^*)=$ $o\big((h_n/k_n)^{1/2}\big)$,}

{\bf (L)}\; \parbox[t]{13cm}{$\displaystyle
  \eta  := \liminf_{n\to\infty} \inf_{t\in[0,1], s\in\{1-\eps,1+\eps\}} \bigg| \frac{\nu_{n,t}(su_{n,t},\infty)}{k_n}-1\bigg|>0.
$}

(US) and (US$^*$) are uniform smoothness conditions. While (US) compares the conditional distributions of exceedances over the same high threshold at different time points, (US$^*$) compares the distributions of exceedances over slightly different thresholds at the same time. Condition (B) (jointly with the rate condition (R)) ensures that the bias of the estimator that is caused by the approximation of the limit in \eqref{eq:multRV} is asymptotically negligible. Condition (R) subsumes all conditions on the different rates. Of course, large parts of (R) could have been incorporated in the first three conditions, but the contributions of different types of approximation errors become more transparent  in the present formulation of the technical results of Section \ref{sect:proofs}. Note that the dimension $d$ does not occur in this set of rate conditions and so there is no inherent curse of dimensionality. However, in most cases the accuracy of the approximation (B) will deteriorate as the dimension increases, e.g.\ if $S_t$ has a density, because then a continuous measure is approximated by discrete ones whose points of mass become more sparsely scattered as the dimension increases.

Condition (IS) is satisfied if the spectral measure varies with $t$ sufficiently smoothly. In particular, if $t\mapsto S_t(A)$ is Lipschitz continuous uniformly in $A$, then (IS) is fulfilled if $k_nh_n\to 0$. Finally, condition (L) is a technical condition on the behavior of $R_t$ as a function of $t$, which is substantially weaker than the condition (RV1) used in the analysis of $\hat S_{n,t}$. All these conditions are verified for a semiparametric model with Fr\'{e}chet marginals and Gumbel copulas in Section 10 of the Supplement \cite{Drees22}.


The next result describes the asymptotic behavior of the pseudo-estimator.
\begin{proposition} \label{prop:IZtilde}
    Under the Conditions  (US), (B), (IS), (A$^*$) and (R), the  processes 
    $$ \widetilde{IZ}_{n,t}(A)  := \Big(\frac{k_n}{h_n}\Big)^{1/2}\big(\widetilde{IS}_{n,t}(A)-IS_{t}(A)\big), \qquad t\in[0,1],A\in\AA,
    $$
     converge to  a centered Gaussian process $IZ$ with covariance function \\ $\Cov(IZ_s(A),IZ_t(B))=2\int_0^{s\wedge t} c_r(A,B)\, dr$.
\end{proposition}

Our main result shows that replacing the unknown quantile $u_{n,t}$ with its empirical counterpart does not change the limit distribution of the estimator of the integrated spectral measure:
\begin{theorem} \label{theo:IZhat}
  Under the Conditions (US), (US$^*$), (B), (IS), (A$^*$), (R) and (L), the  processes
  $$ \widehat{IZ}_{n,t}(A)  := \Big(\frac{k_n}{h_n}\Big)^{1/2}\big(\widehat{IS}_{n,t}(A)-IS_{t}(A)\big), \qquad t\in[0,1],A\in\AA,
    $$
     converge to the centered Gaussian process $IZ$ described in Proposition \ref{prop:IZtilde}.
\end{theorem}

Note that $\widehat{IS}_{n,t}$ converges to the true integrated spectral measure at a much faster rate than $\hat S_{n,t}$ converges to the spectral measure. It is thus better suited to test hypotheses about changes of the spectral measure over time. In the next section, we discuss how Theorem \ref{theo:IZhat} can be employed to construct  consistent tests for the null hypothesis that the dependence structure does not change over time, i.e.\ that $S_t$ is the same measure for all $t\in[0,1]$.

\section{Testing for a changing dependence structure}
\label{sect:tests}

In this section, we derive tests for the null hypothesis that the spectral measure remains constant  over time, i.e.\ $S_t=S_1$ for all $t\in[0,1]$ (while the marginal distributions may change). If one assumes that the functions $t\mapsto S_t(A)$ are continuous for all sets of some measure determining family of sets $A$ (which we do for the sets $A\in\AA$ anyway),  then the null hypothesis can be rephrased as $IS_t=t \cdot IS_1$ for all $t\in[0,1]$.

In view of Theorem \ref{theo:IZhat}, it suggests itself to choose some functional of the process $\widehat{IS}_{n,t}-t \widehat{IS}_{n,1}$ as test statistic, e.g.\
\begin{align}
  T_n^{(KS)} & := \Big(\frac{k_n}{2h_n}\Big)^{1/2} \sup_{t\in[0,1],A\in\AA} \big|\widehat{IS}_{n,t}(A)-t \widehat{IS}_{n,1}(A)\big|, \label{eq:KStest}\\
  T_n^{(CM)} & := \frac{k_n}{2h_n} \sup_{A\in\AA} \int_0^1 \big(\widehat{IS}_{n,t}(A)-t \widehat{IS}_{n,1}(A)\big)^2\, dt. \label{eq:CMtest} 
\end{align}
\begin{remark}
  Instead of $ T_n^{(CM)}$, one may also consider a test statistic which integrates the supremum of the squared difference. However, it might be challenging to calculate this integral exactly, since, for fixed $i$, the set at which the supremum is attained need not be the same for all values $t\in [(i-1)2h_n,i2h_n]$.
\end{remark}

The following result establishes the asymptotic behavior of these test statistics under the null hypothesis and under fixed alternatives.
\begin{corollary} \label{cor:testconstancy}
  \begin{enumerate}
    \item If $S_t=S_1$ holds for all $t\in[0,1]$, and the Conditions (US), (US$^*$), (B), (A$^*$), (R) and (L) are fulfilled, then
        \begin{align*}
           T_n^{(KS)}  \to \sup_{t\in [0,1], A\in\AA} |\ZZ_t(A)|,\qquad
             T_n^{(CM)}  \to \sup_{A\in\AA} \int_0^1 \ZZ_t^2(A)\, dt
        \end{align*}
       weakly for a centered Gaussian process $\ZZ = (\ZZ_t(A))_{t\in[0,1],A\in\AA}$ with covariance function $Cov(\ZZ_s(A),\ZZ_t(B))=(s\wedge t-st) (S_1(A\cap B)-S_1(A)S_1(B))$.
    \item If $IS_t(A)\ne t\cdot IS_1(A)$ holds for some $t\in[0,1]$ and some $A\in\AA$, and the Conditions (US), (US$^*$), (B), (IS), (A$^*$), (R) and (L) are all met, then the statistics $T_n^{(KS)}$ and $ T_n^{(CM)}$ converge to $\infty$ in probability.
  \end{enumerate}
\end{corollary}

By Corollary \ref{cor:testconstancy} (ii) any test that rejects the null hypothesis if one of the test statistics exceeds a critical value will be consistent against alternatives that fulfill the conditions of Theorem \ref{theo:IZhat}, provided that the family $\AA$ is sufficiently rich to pick up the deviation from the null. To ensure consistency of such a test against general alternatives, $\AA$ must be measure determining, i.e.\ any two measures $\mu_1$ and $\mu_2$ on the unit sphere coincide if $\mu_1(A)=\mu_2(A)$ holds for all $A\in\AA$.

Since the limit distributions in Corollary \ref{cor:testconstancy} (i) may depend on the unknown spectral measure $S_1$,  in general it is not straightforward to determine a critical value $c_\alpha$ such that $\Indd{T_n>c_\alpha}$ is an asymptotic level $\alpha$ test. However, if the family $\AA$ is linearly ordered (i.e.\ for all $A,B\in\AA$ one has $A\subset B$ or $B\subset A$), then the problem can be reduced to analyzing the pertaining functionals of a Brownian pillow $W$, that is, a centered Gaussian process on $[0,1]^2$ with covariance function
$$ Cov\big(W(s_1,t_1),W(s_2,t_2)\big)=(s_1\wedge s_2-s_1s_2)\cdot(t_1\wedge t_2-t_1t_2). $$
To see this, note that then one has $S_1(A\cap B)=S_1(A)\wedge S_2(B)$ for all $A,B\in\AA$. Hence the process $\ZZ$ has the same distribution as $(W(t,S_1(A)))_{t\in[0,1], A\in\AA}$ and thus
$$
P\Big\{\sup_{t\in[0,1],A\in\AA}|\ZZ_t(A)|>c_\alpha\Big\}\le P\Big\{\sup_{s,t\in[0,1]}|W(s,t)|>c_\alpha\Big\}.
$$
Here even equality holds if $\{S_1(A)|A\in\AA\}=[0,1]$, which will typically be fulfilled if $S_1$ is continuous. The distributions of the supremum  and other functionals of a Brownian pillow have been examined by \cite{Hash10} and in Example A.2.12 of \cite{vdVW96}. In \cite{KonProt03} numerical approximations for critical values  of Kolmogorov-Smirnov type statistics (as for $T_n^{(KS)}$) and of combinations of Cram\'{e}r-von Mises type and Kolmogorov-Smirnov type statistics (as in $T_n^{(CM)}$) are given.

If $\AA$ is not linearly ordered, then in general $Cov(\ZZ_s(A),\ZZ_t(B))\le Cov(W(s,S_1(A)),$ $W(t,S_1(B)))$, and the Slepian inequality shows that the supremum of the Brownian pillow is stochastically dominated by the supremum of the process $\ZZ$. (Whether this also holds true for the supremum of the absolute value is not clear, though.) Hence, most likely the above approach does not work any more.
Unfortunately, the conditions that $\AA$ is linearly ordered (used to determine a critical value)  and that it is measure determining (to ensure consistency against general alternatives) cannot jointly be fulfilled unless $d=2$, that is, if bivariate random vectors are observed.

Using general bounds on exceedance probabilities of Gaussian processes, one may derive critical values that ensure that the probability of a type 1 error of a test based on $T_n^{(KS)}$ does not exceed a given size; see the Supplement \cite{Drees22} for details. However, in general these tests will be extremely conservative.
 As an alternative approach, we thus suggest to determine a critical value by simulations from a centered Gaussian process $\hat\ZZ$ with covariance function $Cov(\hat\ZZ(s,A),\hat\ZZ(t,B)) = (s\wedge t-st)(\widehat{IS}_1(A\cap B)-\widehat{IS}_1(A)\widehat{IS}_1(B))$. (Note that even under the null hypothesis of a stationary dependence structure, one cannot use the standard estimator for the spectral measure, because the marginal distributions may change nevertheless.)

Since $\widehat{IS}_1$ is a discrete measure with finite support (i.e.\ of the form $\sum_{l=1}^m p_l \eps_{\theta_l}$ with $\eps_\theta$ denoting the Dirac measure at $\theta$), the simulation of $\hat\ZZ$ is quite easy. For each point $\theta_l$ of the support, simulate an independent copy  $B_l$ of a Brownian bridge. Then $(\hat \ZZ_t(A))_{t\in[0,1],A\in\AA}$ has the same distribution as
\begin{equation} \label{eq:simlimitprocess}
  \bigg(\sum_{l:\theta_l\in A} \sqrt{p_l}B_l(t) - \sum_{l:\theta_l\in A}p_l \cdot\sum_{l=1}^m \sqrt{p_l}B_l(t)\bigg)_{t\in[0,1],A\in\AA},
\end{equation}
since both processes are centered Gaussian with the same covariance function. Under condition (A) or (A$^*$), for each $t\in[0,1]$ the processes attain at most $O(m^{\VC})$ different values, and so the supremum can be approximately calculated if the Brownian bridges are discretized in a suitable way. If  $ 2nh_n$ is a natural number, $n$ is a multiple of this number and $S_1$ is continuous, then one has $m=k_n/(2h_n)$ and $p_l=1/m$ almost surely, which further simplifies the numerical calculations.

If $\VC$ is large, then the above procedure is computationally very demanding. This will typically be a problem if the dimension $d$ of the observations is greater than 5, say, and the sample size $n$ is large, too. In that case, we propose to evaluate the process $\hat \ZZ_t$ only on a finely discretized subset of $\AA$. In any case, if very high dimensional data is observed then, as explained in Section 2, the rate of convergence implicitly given by the bias condition (B) will typically be slow, and one should not expect that a fully nonparametric approach gives reliable results. In such a situation, it seems advisable to apply some dimension reduction technique first; see \cite{EI21} for a recent survey of such methods.

\section{Simulations}  \label{sect:simus}

In this section, the finite sample performance of the tests proposed in Section \ref{sect:tests} is investigated in Monte Carlo simulations.

We consider $d$ dimensional observations for $d\in\{2,3\}$ with copula belonging to one of the following families:
\begin{itemize}
  \item Gumbel copula $C^G_\lambda(x)= \exp\big( -\big(\sum_{j=1}^d (-\log x_j)^\lambda\big)^{1/\lambda}\big)$, $x\in[0,1]^d$, for some $\lambda\ge 1$,
  \item $t$-copula, i.e., the copula of a multivariate $t$-distribution with density
   $$t_{\nu,\Sigma_\rho} (x)=\big(\det(\Sigma_\rho)(\nu\pi)^d\big)^{-1/2}\frac{\Gamma((\nu+d)/2)}{\Gamma(\nu/2)} \Big(1+\frac{x'\Sigma_\rho^{-1}x}{\nu}\Big)^{-(\nu+d)/2}
   $$
   with $\nu$ degrees of freedom, $(\Sigma_\rho)_{ii}=1$ and $(\Sigma_\rho)_{ij}=\rho$ for $i\ne j$ and some $\rho\in[0,1]$.
\end{itemize}
In most cases, we choose a Fr\'{e}chet distribution with cdf $F_\alpha(x)=\exp(-x^{-\alpha})$ as marginal distributions, but in some simulations we multiply the vector $X_t$ with a time-varying factor $c(t)=1+\sin(2\pi t)/2$, $t\in[0,1]$,
in order to check the (in)sensitivity of the tests against changes of the  marginal distributions.

All simulated samples have size $n=2000$. They are divided into blocks of length $b\in\{50,100,200\}$ (which corresponds to $h_n\in\{1/80, 1/40, 1/20\}$). In each block, the $k_n\in\{5,10,20\}$ vectors with largest Euclidean norm are used to estimate the local spectral measure.
The tests under consideration are based on the family of sets of the type $A_y:=\{x\in\sphere\mid x_i\le y_i, \forall 1\le i\le d-1\}$ for all $y\in [0,1]^{d-1}$ with $\|y\|\le 1$. 

For the bivariate observations, we use the critical values suggested by Corollary \ref{cor:testconstancy}, which can be obtained from simulations of Brownian pillows as described in Section \ref{sect:tests}. To this end, we have simulated 10\,000 Brownian pillows on the grid $\{0.001\cdot(i,j)\mid 0\le i,j\le 1000\}$. The resulting critical values for both tests and nominal size $5\%$ and $10\%$ are given in Table \ref{table:critval}. In what follows, always tests with nominal size $5\%$ are considered.

\begin{table}
 \caption{Asymptotic critical values of tests based on \eqref{eq:KStest} and \eqref{eq:CMtest} in dimension $d=2$ \label{table:critval}}
  \centerline{\begin{tabular}{l|cc}
    & \multicolumn{2}{c}{nominal size} \\
    & 0.05 & 0.10\\ \hline
    $T_n^{(KS)}$ & 0.8135 & 0.7626\\
    $T_n^{(CM)}$ & 0.1939 & 0.1621
  \end{tabular}}
\end{table}

In dimension $d=3$, we employ  the approach outlined at the end of the preceding chapter, which is based on the limit process under the null hypothesis with estimated spectral measure. To this end, first a finite subfamily of sets $A\in\AA$ is determined for which $\widehat{IS}_1$ attains all possible values. Then the processes $(\hat\ZZ_t(A))_{t\in[0,1]}$ were simulated $m=200$ times on the grid $2h_ni$, $i\in\{0,1,\ldots,1/(2h_n)\}$ using \eqref{eq:simlimitprocess}. While this grid is quite coarse for some of the bandwidths, this choice seems natural as it mimics the discretization used to calculate the test statistics. From these simulations, one can easily calculate an estimate of the limit distributions of the normalized test statistics, and hence an estimated $p$-value.
All reported values of the empirical power function are obtained from 1000 simulations for each model and each parameter setting and are rounded to two digits.

We first present the results for the bivariate models.
In Table \ref{table:empsizes2d}, the empirical probability of a type 1 error of both tests are given for different models, block lengths $b$ and numbers $k$ of order statistics. Note that in the second model the observations are not identically distributed, since the marginal scale parameters vary  over time by a factor of 3.
 (Further simulation results can be found in the Supplement \cite{Drees22}.) The upper value in each field corresponds to the test based on $T_n^{(KS)}$, while the lower gives the empirical probability of a type 1 error of the test pertaining to $T_n^{(CM)}$. In all models, the size exceeds the nominal level 0.05 by at most 0.01 for both tests. The Kolmogorov-Smirnov type test is often quite conservative, in particular for large block lengths. This may be explained by the fact that, in the definition of $T_n^{(KS)}$, the supremum over $t$ can only be attained at multiples of $2h_n=b/n$, while in the simulation of the limit distribution the maximum over a much finer grid is calculated, leading to larger critical values. In contrast, for the Cram\'{e}r-von Mises type statistic $T_n^{(CM)}$ a finer grid need not result in a larger critical value.

\begin{table}
  \caption{Empirical probability of a type 1 error of test based on \eqref{eq:KStest} (upper value) and \eqref{eq:CMtest} (lower value) for $d=2$ \label{table:empsizes2d}}
  \centerline{\small
  \begin{tabular}{r|rrr|rrr|rrr}
    $b$ & \multicolumn{3}{c|}{50} & \multicolumn{3}{c|}{100} & \multicolumn{3}{c}{200}\\
    $k$ & 5 & 10 & 20 & 5 & 10 & 20 & 5 & 10 & 20 \\ \hline
    Gumbel copula $\lambda=2$,  & 0.04 &	0.04& 0.04	&	0.02 &	0.02 &	0.03 &		0.01 &	0.02 &	0.02 \\
    Fr\'{e}chet $\alpha=2$ & 0.05 &	0.06&	0.06	&	0.04&	0.04	&0.05	&	0.03&	0.04&	0.05\\ \hline
    Gumbel copula $\lambda=2$, &   0.03&	0.03	&0.04	&	0.03&	0.02&	0.03&		0.02	&0.01 & 0.01\\
    Fr\'{e}chet $\alpha=4$, with sine-factor & 0.04&	0.04&	0.04	&	0.04&	0.05&	0.05&		0.03&	0.03&	0.04\\ \hline
    $t_2$-copula, $\rho=0$ & 0.02&	0.02&	0.03&		0.02&	0.02&	0.02&		0.02&	0.03&	0.02\\
    Fr\'{e}chet $\alpha=4$ & 0.04 &	0.04	&0.03	&	0.05&	0.04&	0.06	&	0.04	& 0.05&	0.04\\ \hline
    $t_1$-copula,$\rho=0$ &  0.04	&0.04	&0.03	&	0.03&	0.02&	0.02	&	0.02	& 0.02&	0.02\\
    Fr\'{e}chet $\alpha=4$ & 0.05 &	0.06	&0.06	& 0.05 &0.04	&0.04	&	0.03&	0.04&	0.04
   \end{tabular}}
\end{table}

To examine the power of the tests under the alternative hypothesis, we consider three different models with time varying spectral measures:
\begin{itemize}
  \item {\bf``G linear''}: The dependence of the observations is modeled by a Gumbel copula with parameter $\lambda$ increasing linearly over time from 2 to $\lambda_1\in\{2.5,3, 3.5,\ldots,6\}$.
  \item {\bf``t linear''}: The random vectors have a $t_2$-copula with matrix $\Sigma_\rho$ and $\rho$ linearly increasing from 0 to $\rho_1\in\{0.25, 0.5, 0.6, 0.7, 0.75, 0.8, 0.9, 0.95\}$.
   \item {\bf``t jump''}: The random vectors have a $t_2$-copula with matrix $\Sigma_\rho$ and $\rho$ being equal to 0 for $t\le 1/2$ and equal to $\rho_1\in\{0.2, 0.3, 0.4, 0.5, 0.6, 0.7, 0.75, 0.8, 0.9\}$ for $t>0.5$.
\end{itemize}
While in the first two models the dependence structure changes smoothly over time, in the third model there is a structural break at $t=1/2$. In all models, the marginal distributions are Fr\'{e}chet with $\alpha=4$. The corresponding power functions are shown in Figures \ref{fig:Gumbellin2d}--\ref{fig:t2jump2d}.

In all settings, the test based on the statistic $T_n^{(CM)}$ performs substantially better than the test using $T_n^{(KS)}$. The superiority  of the Cram\'{e}r-von Mises type test is particularly pronounced in the first two models when the dependence structure changes gradually. Moreover, it is clearly advisable to use rather a short block length, because the power is very low for block length $b=200$. This is not surprising since it is difficult to detect any change in the dependence structure if there are only a few blocks available. Conversely, the test based on $T_n^{(KS)}$ performs better for distributions in the alternative hypothesis when the dependence structure is similar in the beginning and in the end of the time interval, but differs in the middle; see the Supplement \cite{Drees22} for details.

Finally, in all simulations the tests have larger power if one uses a larger number of observations in each block. However, one has to ensure that these observations actually reflect the extreme value dependence structure. Otherwise, the tests may detect changes in the dependence structure which are not present in the extreme regions  one is interested in.

\begin{figure}[tb]
\includegraphics[width=13cm,height=4cm]{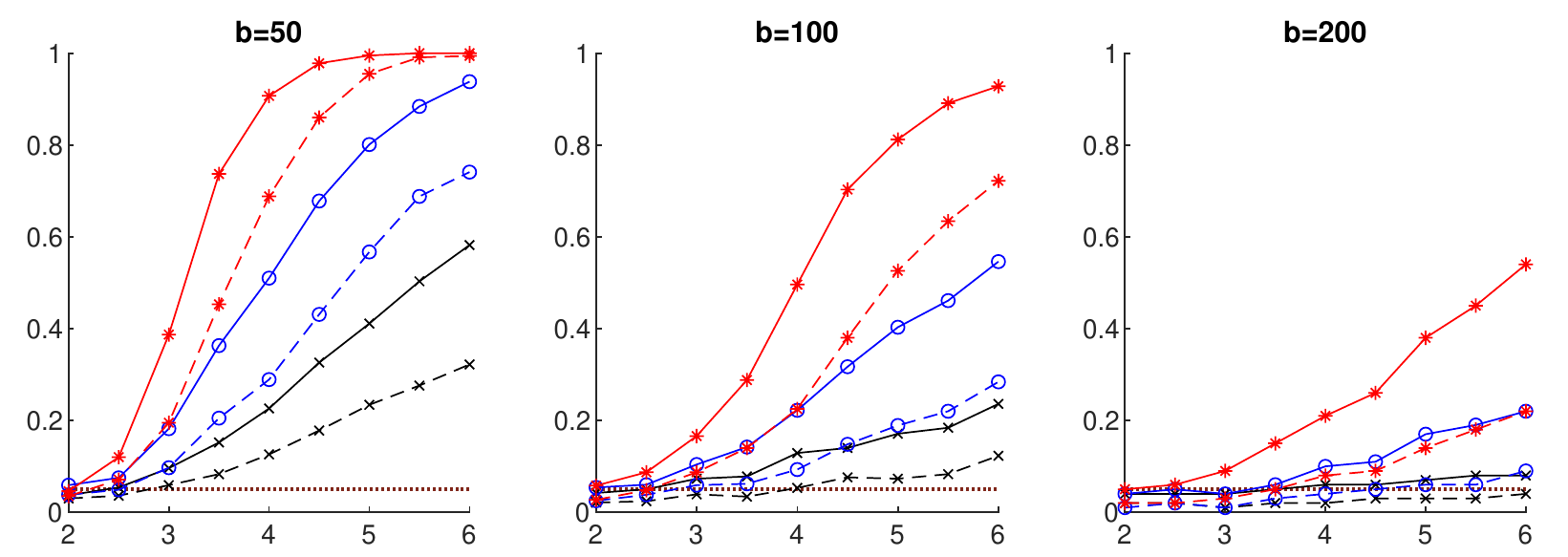}
\caption{Empirical power of tests based on $T_n^{(CM)}$ (solid line) and $T_n^{(KS)}$ (dashed line) in model ``G linear''  for $k=5$ (black $\times$), $k=10$ (blue $\circ$) and $k=20$ (red $*$) largest observations in each block and different block lengths $b$; the nominal size is indicated by the brown dotted horizontal line.
\label{fig:Gumbellin2d}}
\end{figure}

\begin{figure}[tb]
\includegraphics[width=13cm,height=4cm]{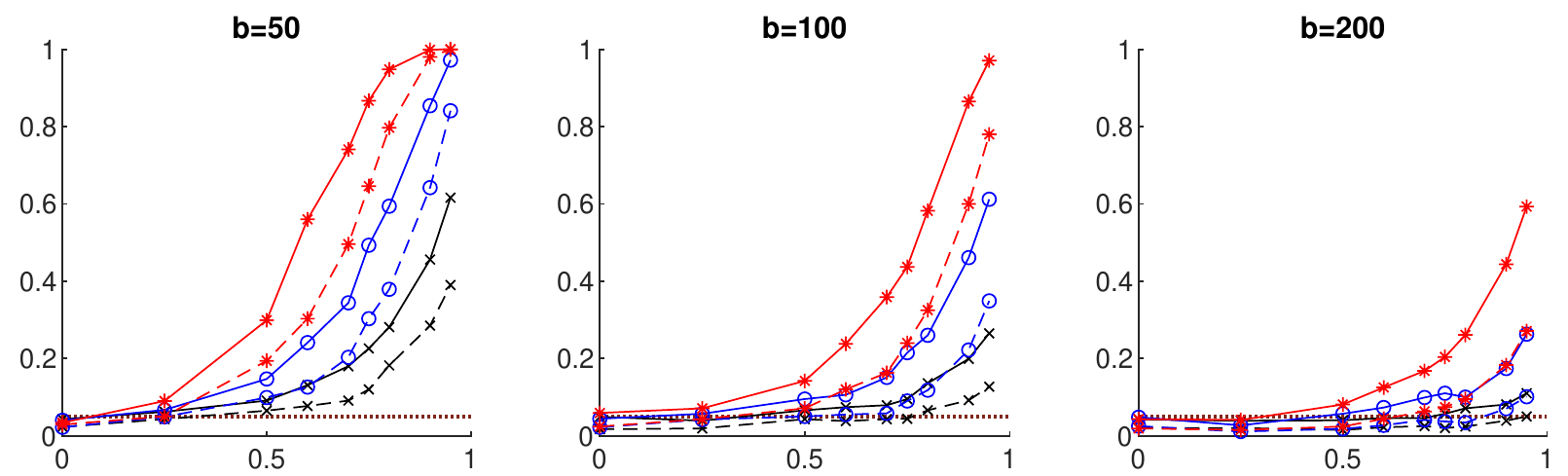}
\caption{Empirical power of tests based on $T_n^{(CM)}$ (solid line) and $T_n^{(KS)}$ (dashed line) in model ``t linear'' 
\label{fig:t2lin2d}}
\end{figure}

\begin{figure}[tb]
\includegraphics[width=13cm,height=4cm]{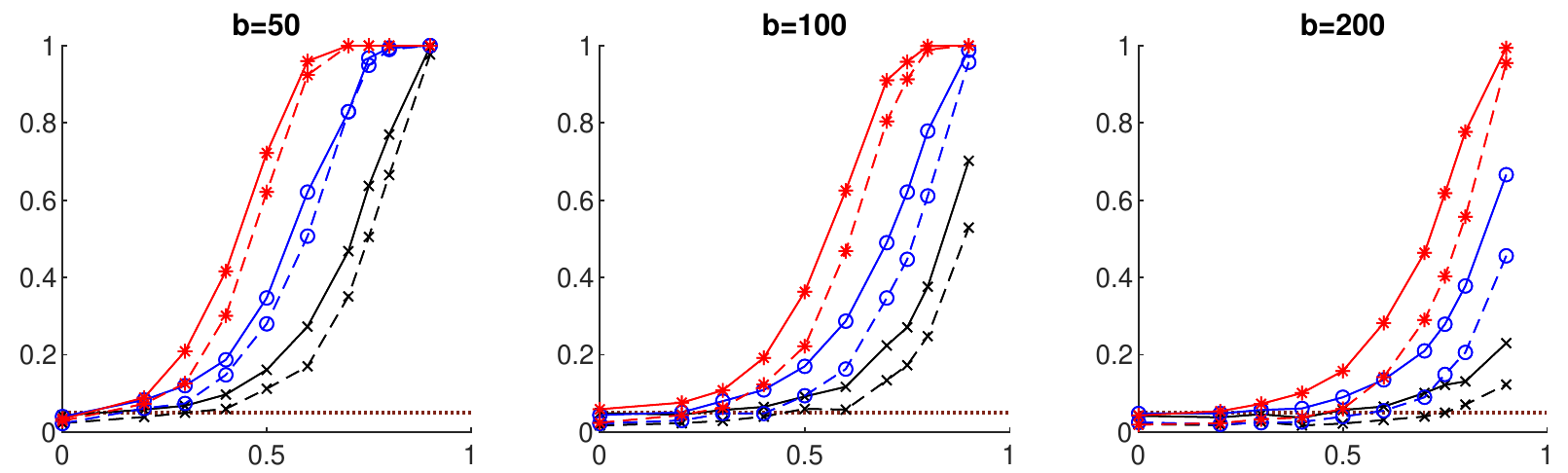}
\caption{Empirical power of tests based on $T_n^{(CM)}$ (solid line) and $T_n^{(KS)}$ (dashed line) in model ``t jump''
\label{fig:t2jump2d}}
\end{figure}

We now discuss our simulation results in dimension $d=3$. Since then the critical value is determined by simulation for each sample separately, simulating a large number of samples becomes computationally demanding. We thus examine fewer parameter constellations and focus on block lengths $b\in\{50,100\}$. Here we also consider a model with different marginal distributions (denoted by ``different Fr\'{e}chet'' in the tables): while first all marginal cdfs are again chosen as $F_4$, we then add 1 to the $i$th coordinate and multiply it with $i$ for $i\in\{2,3\}$. Table \ref{table:empsizes3d} gives the empirical probability of a type 1 error of our tests for different models with Gumbel- or $t_2$-copula. Again, the empirical sizes exceed the nominal value 0.05 only slightly. Since we have adapted the discretization of the estimated limit process to the block length used by the test as described above, the Kolmogorov-Smirnov type test is not as conservative as in the bivariate models.

\begin{table}[bt]
    \caption{Empirical probability of  type 1 error of test based on $T_n^{(KS)}$ (upper value) and $T_n^{(CM)}$ (lower value) for $d=3$ \label{table:empsizes3d}}
  \centerline{\small
  \begin{tabular}{r|rrr|rrr}
    $b$ & \multicolumn{3}{c|}{50} & \multicolumn{3}{c}{100}\\
    $k$ & 5 & 10 & 20 & 5 & 10 & 20  \\ \hline
    Gumbel copula  $\lambda=2$,  & 0.07	&0.06	&0.07	&	0.06	&0.06	&0.06
 \\
    different Fr\'{e}chet & 0.07&	0.06&	0.06	&	0.06&	0.07&	0.05
\\ \hline
    Gumbel copula $\lambda=2$, &   0.06	&0.07	&0.05	&	0.08	&0.06	&0.06\\
    Fr\'{e}chet $\alpha=4$, with sine-factor & 0.05	&0.07&	0.06	&	0.06	&0.06	& 0.05\\ \hline
    $t_2$-copula,  & 0.07&	0.06&	0.06	&	0.07&	0.07	&0.05
\\
    Fr\'{e}chet $\alpha=4$ & 0.07&	0.05&	0.06	&	0.06&	0.07&	0.04
\\ \hline
    $t_2$-copula, &  0.05	&0.05&	0.05	&	0.08&	0.07&	0.07
\\
    Fr\'{e}chet $\alpha=4$, with sine-factor & 0.05	&0.05	&0.05	&	0.05	&0.06	&0.06
   \end{tabular}}
\end{table}

Table \ref{table:emppower3d} summarizes the power of the tests for our models with changing extreme value dependence structure and Fr\'{e}chet marginals with $\alpha=4$.
\begin{table}[bt]
  \caption{Empirical power of tests based on $T_n^{(KS)}$ (upper value) and $T_n^{(CM)}$ (lower value), $d=3$.\label{table:emppower3d}}
  \centerline{\small
  \begin{tabular}{r|rrr|rrr}
    $b$ & \multicolumn{3}{c|}{50} & \multicolumn{3}{c}{100}\\
    $k$ & 5 & 10 & 20 & 5 & 10 & 20  \\ \hline
    model ``Gumbel linear'', $\lambda_1=3$ & 0.11 &	0.14&	0.25	&	0.08&	0.09	& 0.14\\
    different Fr\'{e}chet \hspace*{1.2cm}& 0.12	&0.18	&0.3	&	0.08	&0.09	&0.15\\[0.5ex]
    $\lambda_1=4$ & 0.18	&0.34	&0.7	&	0.10&	0.15&	0.26\\
    & 0.20	&0.41&	0.83	&	0.10&	0.19	&0.44 \\[0.5ex]
    $\lambda_1=5$ & 0.26	&0.53&	0.94	&	0.15&	0.29	&0.60\\
    & 0.34&	0.67&	0.98&		0.18&	0.35&	0.7 \\ \hline
    model ``t linear'', $\rho_1=0.05$ & 0.14&	0.26&	0.46	&	0.10&	0.16	&0.28\\
    & 0.16	&0.29&	0.54	&	0.11&	0.18&	0.31\\[0.5ex]
    $\rho_1=0.75$ & 0.41&	0.73&	0.97&		0.20&	0.39&	0.74\\
    & 0.48	&0.81	&0.99	&	0.21	&0.46&	0.83\\[0.5ex]
    $\rho_1=0.9$ & 0.85	&0.98&	1.00	&	0.59&	0.83&	0.98\\
    & 0.89&	0.99	&1.00	&	0.63	&0.87&	1.00\\ \hline
    model ``t jump'', $\rho_1 = 0.5$ & 0.40	&0.72	&0.98	&	0.22&	0.38&	0.73\\
        & 0.38	&0.68&	0.97	&	0.18&	0.36&	0.69\\[0.5ex]
        $\rho_1=0.75$ & 0.95&	1.00	&1.00	&	0.63&	0.95&	1.00\\
        & 0.95	&1.00&	1.00	&	0.60&	0.94	& 1.00
  \end{tabular}}
\end{table}
By and large, the findings are the same as for the bivariate models. However, now the Kolmogorov-Smirnov type test performs almost as well as the Cram\'{e}r-von Mises type test.

\section{Application}

We consider a data set of 6870 fire insurance claims of Danish insurance companies from 01/1980 to 12/2002 in Danish crowns (DKK). Each claim size is divided into a loss to buildings, a loss to content and a loss to profit, whereby one or two components can be equal to 0 and only claims with total size of at least one million DKK have been recorded. More details about this data set can be found in \cite{DM08}.

Since for most claims the loss of profit equals 0, here we only analyze the dependence between the other two components. To this end, we consider the same family of test sets as in Section \ref{sect:simus}, that is, we compare the functions $\widehat{IS}_{n,t}\big(\{(x_1,x_2)\in\mathcal{S}^1|x_1\le x\}$, $x\in[0,1]$, which can be interpreted as the cdf of the measure obtained by projecting the sphere onto the unit interval (cf.\ the discussion of condition (A)).

We estimate the integrated spectral measure for different block lengths (for simplicity discarding any remaining claims) and different numbers of order statistics $k_n$ per block. It turns out that $\widehat{IS}_{n,1}$ is less sensitive to the choice of $k_n$ if one uses larger blocks, e.g.\ $b=300$. The estimates are then quite stable up to $k_n=40$. The corresponding values of the
test statistics $T_n^{(KS)}$ and $T_n^{(CM)}$ are 0.415 and 0.043, so that the null hypothesis of a constant spectral measure cannot be rejected at 10\% level.

However, the sampling scheme introduces an artificial negative dependence between both components, because a small loss to a building is only registered if it is accompanied by a large loss to the content, and vice versa. This artificial dependence may mask a change in the ``true'' dependence between the claim components over time. To overcome this potential artefact,  in addition we analyze the subsample of those claims where both components exceed the threshold of one million DKK, resulting in 779 observations.

The methodology sketched above suggests a block length of $b=50$ and to use $k_n=10$ order statistics in each block. The resulting values of the test statics are now $T_n^{(KS)}=0.953$ and $T_n^{(CM)}=0.384$, i.e.\ both tests reject the null hypothesis at level $1\%$.

In order to confirm this finding and to understand better in which way the dependence structure changes over time we plot the ``cdf'' of the estimated integrated spectral measure based on the subsamples of the first 400 and of the remaining observations (see Figure \ref{fig:cdfcomp}). Whereas both cdf's are quite similar on $[0,0.3]$ and $[0.6,1]$, the cdf of the first subsample is larger in the middle part, indicating that in the second part of the observations more mass is concentrated near the main diagonal, i.e., the losses to buildings and to content are more often of a similar size than in the first part.

\begin{figure}[tb]
\includegraphics[width=10cm,height=5cm]{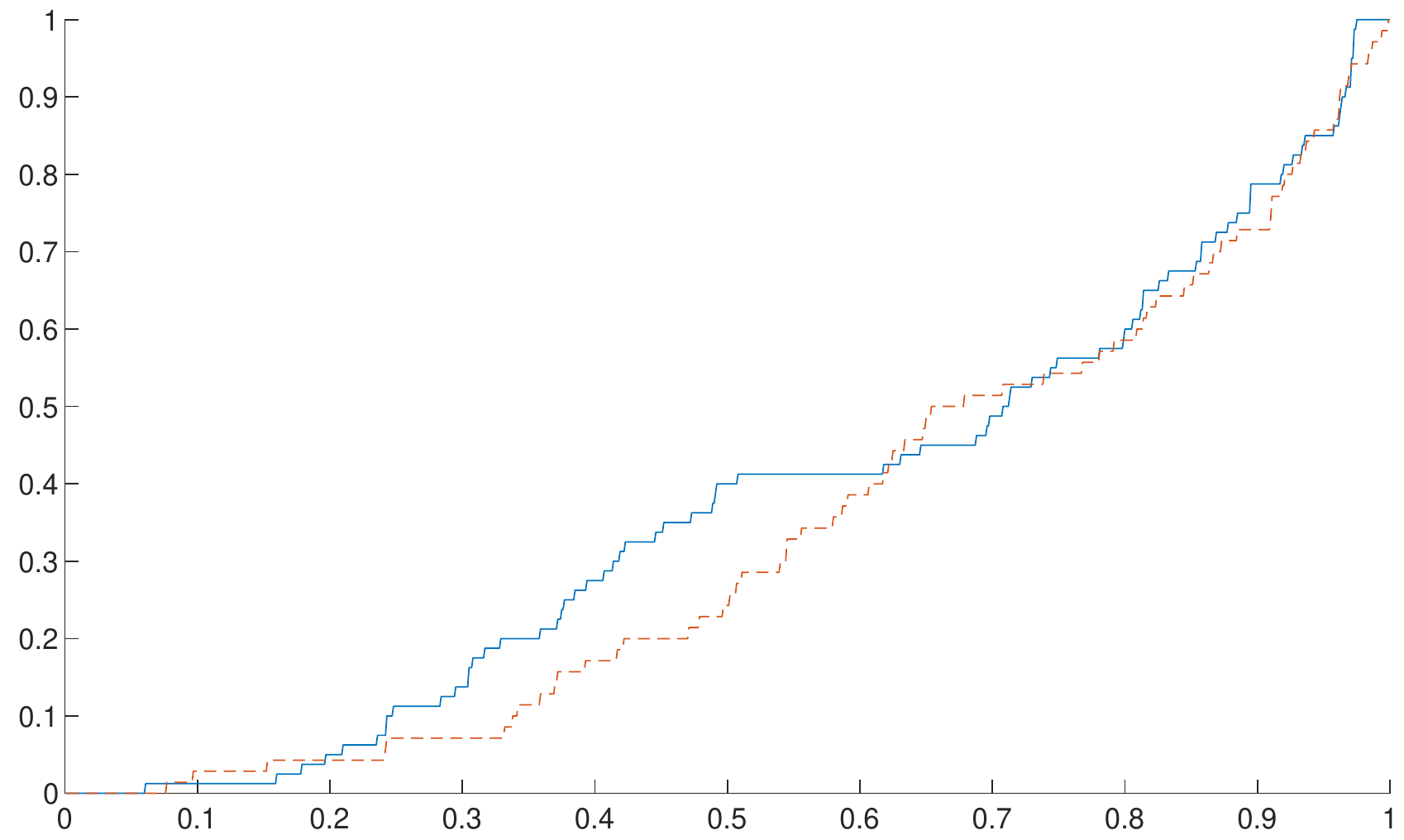}
\caption{Estimated ``cdf'' of the integrated spectral measure for the first 400 (blue, solid line) and the remaining observations (red, dashed line)
\label{fig:cdfcomp}}
\end{figure}

\section{Modifications and Outlook}  \label{sect:outlook}

\subsection*{More general sampling schemes}

So far we have assumed that the observations are equidistant in time. This, however, does not play any role in the analysis of the local estimator $\hat S_{n,t}$ for fixed $t$, and it is also not essential for the estimator of the integrated spectral measure, provided the same number $k_n$ of extreme observations is used in all sub-intervals $(2(j-1)h_n,2jh_n]$, $1\le j\le \floor{1/(2h_n)}$. Indeed, if the random variables $X_{t_{n,i}}$, $1\le i\le n$, are observed and one redefines $I_{n,t}:=\big\{i\in \{1,\ldots,n\}\mid t_{n,i}\in (t-h_n,t+h_n]\big\}$ and modifies the definitions of $\hat u_{n,t}$, $\nu_{n,t}$, $S_{n,t}$, $\tilde S_{n,r}$, $N_r$, $w_{n,r}$, $u_{n,j}^{(i)}$ and $u_{n,j}^{(i,k)}$ accordingly, then basically the same proofs still yield the main result Theorem \ref{theo:IZhat}. If the number of extremes used for estimation varies smoothly over time, but it is always of the same order as $k_n$, i.e., for some smooth function $c$, $k_n(t)\sim c(t) k_n$ largest observations are used to estimate $S_t$, then one may still prove an analog to Theorem \ref{theo:IZhat}, but the limit process depends on $c$. This dependence makes it more cumbersome to determine the critical values of the test procedures discussed in Section \ref{sect:tests}, but they can still be applied.
Note that the assumption that a similar number of extreme observations is used for all sub-intervals is quite mild if the number of observations does not vary strongly. In particular, we have to assume that the number of observations in each subinterval is of larger order than $k_n$ for condition (B) to make sense.

\subsection*{Degenerate spectral measures}

If one component of the random vector $X_t$ is asymptotically independent from the others (in the extreme value sense), then the spectral measure is concentrated on the union of the corresponding axis and its orthogonal complement. (More generally, the spectral measure may be supported by a finite union
 of lower dimensional spaces.)
In principle, our theory also applies in this case if the family $\AA$ of test sets is chosen carefully. As already known from the classical Donsker theorem, one cannot expect that uniform convergence of an empirical process holds in a neighborhood of a jump point of the true cdf. Similarly, in our setting the boundary of our test sets ought to be bounded away from any lower dimensional subspaces that have positive mass under the spectral measures.

For example, assume that for $d=3$ one wants to test for changes of the dependence structure in a situation when $S_t$ is concentrated on $M:=\{(1,0,0)\}\cup(\{(0,y,z)|y,z\ge 0\}\cap\mathcal{S}^2)$ for at least some time points $t$. We propose to first use our tests with a family of sets whose boundary if bounded away from $M$, like $\{[\eps\mathbf{1},\bfx]\cap\mathcal{S}^2 \mid \bfx\in[\eps,\infty)^3\}$ for some small $\eps>0$, in order to check whether the spectral measure is concentrated on $M$ for {\em all} $t\in[0,1]$. If this is the case, then in the next step one may test the constancy of the spectral measures of just the last two components of the vectors (e.g.\ using the family of test sets employed in the simulations for dimension 2). If these seem to be constant, too, then finally one may easily check whether all spectral measures put the same mass on the $x$-axis. If none of these three tests indicate a change, then one may work with the assumption that the dependence structure does not change over time.

\subsection*{Marginal standardization}

While we allow for a very general temporal development of the marginal distributions, at any fixed time we basically require a similar tail behavior for all marginal distributions. More precisely, the spectral measure $S_t$  only captures the extreme value dependence between the components with the heaviest tails, whereas all lighter tailed components vanish asymptotically. Such an approach is natural in financial risk management when one is interested in the overall risk of some investment portfolio; see Remark 1.3 of \cite{YSC20}.

In contrast, if one is interested in the dependence structure in the sense of classical extreme value theory, then the marginal distributions must first be standardized; see Part II of \cite{dehaan2006}. If the marginal distributions are unknown, they must be estimated either semi-parametrically using generalized Pareto approximations or non-parametrically by the empirical cdf.

In the former approach, the resulting limit distribution of $\hat S_t$ and $\widehat{IS}_t$ will depend on the marginal distribution if the same number of largest observations is used for marginal tail estimation and for estimation of the dependence; see \cite{EdHS97}, Theorem 3.4. Consequently testing for a constant extreme value dependence while still allowing for quite general smoothly varying marginal distributions will become infeasible. To avoid this problem, one has to impose much more severe restrictions on the marginal behavior over time, e.g., by assuming some heteroscedastic model as in \cite{EdHZ16}. Since such an assumption does not fit well to our completely nonparametric model of the dependence structure, we do not follow this path.

The nonparametric approach to marginal standardization does not lead to the problem sketched above, but even in a setting with identically distributed observations few results are available about the asymptotic behavior of the estimator of the spectral measure after rank standardization. Most publications about the statistical analysis of the extreme value dependence consider the tail copula or the stable tail dependence function, because these functions can be expressed in terms of probabilities of rectangles or the complement of rectangles and the shape of these rectangles is not altered by the marginal standardization. However, the estimation of the spectral measure requires to analyze the probability that a marginally standardized observation falls into a set of the form $\{rx| r>1,x\in A\}$ for some $A\subset\sphere$, which is a much more challenging task.
\cite{EdHP01} established the limit distribution in dimension $d=2$ for sets $A$ of the form $A=\{(x_1,x_2)\in\mathcal{S}^1| x_1\le x\}$ under quite involved conditions which are difficult to verify. To the best of our knowledge, no such results are known for substantially different families of sets or for dimension $d>2$. For that reason, it will be difficult to prove a counterpart to our main results after a rank based marginal standardization.

\subsection*{Change point detection}

Assume that the spectral measure is constant up to some change point $t_0$, it then changes and remains constant after this change point. If Theorem \ref{theo:IZhat} applies to both subsamples $X_{i/n}$, $i/n \in(0,t_0]$ and $X_{i/n}$, $i/n \in(t_0,1]$, then one may conclude the asymptotic behavior of
 the process
\begin{align*}
   \sup_{A\in\AA} \big|\widehat{IS}_{n,t}(A)-t \widehat{IS}_{n,1}(A)\big| & = t(1-t) \sup_{A\in\AA} \big| \widehat{MS}_{n,(0,t]}(A)-\widehat{MS}_{n,(t,1]}(A)\big|,
\end{align*}
$   t\in [0,1]$, with
$$ \widehat{MS}_{n,(a,b]}(A):= \frac 1{b-a} \int_a^b \hat S_{n,(2\ceil{r/(2h_n)}-1)h_n}(A)\, dr
$$
being an estimator of the average spectral measure over the time interval $(a,b]$.
The point of maximum of this process may be considered a (CUSUM type) estimator of $t_0$. Since by definition the process is small for $t$ close to the boundary of the unit interval, this change point estimator will be biased in that it tends to yield some value near the center of the interval. To overcome this disadvantage, it seems natural to introduce a weight factor that is a decreasing function of $t(1-t)$. However, then a refined version of our main result would be needed to derive the asymptotic behavior of the new estimator of the change point, which is beyond the scope of our investigations.

\subsection*{Serial dependence}

Throughout we assumed independence of the observations, because otherwise the limit distribution of the integrated spectral measure will depend on the form of serial dependence, rendering the tests for constant spectral measure infeasible. In practical applications, e.g.\ based on environmental data, though, one will often encounter some serial dependence which vanishes for observations that are sufficiently separated in time. In such a case, it may still be justified to use our results (at least from a practical point of view) if the time periods over which the serial dependence seem relevant are much shorter than the time scale over which a change of the spectral measure may happen. For example, while weather patterns rarely last longer than a couple of weeks, decades seems a more appropriate scale to analyze changes in the climate. Hence an analysis using the methods outlined here may be feasible if one uses observations that are sufficiently separated in time (like monthly maxima).

\section{Proofs}
  \label{sect:proofs}

The proof of Theorem  \ref{theo:fixedt} is given in the Supplement \cite{Drees22}.

\subsection{Proofs to Subsection \ref{sect:integratedSM}}

The proof of Proposition \ref{prop:IZtilde} consists of two main steps. In Proposition \ref{cor:intprocapprox} we show that $\widetilde{IZ}_{n,\cdot}(\cdot)$ can be approximated by a sum of $n$ {\em independent} random variables. This enables us to apply standard techniques from \cite{vdVW96} to prove asymptotic tightness of the process and, by applying the CLT by Lindberg and Feller, the asserted convergence.

The most  intricate part in the proof of our main result Theorem \ref{theo:IZhat} is to show that the difference between the processes $\widetilde{IZ}_n$ and $\widehat{IZ}_n$ using deterministic respectively random thresholds is asymptotically negligible. This difference can be expressed as a sum of independent terms $\Delta_{n,j}(A)$, the first two moments of which are uniformly bounded in the Lemmas \ref{lemma:firstmoment} and \ref{lemma:secondmoment}. We then employ an idea from a proof given in \cite{vdVW96} to show in Proposition \ref{prop:sumdeltasqconv} that the sum of the squared terms $\Delta_{n,j}^2(A)$ are uniformly negligible, from which the convergence of $\widehat{IZ}_n-\widetilde{IZ}_n$ can be concluded.

Over the course of these proofs, the following uniform bound of Bernstein type on the numerator of one summand of $\widetilde{IZ}_{n,\cdot}(\cdot)$, which is proved in the Supplement \cite{Drees22}, is used several times. Let
\begin{equation*} 
  w_{n,r}(A) := \sum_{i\in I_{n,r}} P\{R_{i/n}>u_{n,r},\Theta_{i/n}\in A\}, \quad r\in[0,1], A\in\AA^*.
\end{equation*}
In particular, eventually $w_{n,r}(\sphere)=\nu_{n,r}(u_{n,r},\infty)=k_n$ for all $r\in[0,1]$.
\begin{lemma} \label{lemma:bernstein}
  Fix some $d>0$. If (A$^*$) holds and $\log h_n =o(k_n)$, then there exists a constant $K$, depending only on $\VC$, such that eventually for all $r\in[0,1]$
  \begin{align}
    P\Big\{\sup_{A\in\AA}\Big|\sum_{i\in I_{n,r}} \Indd{R_{i/n}>u_{n,r},\Theta_{i/n}\in A} -w_{n,r}(A)\Big|\ge d\big(k_n\log(k_n/h_n)\big)^{1/2}\Big\} 
     & \le K\Big(\frac{h_n}{k_n}\Big)^{d^2/50-\VC}.\label{eq:bernstein}
  \end{align}
\end{lemma}

As a consequence we can derive an approximation of the standardized estimator of the integrated spectral measure by a structurally simpler expression. To this end, let
\begin{align*}
  Z_{n,t}(A) & :=  k_n^{-1/2}\sum_{i\in I_{n,t}} \Big( \Indd{R_{i/n}>u_{n,t},\Theta_{i/n}\in A}- P\{R_{i/n}>u_{n,t},\Theta_{i/n}\in A\}\Big),\\
  S_{n,t}(A) &:= \frac{\sum_{i\in I_{n,t}}P\{R_{i/n}>u_{n,t},\Theta_{i/n}\in A\}}{\sum_{i\in I_{n,t}} P\{R_{i/n}>u_{n,t}\}}.
\end{align*}
Note that  $S_{n,t}(A)=w_{n,t}(A)/k_n$ for sufficiently large $n$. Moreover,   for $ t\in[0,1], A\in\AA$, let
\begin{align*}
 Y_{n,i}(t,A) & := (k_nh_n)^{-1/2} \big[ D_{n,i}(A) - S_{n,r(i)}(A) D_{n,i}\big]\times \int_{r(i)-h_n}^{r(i)+h_n} \indd{[0,t]}(r)\, dr
\end{align*}
with
\begin{align*}
  r(i) & := \bigg( 2\ceill{\frac{i}{2nh_n}}-1\bigg) h_n,\\
  D_{n,i}(A) & := \Indd{R_{i/n}>u_{n,r(i)},\Theta_{i/n}\in A}-P\{R_{i/n}>u_{n,r(i)},\Theta_{i/n}\in A\},\\
  D_{n,i} & := D_{n,i}(\sphere) = \Indd{R_{i/n}>u_{n,r(i)}}-P\{R_{i/n}>u_{n,r(i)}\}.
\end{align*}
\begin{proposition} \label{cor:intprocapprox}
    If the Conditions  (US), (B), (IS) and (R) are met, then the process $\widetilde{IZ}_{n,\cdot}$ defined in Proposition \ref{prop:IZtilde} fulfills
\begin{align*}
    \sup_{t\in[0,1],A\in\AA}\bigg|\widetilde{IZ}_{n,t}(A)-\sum_{i=1}^n Y_{n,i}(t,A)\Big| =o_P(1).
\end{align*}
\end{proposition}
The proof can be found in the Supplement \cite{Drees22}.
\smallskip

\begin{proofof} Proposition \ref{prop:IZtilde}. \rm
   As usual, we first prove convergence of the fidis.
   By Proposition \ref{cor:intprocapprox}, the Cram\'{e}r-Wold device and the CLT of Lindeberg-Feller, we have to show that the covariances of $Y_{n,i}$  converge and the Lindeberg condition is fulfilled. Using  Conditions (US) and (B), one can easily show that $S_{n,r(i)}(A)= S_{r(i)}(A)+o\big((h_n/k_n)^{1/2}\big)$ (cf.\ (8.8) in \cite{Drees22}).
   From $\Cov(D_{n,i}(A),D_{n,i}(B))=P\{R_{i/n}>u_{n,r(i)}, \Theta_{i/n}\in A\cap B\}+o(P\{R_{i/n}>u_{n,r(i)}\})$ and $r(i)=r_j$ for $i\in I_{n,r_j}$, it follows for $0\le s\le t\le 1$ and $A,B\in\AA$
   \begin{align*}
     & \sum_{i=1}^n  \Cov\big(Y_{n,i}(s,A),Y_{n,i}(t,B)\big)\\
     & = \frac 1{k_nh_n} \sum_{i=1}^n \Big[ \Cov\big(D_{n,i}(A),D_{n,i}(B)\big)-S_{n,r(i)}(A) \Cov\big(D_{n,i}(B),D_{n,i}\big)\\
     & \hspace{2cm} -S_{n,r(i)}(B) \Cov\big(D_{n,i}(A),D_{n,i}\big) +S_{n,r(i)}(A)S_{n,r(i)}(B) \Var(D_{n,i})\Big]\\
     & \hspace{2cm} \times \int_{r(i)-h_n}^{r(i)+h_n} \indd{[0,s]}(r)\,dr \times \int_{r(i)-h_n}^{r(i)+h_n} \indd{[0,t]}(r)\,dr\\
      & = \frac 1{k_nh_n} \sum_{i=1}^n \Big[ P\{R_{i/n}>u_{n,r(i)}, \Theta_{i/n}\in A\cap B\} -S_{n,r(i)}(A) P\{R_{i/n}>u_{n,r(i)}, \Theta_{i/n}\in B\}\\
     & \hspace{1cm} -S_{n,r(i)}(B) P\{R_{i/n}>u_{n,r(i)}, \Theta_{i/n}\in A\} +S_{n,r(i)}(A)S_{n,r(i)}(B) P\{R_{i/n}>u_{n,r(i)}\}\\
     & \hspace{1cm} +o(P\{R_{i/n}>u_{n,r(i)}\})\Big]\times \int_{r(i)-h_n}^{r(i)+h_n} \indd{[0,s]}(r)\,dr \times \int_{r(i)-h_n}^{r(i)+h_n} \indd{[0,t]}(r)\,dr\\
      & = \frac 1{h_n} \sum_{j=1}^{J_n+1}\big[S_{n,r_j}(A\cap B)-S_{n,r_j}(A) S_{n,r_j}(B)+o(1)\big]\\
       & \hspace*{4cm} \times \int_{r_j-h_n}^{r_j+h_n} \indd{[0,s]}(r)\,dr \times \int_{r_j-h_n}^{r_j+h_n} \indd{[0,t]}(r)\,dr\\
      & = 2 \sum_{j=1}^{J_n+1}\big[S_{r_j}(A\cap B)-S_{r_j}(A) S_{r_j}(B)\big] \int_{r_j-h_n}^{r_j+h_n} \indd{[0,s]}(r)\,dr +o(1)\\
      & \to 2\int_0^s S_r(A\cap B)-S_r(A)S_r(B)\, dr.
     \end{align*}
     In the last step we have used Condition (IS). The Lindeberg condition is trivial, because $Y_{n,i}(s,A)$ is bounded by $2(h_n/k_n)^{1/2}\to 0$.

   It remains to show that $\sum_{i=1}^n Y_{n,i}$ is asymptotically equicontinuous. To this end, we apply Theorem 2.11.1 of \cite{vdVW96} to the uncentered processes
   \begin{align*}
   &Y_{n,i}^*(t,A) \\
   &= (k_nh_n)^{-1/2} \big(\Indd{R_{i/n}>u_{n,r(i)}, \Theta_{i/n}\in A}-S_{n,r(i)}(A)\Indd{R_{i/n}>u_{n,r(i)}}\big) \int_{r(i)-h_n}^{r(i)+h_n} \indd{[0,t]}(r)\,dr.
   \end{align*}
   As semi-metric on $\FF:=[0,1]\times\AA$ we choose $\rho\big((s,A),(t,B)\big):=|t-s|+\rho_I(A,B)$.
   In view of Condition (A$^*$), $\FF$ is obviously totally bounded w.r.t.\ $\rho$ and the measurability condition and the Lindeberg type condition are fulfilled, too.

   The second displayed condition of Theorem 2.11.1 is fulfilled if the following two conditions are met:
   \begin{align}
     \lim_{\delta\downarrow 0} \limsup_{n\to \infty} \sup_{|s-t|\le \delta,A\in\AA} \sum_{i=1}^n E\big(Y_{n,i}^*(s,A)-Y_{n,i}^*(t,A)\big)^2 & = 0, \label{eq:cond1}\\
     \lim_{\delta\downarrow 0} \limsup_{n\to \infty} \sup_{\rho_I(A,B)\le\delta, t\in[0,1]} \sum_{i=1}^n E\big(Y_{n,i}^*(t,A)-Y_{n,i}^*(t,B)\big)^2 & = 0  \label{eq:cond2}.
   \end{align}
   W.l.o.g.\ assume $s\le t$. Condition \eqref{eq:cond1} follows from
   \begin{align*}
     & \sum_{i=1}^n  E\big(Y_{n,i}^*(s,A)-Y_{n,i}^*(t,A)\big)^2\\
     & = \frac 1{k_nh_n} \sum_{i=1}^n E\big(\Indd{R_{i/n}>u_{n,r(i)}, \Theta_{i/n}\in A}-S_{n,r(i)}(A)\Indd{R_{i/n}>u_{n,r(i)}}\big)^2 \Big(\int_{r(i)-h_n}^{r(i)+h_n} \indd{(s,t]}(r)\, dr\Big)^2\\
     & \le \frac 2{k_n} \sum_{j=1}^{J_n+1} \sum_{i\in I_{n,r_j}} P\{R_{i/n}>u_{n,r_j}\}\int_{2(j-1)h_n}^{2jh_n} \indd{(s,t]}(r)\, dr\\
     & = 2 \sum_{j=1}^{J_n+1} \int_{2(j-1)h_n}^{2jh_n} \indd{(s,t]}(r)\, dr
      = 2|t-s|.
   \end{align*}

   To verify \eqref{eq:cond2}, we again use $S_{n,r}(A)=S_r(A)+o(1)$ (which follows from (US), (B) and (R)) and (IS) to obtain, uniformly for all $A,B\in\AA$ and $t\in[0,1]$,
   \begin{align*}
     \sum_{i=1}^n & E\big(Y_{n,i}^*(t,A)-Y_{n,i}^*(t,B)\big)^2\\
     & \le \frac{4h_n}{k_n}  \sum_{i=1}^n E\Big[\Indd{R_{i/n}>u_{n,r(i)}, \Theta_{i/n}\in A} - \Indd{R_{i/n}>u_{n,r(i)}, \Theta_{i/n}\in B}\\
     & \hspace*{3cm} -(S_{n,r(i)}(A)-S_{n,r(i)} (B))\Indd{R_{i/n}>u_{n,r(i)}}\Big]^2\\
     & \le  \frac{8h_n}{k_n} \bigg[ \sum_{i=1}^n  P\{R_{i/n}>u_{n,r(i)},\Theta_{i/n}\in A\triangle B\}\\
      & \hspace{4cm} +  \sum_{i=1}^n \big(S^2_{r(i)}(A\triangle B)+o(1)\big)P\{R_{i/n}>u_{n,r(i)}\}\bigg]\\
     & \le 16h_n \sum_{j=1}^{J_n+1} S_{r_j}(A\triangle B)+o(1)\\
     & \to 8\int_0^1 S_r(A\triangle B)\, dr = 8\rho_I(A,B).
   \end{align*}
    Now \eqref{eq:cond2} is obvious and it remains to establish the entropy condition in Theorem 2.11.1 of \cite{vdVW96}. Define a random semi-metric $d_n$ on $\FF$ by
   \begin{align*}
     & d_n^2  \big((s,A),(t,B)\big)\\
     & = \sum_{i=1}^n \big(Y_{n,i}^*(s,A)-Y_{n,i}^*(t,B)\big)^2\\
 & \le \frac 2{k_nh_n} \sum_{i=1}^n \big( \Indd{R_{i/n}>u_{n,r(i)}, \Theta_{i/n}\in A}-S_{n,r(i)}(A)\Indd{R_{i/n}>u_{n,r(i)}}\big)^2 \Big(\int_{r(i)-h_n}^{r(i)+h_n} \indd{(s,t]}(r)\, dr\Big)^2\\
     & \hspace{1cm}+ \frac 2{k_nh_n} \sum_{i=1}^n \big( \Indd{R_{i/n}>u_{n,r(i)}, \Theta_{i/n}\in A} -\Indd{R_{i/n}>u_{n,r(i)}, \Theta_{i/n}\in B}\\
     &  \hspace{3cm} -(S_{n,r(i)}(A)-S_{n,r(i)}(B))\Indd{R_{i/n}>u_{n,r(i)}}\big)^2 \Big(\int_{r(i)-h_n}^{r(i)+h_n} \indd{[0,t]}(r)\, dr\Big)^2\\
     &  \le \frac 4{k_n} \sum_{i=1}^n \Indd{R_{i/n}>u_{n,r(i)}}\int_{r(i)-h_n}^{r(i)+h_n} \indd{(s,t]}(r)\, dr\\
     & \hspace{1cm}+ \frac {16h_n}{k_n} \sum_{i=1}^n \big(\Indd{R_{i/n}>u_{n,r(i)},\Theta_{i/n}\in A}-\Indd{R_{i/n}>u_{n,r(i)},\Theta_{i/n}\in B}\big)^2\\
     & \hspace{1cm}+ \frac {16h_n}{k_n} \sum_{i=1}^n \big( S_{n,r(i)}(A)-S_{n,r(i)}(B)\big)^2 \Indd{R_{i/n}>u_{n,r(i)}}\\
     & =: 4 d_{n,1}^2(s,t) + 16 d_{n,2}^2(A,B) + 16 d_{n,3}^2(A,B).
   \end{align*}

   Denote the uniform distribution on an interval $(a,b]$ by $\UU_{(a,b]}$. Define a random probability measure $Q_1$ on $[0,1]$ by
   $$ Q_1 := \frac{\sum_{j=1}^{J_n+1} \sum_{i\in I_{n,r_j}} \Indd{R_{i/n}>u_{n,r_j}} \UU_{(r_j-h_n,r_j+h_n]}}{\sum_{j=1}^{J_n+1} \sum_{i\in I_{n,r_j}} \Indd{R_{i/n}>u_{n,r_j}}}
   $$
   where we assume w.l.o.g.\ that the denominator is positive. (Else $d_{n,1}\equiv 0$ and the corresponding covering number equals 1.) Then the $L_1(Q_1)$-distance between $\indd{[0,s]}$ and $\indd{[0,t]}$ equals
   \begin{align*}
      \| \indd{[0,s]}  -\indd{[0,t]}\|_{L_1(Q_1)}
      & =  \frac{\sum_{i=1}^{n} \Indd{R_{i/n}>u_{n,r(i)}} \int_{r(i)-h_n}^{r(i)+h_n} \indd{(s,t]}(r)\, dr}{2h_n\sum_{i=1}^{n}  \Indd{R_{i/n}>u_{n,r(i)}}}.
   \end{align*}
   According to Lemma \ref{lemma:bernstein}, the denominator is bounded by $2 h_n (J_n+1)(3/2)k_n\le 2k_n$ with probability $1-O((h_n/k_n)^\tau J_n)$ for all $\tau>0$. Thus, with this probability,
   $$ d_{n,1}^2(s,t) \le  2 \| \indd{[0,s]}  -\indd{[0,t]}\|_{L_1(Q_1)}, $$
   which in turn implies the following inequality for covering numbers w.r.t.\ $d_{n,1}$ and $L_1(Q_1)$, respectively:
   $$ N\big(\eta,(\indd{[0,t]})_{t\in[0,1]},d_{n,1}\big) \le N\big(\eta^2/2,(\indd{[0,t]})_{t\in[0,1]},L_1(Q_1)\big)\le 2L \eta^{-2}
   $$
   for some universal constant $L$. The last inequality follows from Theorem 2.6.4 of \cite{vdVW96} and the fact that the family $(\indd{[0,t]})_{t\in[0,1]}$ has VC-index 2.

   Similarly, for
   $$ Q_2 := \frac{\sum_{i=1}^{n} \Indd{R_{i/n}>u_{n,r(i)}} \eps_{\Theta_{i/n}}}{\sum_{i=1}^{n} \Indd{R_{i/n}>u_{n,r(i)}}}
   $$
   (with $\eps_\theta$ denoting the Dirac measure with point mass at $\theta$) one has
   $$ d_{n,2}^2(A,B) = \frac{h_n}{k_n}\sum_{i=1}^{n} \Indd{R_{i/n}>u_{n,r(i)}}\|1_A-1_B\|_{L_1(Q_2)} \le \|1_A-1_B\|_{L_1(Q_2)}
   $$
   with probability $1-O((h_n/k_n)^\tau J_n)$ and hence
   $$ N(\eta,\AA,d_{n,2}) \le N(\eta^2,\AA,L_1(Q_2))\le L \eta^{-2(\VC-1)}. $$

   Next define a probability measure
   $$ Q_3 := \frac 1{J_n+1} \sum_{j=1}^{J_n+1}\frac{\sum_{i\in I_{n,r_j}} P\{R_{i/n}>u_{n,r_j}\} P^{\Theta_{i/n}\mid R_{i/n}>u_{n,r_j}}}{\sum_{i\in I_{n,r_j}} P\{R_{i/n}>u_{n,r_j}\}}.
   $$
   Since $\sum_{i\in I_{n,r_j}} \Indd{R_{i/n}>u_{n,r_j}}\le 2k_n$ for all $1\le j\le J_n$ with probability $1-O((h_n/k_n)^\tau J_n)$, one obtains
   \begin{align*}
     & d_{n,3}^2(A,B)\\
      & \le  \frac{h_n}{k_n} \sum_{j=1}^{J_n+1} |S_{n,r_j}(A)-S_{n,r_j}(B)|\sum_{i\in I_{n,r_j}} \Indd{R_{i/n}>u_{n,r_j}}\\
      & \le \frac{h_n}{k_n} 2k_n \sum_{j=1}^{J_n+1}\bigg| \frac{\sum_{i\in I_{n,r_j}} \big(P\{R_{i/n}>u_{n,r_j},\Theta_{i/n}\in A\}-P\{R_{i/n}>u_{n,r_j},\Theta_{i/n}\in B\}\big)}{\sum_{i\in I_{n,r_j}} P\{R_{i/n}>u_{n,r_j}\}}\bigg|\\
     &  \le  2h_n(J_n+1)  \|\indd{A}-\indd{B}\|_{L_1(Q_3)} \le 2 \|\indd{A}-\indd{B}\|_{L_1(Q_3)}.
   \end{align*}
   As above, we may conclude $P\big\{N\big(\eta,\AA,d_{n,3}\big) \le L  (\eta/2)^{-2(\VC-1)}\big\}=1-O((h_n/k_n)^\tau J_n)$.

   A combination of the bounds on the three covering numbers yields
   \begin{align*}
    N(\eta,\FF,d_n) & \le N\big(\eta/4,(\indd{[0,t]})_{t\in[0,1]},d_{n,1}\big) \cdot N(\eta/8,\AA,d_{n,2}) \cdot N\big(\eta/8,\AA,d_{n,3}\big)
     \le M \eta^{2-4\VC}
   \end{align*}
   for some constant $M$ with probability tending to 1. Now the entropy condition of Theorem 2.11.1 is immediate, which concludes the proof.
\end{proofof}

Check that
\begin{align*}
  & \widetilde{IZ}_{n,t}(A)-\widehat{IZ}_{n,t}(A)\\
   & = 2 (k_nh_n)^{1/2} \sum_{j=1}^{\floor{t/(2h_n)}} \Delta_{n,j}(A) + \Big(\frac{k_n}{h_n}\Big)^{1/2}\Big(t-2h_n \floorr{\frac{t}{2h_n}}\Big) \Delta_{n,\floor{t/(2h_n)}+1}(A),
\end{align*}
where $\Delta_{n,j}(A) := \hat S_{n,r_j}(A)-\tilde S_{n,r_j}(A)$.
Since $k_nh_n\to 0$ by Condition (R), to conclude our main result Theorem \ref{theo:IZhat}, we have to show  that
\begin{align} \label{eq:diffIZhIZt}
  \sup_{1\le J\le J_n,A\in\AA} \bigg|\sum_{j=1}^{J}\Delta_{n,j}(A)\bigg| = o_P\big((k_nh_n)^{-1/2}\big).
\end{align}

 To this end, we first bound the probability that certain order statistics of the norm of observed vectors in the $j$th block substantially deviate from the deterministic bounds $u_{n,r_j}$.
\begin{lemma} \label{lemma:orderstatapprox}
  For all $i,k\in I_{n,r_j}$, $i\ne k$, denote the $k_n$th largest order statistic among $R_{l/n}$, $l\in I_{n,r_j}\setminus\{i\}$, by $u_{n,j}^{(i)}$, and the $(k_n-1)$th largest order statistic among $R_{l/n}$, $l\in I_{n,r_j}\setminus\{i,k\}$, by $u_{n,j}^{(i,k)}$. Then, under Condition (L), there exists a constant $c$, depending only on $\eta$, such that
  \begin{align}
    \sup_{1\le j\le J_n} P\Big\{ \Big|\frac{u_{n,j}^{(i)}}{u_{n,r_j}}-1\Big|>\eps \text{ for some } i\in I_{n,r_j}\Big\} & = o\big( \exp(-ck_n)\big) \label{eq:orderstatapprox1}\\
    \sup_{1\le j\le J_n} P\Big\{ \Big|\frac{u_{n,j}^{(i,k)}}{u_{n,r_j}}-1\Big|>\eps \text{ for some } i,k\in I_{n,r_j}\Big\} & = o\big( \exp(-ck_n)\big). \label{eq:orderstatapprox2}
  \end{align}
\end{lemma}
The proof, which uses standard arguments based on Bernstein's inequality, is given in the Supplement \cite{Drees22}.

Next we bound the first  two moments of $\Delta_{n,j}(A)$. In what follows, we use the abbreviation
$$ \Snt (A) := P(\Theta_t\in A \mid R_t>u_{n,t}). $$
\begin{lemma} \label{lemma:firstmoment}
  If the Conditions (US), (US$^*$) and (L) are met, then for some $c>0$
   \begin{equation*}
   \sup_{1\le j\le J_n,A\in\AA} |E \Delta_{n,j}(A)| = O(q_n+q_n^*)+ o\big(\exp(-ck_n)\big).
   \end{equation*}
\end{lemma}
\begin{proof}
   First note that, by Condition (US), one has uniformly for all $t\in[0,1]$
   \begin{align}
     E \tilde S_{n,t}(A) & = E\Big[ E\Big( N_t^{-1}\sum_{i\in I_{n,t}} \Indd{R_{i/n}>u_{n,t},\Theta_{i/n}\in A}\,\Big|\, \big(\Indd{R_{l/n}>u_{n,t}}\big)_{l\in I_{n,t}}\Big)\Big]  \nonumber\\
     & = E\Big[ N_t^{-1}\sum_{i\in I_{n,t}} \Indd{R_{i/n}>u_{n,t}} P(\Theta_{i/n}\in A\mid R_{i/n}>u_{n,t})\Big]  \nonumber\\
     & =     \Snt(A) + O(q_n). \label{eq:firstT1bound}
   \end{align}

   To approximate $E (\hat S_{n,r_j})$, note that for $i\in I_{n,r_j}$ the condition $R_{i/n}>\hat u_{n,r_j}$ is equivalent to $R_{i/n}>u_{n,j}^{(i)}$, with $u_{n,j}^{(i)}$ defined in Lemma \ref{lemma:orderstatapprox}. Thus, by \eqref{eq:orderstatapprox1},
   \begin{align*}
      k_n E(\hat S_{n,r_j}(A))
      & = E\Big(  \sum_{i\in I_{n,r_j}} \Indd{R_{i/n}>\hat u_{n,r_j}, \Theta_{i/n}\in A} \Indd{|u_{n,j}^{(i)}/u_{n,r_j}-1|\le\eps}\Big)\\
      & \hspace{1cm} + O\bigg(E\Big(  \sum_{i\in I_{n,r_j}} \Indd{R_{i/n}>\hat u_{n,r_j}}
     \Indd{|u_{n,j}^{(i)}/u_{n,r_j}-1|>\eps}\Big)\bigg)\\
     & =
     \sum_{i\in I_{n,r_j}} E\Big(\Indd{R_{i/n}>u_{n,j}^{(i)}, \Theta_{i/n}\in A} \Indd{|u_{n,j}^{(i)}/u_{n,r_j}-1|\le\eps}\Big) + o\big(k_n\exp(-ck_n)\big).
   \end{align*}
   The expectation on the right hand side equals
   \begin{align*}
    \int & P\{R_{i/n}>u,\Theta_{i/n}\in A\} \indd{[(1-\eps)u_{n,r_j}, (1+\eps)u_{n,r_j}]}(u)\, P^{\unji}(du)\\
      & = \int \big(\Snrj(A)+O(q_n+q_n^*)\big)P\{R_{i/n}>u\} \indd{[(1-\eps)u_{n,r_j}, (1+\eps)u_{n,r_j}]}(u)\, P^{\unji}(du)\\
      & = \big(\Snrj(A)+O(q_n+q_n^*)\big)P\big\{R_{i/n}>\unji\in [(1-\eps)u_{n,r_j}, (1+\eps)u_{n,r_j}]\big\},
    \end{align*}
    where we have used the Conditions (US) and (US$^*$) in the first step.
    Hence, applying again \eqref{eq:orderstatapprox1}, we conclude
    \begin{align}
     k_n  E(\hat S_{n,r_j}(A)) & = \big(\Snrj(A)+O(q_n+q_n^*)\big) \times  \nonumber\\
      & \hspace{-1.5cm} \times \sum_{i\in I_{n,r_j}} P\{R_{i/n}>\hat u_{n,r_j}, \unji\in [(1-\eps)u_{n,r_j}, (1+\eps)u_{n,r_j}]\big\} + o\big(k_n  \exp(-ck_n)\big) \nonumber\\
     & = k_n \Snrj(A)+ O\big(k_n(q_n+q_n^*)\big) + o\big(k_n  \exp(-ck_n)\big)\label{eq:secondT1bound}
    \end{align}
    uniformly for all $1\le j\le J_n$ and $A\in\AA$.
    A combination of \eqref{eq:firstT1bound} and \eqref{eq:secondT1bound} yields
    the assertion.
\end{proof}

\begin{lemma} \label{lemma:secondmoment}
  If the Conditions (US), (US$^*$) and (L) are met and $\log h_n=o(k_n)$, then
   \begin{equation*}
   \sup_{1\le j\le J_n,A\in\AA} E \Delta_{n,j}^2(A) = O\big(q_n+q_n^*+\big(k_n^{-3}\log(k_n/h_n)\big)^{1/2}\big).
   \end{equation*}
\end{lemma}
The proof, given in the Supplement \cite{Drees22}, resembles the one of Lemma \ref{lemma:firstmoment}. It is, though, substantially more involved, because in the mixed terms $E(\hat S_{n,r_j}(A)\tilde S_{n,t}(A))$ exceedances over random thresholds and exceedances over deterministic threshold occur jointly.

While Theorem 2.11.1 of \cite{vdVW96} cannot be applied directly to prove that $\sum_{j=1}^J \Delta_{n,j}(A)$ is uniformly negligible (i.e.\ that \eqref{eq:diffIZhIZt} holds), ideas from its proof turn out to be useful.
\begin{proposition} \label{prop:sumdeltasqconv}
  If the Conditions (US), (US$^*$), (A$^*$), (R) and (L) are met, then
  \begin{equation*}
   k_nh_n\log(k_n/h_n)E\bigg(\sup_{A\in\AA}\sum_{j=1}^{J_n} \Delta_{n,j}^2(A)\bigg)\to 0.
  \end{equation*}
\end{proposition}
\begin{proof}
  By Lemma \ref{lemma:secondmoment} and Condition (R), we have
  \begin{align*}
   k_nh_n\log\frac{k_n}{h_n}\sup_{ A\in\AA}\sum_{j=1}^{J_n} E\Delta_{n,j}^2(A)
   &  = O\Big(k_n\log\frac{k_n}{h_n}(q_n+q_n^*) +\Big(k_n^{-1}\log^3\frac{k_n}{h_n}\Big)^{1/2}\Big)  = o(1)
  \end{align*}
  and so
  \begin{equation*}
    k_nh_n\log\frac{k_n}{h_n} E\sup_{A\in\AA}\sum_{j=1}^{J_n} \Delta_{n,j}^2(A) = k_nh_n\log\frac{k_n}{h_n} E\sup_{A\in\AA}\bigg|\sum_{j=1}^{J_n} \big(\Delta_{n,j}^2(A)-E\Delta_{n,j}^2(A)\big)\bigg|+o(1).
  \end{equation*}
  By Lemma 2.3.6 of \cite{vdVW96},  the expectation on the right hand side can be bounded by
  $ 2E\sup_{A\in\AA}\Big|\sum_{j=1}^{J_n} \xi_j \Delta_{n,j}^2(A)\Big|, $
  with $\xi_j$ denoting iid Rademacher random variables, independent of $\XX=(\Indd{R_{i/n}>u_{n,r(i)}},\Theta_{i/n})_{1\le i\le n}$.

  Recall from Lemma \ref{lemma:bernstein} 
   that for a sufficiently large constant $M$ and $n\in\N$ the set
  \begin{equation} \label{eq:Bndef}
  B_n:=\{|N_{r_j}-k_n|\le M(k_n\log(k_n/h_n))^{1/2} \text{ for all }1\le j\le J_n\}
   \end{equation}
   has probability at least $1-(h_n/k_n)^2$. Direct calculations show that $|\Delta_{n,j}(A)|\le|N_{r_j}-k_n|/N_{r_j}$ which is hence eventually bounded by $2M(k_n/\log(k_n/h_n))^{-1/2}$ on the set $B_n$. Moreover, on this event, there are $O(k_n/h_n)$ indices $i\in\{1,\ldots,n\}$ with $R_{i/n}>u_{n,r(i)}\wedge \hat u_{n,r(i)}$. Thus by Condition (A$^*$) and Sauer's lemma (\cite{vdVW96}, Cor.\ 2.6.3), for fixed $\XX$ and fixed Rademacher variables $\xi_i$, the sum $\sum_{j=1}^{J_n} \xi_j\Delta_{n,j}^2(A)$ attains only $O\big((k_n/h_n)^{\VC-1}\big)$ different values as $A$ varies over $\AA$.

  Denote by $\|\cdot\|_{\psi_2,\xi}$ the Orlicz norm w.r.t.\ $\psi_2(x):=\exp(x^2)-1$ and the conditional distribution given $\XX$. Then  the inequalities on p.\ 95 and Lemma 2.2.2 of \cite{vdVW96} yield
  \begin{align*}
    E \bigg(\sup_{A\in\AA} \Big|\sum_{j=1}^{J_n} \xi_j \Delta_{n,j}^2(A)\Big|\,\bigg|\, \XX\bigg)
    & \le \sqrt{\log 2} \bigg\| \sup_{A\in\AA} \Big|\sum_{j=1}^{J_n} \xi_j \Delta_{n,j}^2(A)\Big|\bigg\|_{\psi_2,\xi}\\
    & \le K \Big(\VC \log\frac{k_n}{h_n}\Big)^{1/2} \sup_{A\in\AA} \bigg\| \sum_{j=1}^{J_n} \xi_j \Delta_{n,j}^2(A)\bigg\|_{\psi_2,\xi}
  \end{align*}
  for some universal constant $K$. Now, again on the set $B_n$, by Hoeffding's inequality (\cite{vdVW96}, Lemma 2.2.7), the norm on the right hand side is bounded by
  $$ \sqrt{6}\bigg(\sum_{j=1}^{J_n} \Delta_{n,j}^4(A)\bigg)^{1/2} \le 2\sqrt{6} M  (k_n/\log(k_n/h_n))^{-1/2} \bigg(\sum_{j=1}^{J_n} \Delta_{n,j}^2(A)\bigg)^{1/2}.
  $$
  Therefore, by Markov's inequality
  \begin{align}
    &P \Big\{ \sup_{A\in \AA} \Big|\sum_{j=1}^{J_n} \xi_j\Delta_{n,j}^2(A)\Big|>t\Big\} \nonumber\\
    & \le P(B_n^c) + t^{-1}2\sqrt{6\VC}KM\frac{\log(k_n/h_n)}{k_n^{1/2}} E\bigg(\bigg(\sup_{A\in\AA}\sum_{j=1}^{J_n} \Delta_{n,j}^2(A)\bigg)^{1/2}\bigg). \label{eq:quantbound}
  \end{align}
  Next, we apply the Hoffmann-J{\o}rgensen inequality for moments as given in Proposition A.1.5 of \cite{vdVW96}. Let $K^*:= 4\sqrt{6\VC}KM/(1-u_1)$ (with $u_1$ denoting the constant of this proposition) and
  $$ t_1 := K^* \frac{\log(k_n/h_n)}{k_n^{1/2}} E\bigg(\bigg(\sup_{A\in\AA}\sum_{j=1}^{J_n} \Delta_{n,j}^2(A)\bigg)^{1/2}\bigg).
  $$
  If $n$ is sufficiently large such that $P(B_n^c)<(1-u_1)/2$, then the right hand side of \eqref{eq:quantbound} evaluated for $t=t_1$ is less than $1-u_1$, which shows that $t_1$ is not smaller than the $u_1$-quantile of $\sup_{A\in \AA} \Big|\sum_{j=1}^{J_n} \xi_j\Delta_{n,j}^2(A)\Big|$. Hence the Hoffmann-J{\o}rgensen inequality and  $|\Delta_{n,j}(A)|\le 1 $ on $B_n^c$ and $|\Delta_{n,j}(A)|\le 2M\big(k_n/\log(k_n/h_n)\big)^{-1/2}$ on $B_n$ imply
  \begin{align*}
    & E \sup_{A\in \AA} \Big|\sum_{j=1}^{J_n} \xi_j\Delta_{n,j}^2(A)\Big|\\
    & \le C_1\Big( E \sup_{A\in \AA} \Delta_{n,j}^2(A) + t_1\Big)\\
    & \le C_1\big( 4 M^2 k_n^{-1} \log(k_n/h_n) + P(B_n^c) + t_1\big)\\
    & \le o\big((k_nh_n\log(k_n/h_n))^{-1}\big) + C_1K^* k_n^{-1/2}\log(k_n/h_n) \bigg(E\sup_{A\in\AA}\sum_{j=1}^{J_n} \Delta_{n,j}^2(A)\bigg)^{1/2},
  \end{align*}
  the last step following from Condition (R) and Jensen's inequality.
  To sum up, we have shown that
  \begin{align*}
    & k_nh_n\log(k_n/h_n) E\sup_{A\in\AA}\sum_{j=1}^{J_n} \Delta_{n,j}^2(A)\\
    & \le o(1)+2C_1K^* \big(h_n\log^3(k_n/h_n)\big)^{1/2} \Big(k_nh_n\log(k_n/h_n)E\sum_{j=1}^{J_n} \Delta_{n,j}^2(A)\Big)^{1/2}.
  \end{align*}
  Since $h_n\log^3(k_n/h_n)$ tends to 0 by Condition (R), this is only possible if the left hand side tends to 0, which is the assertion.
\end{proof}

\begin{proofof} Theorem \ref{theo:IZhat}.\rm\quad
  Recall that we have to verify \eqref{eq:diffIZhIZt}.
    Lemma \ref{lemma:firstmoment} and Condition (R) imply
  \begin{equation*}
     \sup_{1\le J\le J_n,A\in\AA}  \sum_{j=1}^{J} |E \Delta_{n,j}(A)| =O\Big(\frac{q_n+q_n^*}{h_n}\Big)+o\big(h_n^{-1}\exp(-ck_n)\big) = o\big((h_nk_n)^{-1/2}\big).
  \end{equation*}
  It remains to be shown that
   \begin{equation*}
     \sup_{1\le J\le J_n,A\in\AA} \Big|\sum_{j=1}^{J} \big( \Delta_{n,j}(A)-E \Delta_{n,j}(A)\big)\Big| =o_P\big((h_nk_n)^{-1/2}\big).
   \end{equation*}

   One can easily conclude from Proposition \ref{prop:sumdeltasqconv} that this bound holds for all fixed $J$  and $A$. Therefore, Lemma 2.3.7 of \cite{vdVW96} implies that for all $x>0$ eventually
   \begin{align*}
     P& \Big\{ (h_nk_n)^{1/2}\sup_{1\le J\le J_n,A\in\AA}\Big|\sum_{j=1}^{J} \big( \Delta_{n,j}(A)-E \Delta_{n,j}(A)\big)\Big|>x\Big\}\\
     & \le 3 P \Big\{ 4(h_nk_n)^{1/2}\sup_{1\le J\le J_n,A\in\AA}\Big|\sum_{j=1}^{J}\xi_j \Delta_{n,j}(A)\Big|>x\Big\}
   \end{align*}
   where $\xi_j$ denote iid Rademacher random variables independent of $\XX$. By the same arguments as in the proof of Proposition \ref{prop:sumdeltasqconv}, we see that for fixed $\XX$ and fixed $(\xi_j)_{1\le j\le J_n}$, on the set $B_n$ (defined in \eqref{eq:Bndef}), the sum attains at most $O\big((k_n/h_n)^{\VC}\big)$ different values. Hence, on $B_n$, Hoeffding's inequality yields
   \begin{align}
     P & \Big( 4(h_nk_n)^{1/2}\sup_{1\le J\le J_n,A\in\AA}\Big|\sum_{j=1}^{J}\xi_j \Delta_{n,j}(A)\Big|>x\,\Big|\, \XX\Big)\nonumber \\
     & \le O\Big(\Big(\frac{k_n}{h_n}\Big)^{\VC}\Big)\sup_{1\le J\le J_n,A\in\AA}
     P\bigg( \Big|\sum_{j=1}^{J}\xi_j \Delta_{n,j}(A)\Big|>\frac{x}{4(h_nk_n)^{1/2}}\,\Big|\, \XX\bigg)\nonumber\\
     & \le O\Big(\Big(\frac{k_n}{h_n}\Big)^{\VC}\Big)\sup_{A\in\AA}
      \exp\Big(-\frac{x^2}{32 h_nk_n \sum_{j=1}^{J_n}\Delta_{n,j}^2(A)}\Big).
      \label{eq:finalbound}
   \end{align}
   Now Proposition \ref{prop:sumdeltasqconv} ensures that for all $x,c>0$ with probability tending to 1
    $$ \inf_{A\in\AA}\frac{x^2}{32 h_nk_n \sum_{j=1}^{J_n}\Delta_{n,j}^2(A)}>c\log\Big( \frac{k_n}{h_n}\Big)
    $$
    and so the right hand side of \eqref{eq:finalbound} tends to 0. Since $P(B_n^c)\to 0$, the assertion follows.
\end{proofof}

\subsection{Proofs to Section \ref{sect:tests}}

\begin{proofof} Corollary  \ref{cor:testconstancy}.\rm\quad
  Check that
  \begin{align*}
    T_n^{(KS)} & = \sup_{t\in[0,1], A\in\AA} \frac 1{\sqrt{2}} \big| \widehat{IZ}_{n,t}(A)-t\cdot \widehat{IZ}_{n,1}(A) + \Big(\frac{k_n}{h_n}\Big)^{1/2}(IS_t(A)-t\cdot IS_1(A))\big| \\
     T_n^{(CM)} & =  \sup_{A\in\AA} \int_0^1 \Big( \frac 1{\sqrt{2}} \big(\widehat{IZ}_{n,t}(A)-t\cdot \widehat{IZ}_{n,1}(A) + \Big(\frac{k_n}{h_n}\Big)^{1/2}(IS_t(A)-t\cdot IS_1(A))\big)\Big)^2\, dt
  \end{align*}
  Since in the situation of (i) Condition (IS) is trivially fulfilled, the term $IS_t(A)-t\cdot IS_1(A)$ vanishes  and
  $ \int_0^{s\wedge t} c_r(A,B)\, dr = s\wedge t (S_1(A\cap B)-S_1(A)S_1(B))$, the assertion is an immediate consequence of Theorem \ref{theo:IZhat} and the continuous mapping theorem.

 Because $t\mapsto IS_t(A)-t\cdot IS_1(A)$ is a continuous function,  under the assumptions of (ii) $(k_n/h_n)^{1/2}$ $(IS_t(A)-t\cdot IS_1(A))$ converge to $\infty$ or $-\infty$ for all $t$ in a set of positive Lebesgue measure. Hence also the second assertion  follows from Theorem \ref{theo:IZhat}.
\end{proofof}

\smallskip
{\bf Acknowledgement:} I would like to thank Laurens de Haan for helpful discussions in an early stage of this project. Remarks by anonymous referees have led to an improved presentation of our ideas and results.

\bibliographystyle{imsart-number}
\bibliography{ProcLit}

\end{document}


\begin{frontmatter}
\title{Supplement to ``Statistical Inference on a Changing Extreme Value Dependence Structure''}
\runtitle{Supplement}

\begin{aug}
\author[A]{\fnms{Holger} \snm{Drees}\ead[label=e1]{holger.drees@uni-hamburg.de}}
\address[A]{Department of Mathematics,
University or Hamburg,
\printead{e1}}
\end{aug}

\begin{abstract}
  This supplement gives proofs to some of the results from the accompanying paper as well as additional simulation results. Moreover, the conditions assumed in Theorem 2.3 are verified for one of the models examined in the simulations. Finally, tail bounds on the limit distributions of the test statistics suggested in the main paper are discussed.
\end{abstract}
\end{frontmatter}

To avoid confusion, we continue the enumeration of the main paper.

\setcounter{section}{7}
\section{Additional proofs}

Here we give additional proofs to results of the main paper.

\subsection{Proofs to Subsection 2.1}

The main step in the proof of Theorem 2.1 is to prove convergence of the processes
\begin{align}  \label{eq:Zntdef}
 Z_{n,t}(s,A) := k_n^{-1/2}\sum_{i\in I_{n,t}} \Big( \Indd{R_{i/n}>s u_{n,t},\Theta_{i/n}\in A}& - P\{R_{i/n}>s u_{n,t},\Theta_{i/n}\in A\}\Big),\\
 & s\in[1-\eps,1+\eps], A\in\AA.
\end{align}

\begin{proposition} \label{prop:fixedt_procconv}
  Under the conditions of Theorem 2.1, $Z_{n,t}$, $n\in\N$, converge weakly to a centered Gaussian process $Z_t$ with $Cov\big(Z_t(r,A), Z_t(s,B)\big)=(r\vee s)^{-\alpha_t} S_t(A\cap B)$.
\end{proposition}
\begin{proof}
  As usual we first establish the convergence of all finite dimensional marginal distributions (fidis) and then the asymptotic equicontinuity of the processes.

  The convergence of the fidis follows from the Cram\'{e}r-Wold device and the CLT of Lindeberg-Feller. Note that by (RV1) and (RV2)
  $$ k_n^{-1} \sum_{i\in I_{n,t}} P\{R_{i/n}>r u_{n,t},\Theta_{i/n}\in A\} \to r^{-\alpha_t} S_t(A) $$
  uniformly for all $r\in[1-\eps,1+\eps]$ and $ A\in\AA^*$. Hence
  \begin{align*}
    \Cov\big(Z_{n,t}(r,A),Z_{n,t}(s,B)\big)
     & = k_n^{-1} \sum_{i\in I_{n,t}} \Big( P\{R_{i/n}>(r\vee s) u_{n,t},\Theta_{i/n}\in A\cap B\}-\\
    & \hspace{0.4cm} - P\{R_{i/n}>r u_{n,t},\Theta_{i/n}\in A\}P\{R_{i/n}>s u_{n,t},\Theta_{i/n}\in B\}\Big)\\
    & \to (r\vee s)^{-\alpha_t} S_t(A\cap B).
  \end{align*}
  The Lindeberg condition is trivial because the summands are bounded by 1.

  We now prove asymptotic equicontinuity, and hence the assertion, by applying Theorem 2.11.1 of \cite{vdVW96} with the semi-metric $\rho\big((r,A),(s,B)\big) := |r^{-\alpha_t}-s^{-\alpha_t}|+(1-\eps)^{-\alpha_t}\rho_t(A,B)$ to the non-centered processes $k_n^{-1/2}\sum_{i\in I_{n,t}} \Indd{R_{i/n}>s u_{n,t},\Theta_{i/n}\in A}$. Under Condition (A), $\FF:= [1-\eps,1+\eps]\times\AA$ is obviously totally bounded w.r.t.\ $\rho$ and the measurability condition is fulfilled. Moreover, the Lindeberg type condition of Theorem 2.11.1 is again trivial.

  In the present context, the second displayed condition of Theorem 2.11.1 reads as
  \begin{align}
   \lim_{\delta\downarrow 0} \limsup_{n\to \infty} \sup_{\rho((r,A),(s,B))\le\delta} k_n^{-1}\sum_{i\in I_{n,t}} P\big(  \{R_{i/n}>ru_{n,t},& \Theta_{i/n}\in A\}\nonumber \\
     & \triangle \{R_{i/n}>su_{n,t}, \Theta_{i/n}\in B\}\big) = 0.\label{eq:fixedt_cond2}
  \end{align}
  Condition (RV1) implies
  \begin{align}
      k_n^{-1} & \sum_{i\in I_{n,t}} P\big(\{R_{i/n}>ru_{n,t}, \Theta_{i/n}\in A\}\triangle \{R_{i/n}>su_{n,t}, \Theta_{i/n}\in A\}\big) \nonumber\\
      & \le k_n^{-1} \sum_{i\in I_{n,t}} P\{(r \wedge s)u_{n,t}<R_{i/n}\le (r\vee s)u_{n,t}\} \nonumber\\
      & \to (r \wedge s)^{-\alpha_t} -  (r \vee s)^{-\alpha_t} \label{eq:fixedt_eq2.1}
  \end{align}
  uniformly for $r,s\in[1-\eps,1+\eps]$ and $A\in\AA$. Similarly, (RV1) and (RV2) yield
  \begin{align}
      k_n^{-1} & \sum_{i\in I_{n,t}} P\big(\{R_{i/n}>su_{n,t}, \Theta_{i/n}\in A\}\triangle \{R_{i/n}>su_{n,t}, \Theta_{i/n}\in B\}\big) \nonumber\\
      & = k_n^{-1} \sum_{i\in I_{n,t}} P\{R_{i/n}>su_{n,t}, \Theta_{i/n}\in A\triangle B\} \nonumber\\
      & \to s^{-\alpha_t} S_t(A\triangle B) \label{eq:fixedt_eq2.2}
  \end{align}
  uniformly for $s\in[1-\eps,1+\eps]$ and $A,B\in\AA$. Combining \eqref{eq:fixedt_eq2.1} and \eqref{eq:fixedt_eq2.2}, we arrive at  \eqref{eq:fixedt_cond2}.

  It remains to verify the entropy condition of Theorem 2.11.1 of \cite{vdVW96}. Let
  \[ \tilde\FF :=\big\{\indd{(r,\infty)\times A} \mid r\in[1-\eps,1+\eps],A\in\AA\big\}.
  \]
  By Condition (A) and Lemma 2.6.17 (vii) of \cite{vdVW96}, $\tilde\FF$ is a VC class with some index $VC(\tilde\FF)$. Using Theorem 2.6.7 of \cite{vdVW96}, we may conclude that there exists a constant $C$ such that,  for all probability measures $Q$ on $(0,\infty)\times \sphere$ and all $\eta>0$, the minimal number $N\big(\eta\|\indd{(1-\eps,\infty)\times\sphere}\|_{L_2(Q)},\tilde\FF,L_2(Q)\big)$ of  $L_2(Q)$-balls of radius $\eta\|\indd{(1-\eps,\infty)\times\sphere}\|_{L_2(Q)}$ covering $\tilde\FF$ is bounded by $C\eta^{-2(VC(\tilde\FF)-1)}$.
  Let $m_{n,t}:=|I_{n,t}|$ and define a random probability measure on $(0,\infty)\times \sphere$
  $$
    Q_n := \frac 1{m_{n,t}} \sum_{i\in I_{n,t}} \eps_{(R_{i/n}/u_{n,t},\Theta_{i/n})}
  $$
  and a random metric $d_n$ on $\FF$ by
  $$
   d_n^2\big((r,A),(s,B)\big)  := k_n^{-1} \sum_{i\in I_{n,t}} \big(\Indd{R_{i/n}>ru_{n,t}, \Theta_{i/n}\in A}-\Indd{R_{i/n}>su_{n,t}, \Theta_{i/n}\in B}\big)^2
  $$
  Then
  $$
       \big\|\indd{(r,\infty)\times A}-\indd{(s,\infty)\times B}\big\|_{L_2(Q_n)} = m_{n,t}^{-1/2} k_n^{1/2}  d_n\big((r,A),(s,B)\big)
  $$
   and
  $$ \big\|\indd{(1-\eps,\infty)\times\sphere}\big\|_{L_2(Q_n)}^2= \frac{1}{m_{n,t}} \nu_{n,t}\big((1-\eps)u_{n,t},\infty\big) = (1-\eps)^{-\alpha_t}\frac{k_n}{m_{n,t}}(1+o(1))
  $$
  by (RV1). Hence, eventually
  \begin{align*}
   N(\eta,\tilde\FF,d_n) & = N\big(\eta (k_n/m_{n,t})^{1/2},\tilde\FF,L_2(Q_n)\big) \\
   & \le   N\Big((\eta/2) (1-\eps)^{\alpha_t/2}\|\indd{(1-\eps,\infty)\times\sphere}\|_{L_2(Q_n)},\tilde\FF,L_2(Q_n)\Big)\\
   & \le C \Big( \frac{\eta(1-\eps)^{\alpha_t/2}}{2}\Big)^{-2(VC(\tilde\FF)-1)}.
  \end{align*}
  Now the entropy condition follows readily, which concludes the proof.
 \end{proof}

\begin{corollary} \label{corol:consist}
 If the conditions of Theorem 2.1 are met, then $\hat u_{n,t}/u_{n,t}\to 1$ in probability.
\end{corollary}
\begin{proof}
   First note that $\hat u_{n,t}<(1-\delta)u_{n,t}$ is equivalent to $\sum_{i\in I_{n,t}}\Indd{R_{i/n}>(1-\delta)u_{n,t}}  \le k_n$. Rewrite the sum  in terms of the process $Z_{n,t}$ defined in \eqref{eq:Zntdef} to obtain for all $\delta\in (0,\eps)$
  \begin{align*}
    P &\{\hat u_{n,t}\le (1-\delta)u_{n,t}\} \\
      & = P\Big\{ k_n^{1/2}Z_{n,t}(1-\delta,\sphere)+\nu_{n,t}\big((1-\delta)u_{n,t},\infty\big)  \le k_n\Big\}\\
      & = P\Big\{ Z_{n,t}(1-\delta,\sphere)\le k_n^{-1/2} \big(k_n-(1-\delta)^{-\alpha_t}k_n(1+o(1)))\\
      & \hspace*{5cm} = k_n^{1/2}\big(1-(1-\delta)^{-\alpha_t}+o(1)\big)\Big\},
  \end{align*}
   where in the last step we have used (RV1).
  Since $Z_{n,t}(1-\delta,\sphere)$ is asymptotically tight and the right-hand side of the last inequality tends to $-\infty$, it follows that $P\{\hat u_{n,t}\le (1-\delta)u_{n,t}\}\to 0$.
  Likewise, one can show that $P \{\hat u_{n,t}\ge (1+\delta)u_{n,t}\}\to 0$, which concludes the proof.
\end{proof}

Now the proof of the main result of Section 2.1, which is repeated here for convenience, follows by continuous mapping arguments from Proposition \ref{prop:fixedt_procconv} and the preceding corollary.
\setcounter{section}{2}\setcounter{theorem}{0}
\begin{theorem}
  If the Conditions (RV1), (RV2) and (A) are fulfilled then, for all $t\in [0,1]$, $k_n^{1/2}(\hat S_{n,t}(A)-S_t(A))_{A\in\AA}$ converges weakly to a centered Gaussian process $Z_t$ with covariance function $c_t(A,B):= \Cov(Z_t(A),$ $Z_t(B))= S_t(A\cap B)-S_t(A)S_t(B)$.
\end{theorem}
\setcounter{section}{8}
\begin{proof}
  According to Corollary \ref{corol:consist}, Slutsky's lemma and Skorohod's representation theorem, there exist versions of $Z_{n,t}, Z_t$ and $\hat u_{n,t}$ such that $(Z_{n,t},s_n)\to (Z_t,1)$ almost surely for $s_n:=\hat u_{n,t}/u_{n,t}$. Since $Z_{n,t}$ is asymptotically equicontinuous w.r.t.\ the semi-metric $\rho$ (defined in the proof of Proposition \ref{prop:fixedt_procconv}), $Z_t$ has a.s.\ continuous sample paths w.r.t.\ $\rho$. In particular, with probability tending to 1,
  \begin{align*}
    \sup_{A\in\AA} & \big| Z_{n,t}(s_n,A)-Z_t(1,A)\big| \\
     & \le \sup_{A\in\AA} \big| Z_{n,t}(s_n,A)-Z_t(s_n,A)\big| + \sup_{A\in\AA} \big| Z_{t}(s_n,A)-Z_t(1,A)\big|\\
     & \le \sup_{A\in\AA,s\in[1-\eps,1+\eps]} \big| Z_{n,t}(s,A)-Z_t(s,A)\big| + \sup_{A\in\AA} \big| Z_{t}(s_n,A)-Z_t(1,A)\big| \to 0.
  \end{align*}

  Let
  $$ \bar S_{n,t}(s,A) := \frac{\sum_{i\in I_{n,t}} \Indd{R_{i/n}>su_{n,t},\Theta_{i/n}\in A}}{\sum_{i\in I_{n,t}} \Indd{R_{i/n}>su_{n,t}}},
  $$
   so that $\hat S_{n,t}(A)=\bar S_{n,t}(s_n,A)$. By  (RV2)
   \begin{equation*} 
        S_{n,t}(s,A)  :=  \frac{\sum_{i\in I_{n,t}} P\{R_{i/n}>su_{n,t},\Theta_{i/n}\in A\}}{\nu_{n,t}(su_{n,t},\infty)} = S_t(A)+o\big(k_n^{-1/2}\big)
   \end{equation*}
  uniformly for all $s\in[1-\eps,1+\eps]$ and $A\in\AA$.
  Since   $\nu_{n,t}(s_nu_{n,t},\infty)=k_n(1+o(1))$ by (RV1), we conclude
  \begin{align*}
    & k_n^{1/2}  \big(\bar S_{n,t}(s_n,A)-S_{t}(A)\big)\\
      & = k_n^{1/2} \bigg(\frac{k_n^{1/2}Z_{n,t}(s_n,A)+\sum_{i\in I_{n,t}} P\{R_{i/n}>s_n u_{n,t}, \Theta_{i/n}\in A\}}{k_n^{1/2}Z_{n,t}(s_n,\sphere)+\nu_{n,t}(s_nu_{n,t},\infty)}- S_{n,t}(s_n,A)\bigg)+o(1) \\
      & = k_n^{1/2}\frac{k_n^{1/2}Z_{n,t}(s_n,A)-S_{n,t}(s_n,A)k_n^{1/2}Z_{n,t}(s_n,\sphere)}{k_n(1+o(1))}+o(1)\\
      & \to Z_t(1,A)-S_t(A) Z_t(1,\sphere)
  \end{align*}
  uniformly for all $A\in\AA$. Now the assertion follows from straightforward covariance calculations.
\end{proof}

\subsection{Technical results to Subsection 2.2}

\setcounter{section}{7}\setcounter{theorem}{0} \setcounter{equation}{0}
\begin{lemma} \label{lemma:bernstein}
  Fix some $d>0$. If (A$^*$) holds and $\log h_n =o(k_n)$, then there exists a constant $K$, depending only on $\VC$, such that eventually for all $r\in[0,1]$
  \begin{align}
    P\Big\{\sup_{A\in\AA}\Big|\sum_{i\in I_{n,r}} \Indd{R_{i/n}>u_{n,r},\Theta_{i/n}\in A} -w_{n,r}(A)\Big|\ge d\big(k_n\log(k_n/h_n)\big)^{1/2}\Big\} 
     & \le K\Big(\frac{h_n}{k_n}\Big)^{d^2/50-\VC}.\label{eq:bernstein}
  \end{align}
\end{lemma}
\setcounter{section}{8} \setcounter{equation}{4}
\begin{proof}
    By Bernstein's inequality  (see, e.g., \cite{BLM12}, (2.10)), we have, for all fixed $r\in[0,1]$, $c>0$, all $A\in\AA^*$ and sufficiently large $n$
    \begin{align}
        P\Big\{& \Big|\sum_{i\in I_{n,r}} \Indd{R_{i/n}>u_{n,r},\Theta_{i/n}\in A} -w_{n,r}(A)\Big|\ge c\big(k_n\log(k_n/h_n)\big)^{1/2}\Big\} \nonumber\\
        & \le 2\exp\bigg(-\frac{c^2 k_n \log(k_n/h_n)}{2\big[w_{n,r}(A)+c\big(k_n\log(k_n/h_n)\big)^{1/2}/3\big]}\bigg) 
        \le 2\Big(\frac{k_n}{h_n}\Big)^{-c^2/3} \to 0. \label{eq:bern}
    \end{align}
     Thus, according to the symmetrization Lemma 2.3.7 of \cite{vdVW96}, eventually the probability on the left hand side of \eqref{eq:bernstein} can be bounded by
    \begin{equation} \label{eq:Rademacher}
      3 P\Big\{\sup_{A\in\AA}\Big|\sum_{i\in I_{n,r}} \xi_i\Indd{R_{i/n}>u_{n,r},\Theta_{i/n}\in A}\Big|>\frac d4 \big(k_n\log(k_n/h_n)\big)^{1/2}\Big\}
    \end{equation}
    with iid Rademacher random variables $\xi_i$ (i.e., $P\{\xi_i=-1\}=P\{\xi_i=1\}=1/2$) independent of $\bfX:=(X_r)_{r\in[0,1]}$.

    Again by Bernstein's inequality, we have for all $A\in\AA$
    \begin{align*}
     P\Big( \Big|\sum_{i\in I_{n,r}} & \xi_i\Indd{R_{i/n}>u_{n,r},\Theta_{i/n}\in A}\Big|\ge \frac d4 \big(k_n\log(k_n/h_n)\big)^{1/2}\,\Big|\, \bfX\Big)\\
      & \le 2\exp\Big(-\frac{d^2k_n\log(k_n/h_n)}{32\big(V_{n,r}(A)+ d\big(k_n\log(k_n/h_n)\big)^{1/2}/12\big)} \Big)
    \end{align*}
    with
    $$ V_{n,r}(A):=\sum_{i\in I_{n,r}} \Indd{R_{i/n}>u_{n,r},\Theta_{i/n}\in A}\le \frac{3}{2} k_n $$
    on the set $D_n := \big\{ \big|\sum_{i\in I_{n,r}} \Indd{R_{i/n}>u_{n,r}}-k_n\big|\le k_n/2\}$. Therefore, on the set $D_n$ one has eventually
    $$ P\Big( \Big|\sum_{i\in I_{n,r}} \xi_i\Indd{R_{i/n}>u_{n,r},\Theta_{i/n}\in A}\Big|\ge \frac d4 \big(k_n\log(k_n/h_n)\big)^{1/2}\,\Big|\, \bfX\Big)
    \le 2 \Big(\frac{h_n}{k_n}\Big)^{d^2/50}.
    $$
    Since $\AA$ has VC-index $\VC$,  by Sauer's lemma (\cite{vdVW96}, Corollary 2.6.3), there exists a constant $K$ such that on the set $D_n$ for any fixed $\bfX$, there are at most $(K/10) k_n^{\VC-1}$ different sets $\{i\in I_{n,r}\mid R_{i/n}>u_{n,r},\Theta_{i/n}\in A\}$, $A\in\AA$.
    It follows that on $D_n$
    \begin{align}
     P\Big( \sup_{A\in\AA}& \Big|\sum_{i\in I_{n,r}} \xi_i\Indd{R_{i/n}>u_{n,r},\Theta_{i/n}\in A}\Big|\ge \frac d4 \big(k_n\log(k_n/h_n)\big)^{1/2}\,\Big|\, \bfX\Big) \nonumber \\
     & \le \frac{K}5 k_n^{\VC-1}\Big(\frac{h_n}{k_n}\Big)^{d^2/50}.\label{eq:VCpropbound}
    \end{align}
    On the other hand, by \eqref{eq:bern}, $P(D_n^c)=o((h_n/k_n)^{d^2/50})$ because $(k_n\log(k_n/ h_n))^{1/2}=o(k_n)$. Combine this with \eqref{eq:Rademacher} and \eqref{eq:VCpropbound} to conclude that the left hand side of \eqref{eq:bernstein} can be bounded by $K k_n^{\VC-1}(h_n/k_n)^{d^2/50} \le K (h_n/k_n)^{d^2/50-\VC}$.
\end{proof}

Using Lemma \ref{lemma:bernstein}, we conclude an approximation of $\widetilde{IZ}_{n,t}$ by structurally simpler processes.
\setcounter{section}{7}\setcounter{theorem}{1}
\begin{proposition} \label{cor:intprocapprox}
    If the Conditions  (US), (B), (IS), (A$^*$) and (R) are met, then the process $\widetilde{IZ}_{n,\cdot}$ defined in Proposition 2.2 fulfills
\begin{align*}
    \sup_{t\in[0,1],A\in\AA}\bigg|\widetilde{IZ}_{n,t}(A)-\sum_{i=1}^n Y_{n,i}(t,A)\Big| =o_P(1).
\end{align*}
\end{proposition}
\setcounter{section}{8}
\begin{proof}
  We split $(k_n/h_n)^{-1/2}\widetilde{IZ}_{n,t}(A)=\widetilde{IS}_{n,t}(A)-IS_{t}(A)$ into three terms as follows:
  \begin{align*}
     \widetilde{IS}_{n,t}(A)-IS_{t}(A) & = \int_0^t \tilde S_{n,(2\ceil{r/(2h_n)}-1)h_n}(A)-S_{n,(2\ceil{r/(2h_n)}-1)h_n}(A)\, dr \nonumber \\
     & \hspace{1cm} + \int_0^t S_{n,(2\ceil{r/(2h_n)}-1)h_n}(A)- S_{(2\ceil{r/(2h_n)}-1)h_n}(A)\, dr \nonumber \\
     & \hspace{1cm} + \int_0^t S_{(2\ceil{r/(2h_n)}-1)h_n}(A)-S_r(A)\, dr.
  \end{align*}
  We will show that the latter two integrals are asymptotically negligible, while the first one can be uniformly approximated by $(k_n/h_n)^{-1/2}\sum_{i=1}^n Y_{n,i}(t,A)$.

  The last integral is of smaller order than $(h_n/k_n)^{1/2}$ by (IS). In view of the definition of $u_{n,r}$ and Conditions (US) and (B), with $J_n:=\floor{1/(2h_n)}$, the second integral can be bounded by
  \pagebreak 
  \begin{align}
     2h_n & \sum_{j=1}^{J_n+1} \big|S_{n,r_j}(A)-S_{r_j}(A)\big|  \nonumber \\
     & = 2h_n \sum_{j=1}^{J_n+1} \Big| k_n^{-1} \sum_{i\in I_{n,r_j}} \big(P\{R_{i/n}>u_{n,r_j},\Theta_{i,n}\in A\}-S_{r_j}(A) P\{R_{i/n}>u_{n,r_j}\}\big)\Big| \nonumber \\
     & \le 2\frac{h_n}{k_n} \sum_{j=1}^{J_n+1} \sum_{i\in I_{n,r_j}} \big| P(\Theta_{i/n}\in A\mid R_{i/n}>u_{n,r_j})-S_{r_j}(A)\big| P\{R_{i/n}>u_{n,r_j}\} \nonumber \\
     & = O\Big(\frac{h_n}{k_n} J_n (q_n+q_n') k_n\Big) =o\big((h_n/k_n)^{1/2}\big), \label{eq:Snrapprox}
  \end{align}
  where in the last step Condition (R) has been used.

  To derive an approximation to the first integral, recall definition (2.2) of $N_r$ and let
  \begin{align*}
    \Dt_{n,r}(A) & := \tilde S_{n,r}(A)- S_{n,r}(A) - k_n^{-1}  \Big( \sum_{i\in I_{n,r}}  \Indd{R_{i/n}>u_{n,r}, \Theta_{i/n}\in A} - N_r S_{n,r}(A) \Big)\\
    & = \big(N_r^{-1}-k_n^{-1}\big)  \Big( \sum_{i\in I_{n,r}}  \Indd{R_{i/n}>u_{n,r}, \Theta_{i/n}\in A} - N_r S_{n,r}(A) \Big)
  \end{align*}
  for $r\in[0,1]$ and $A\in\AA$. Then direct calculations show that
  \begin{align*}
    \Big(\frac{k_n}{h_n}\Big)^{1/2} \int_0^t & \tilde S_{n,(2\ceil{r/(2h_n)}-1)h_n}(A)-S_{n,(2\ceil{r/(2h_n)}-1)h_n}(A)\, dr - \sum_{i=1}^n Y_{n,i}(t,A)\\
    & = \Big(\frac{k_n}{h_n}\Big)^{1/2} \sum_{j=1}^{J_n+1} \Dt_{n,r_j}(A) \int_{r_j-h_n}^{r_j+h_n} \indd{[0,t]}(r)\, dr.
  \end{align*}
  We are going to prove that the right-hand side tends to 0 in probability. Since
  $$ \Big(\frac{k_n}{h_n}\Big)^{1/2} |\Dt_{n,r_j}(A)| \int_{r_j-h_n}^{r_j+h_n} \indd{[0,t]}(r)\, dr \le \Big(\frac{k_n}{h_n}\Big)^{1/2} \Big|\frac{N_r}{k_n}-1\Big| 2h_n,
  $$
  $k_nh_n\to 0$ by (R), and $N_r/k_n=O_P(1)$ by Lemma \ref{lemma:bernstein}, to this end it suffices to show that
  \begin{equation*} 
    \sup_{J,A} (k_nh_n)^{1/2} \Big| \sum_{j=1}^J \Dt_{n,r_j}(A) \Big| = o_P(1),
  \end{equation*}
  where $\sup_{J,A}$ is a shorthand for $\sup_{J\in\{1,\ldots,J_n\},A\in\AA}$.

  Define $\RR:= (\Indd{R_{i/n}>u_{n,r(i)}})_{1\le i\le n}$, $\XX:= (\Indd{R_{i/n}>u_{n,r(i)}},\Theta_{i/n})_{1\le i\le n}$ and $\St_{n,i}(A):= P(\Theta_{i/n}\in A\mid R_{i/n}>u_{n,r(i)})$. Then Condition (US) implies $\St_{n,i}(A)=S_{n,r_j}(A)+O(q_n)$ for all $i\in I_{n,r_j}$, and thus
  \begin{align*}
    E(\Dt_{n,r_j}(A)\mid \RR) & = \big(N_{r_j}^{-1}-k_n^{-1}\big) \Big( \sum_{i\in I_{n,r_j}} \Indd{R_{i/n}>u_{n,r_j}}\St_{n,i}(A)-N_{r_j}S_{n,r_j}(A)\Big)\\
    & = \big(1-N_{r_j}/k_n\big) O(q_n),
  \end{align*}
  for all $j\in\{1,\ldots,J_n\}$ and $A\in\AA$. Moreover, since the summands $\Indd{R_{i/n}>u_{n,r_j},\Theta_{i/n}\in A}$, $i\in I_{n,r_j}$, are conditionally independent given $\RR$,
  \pagebreak
  \begin{align*}
    Var(\Dt_{n,r_j}(A)\mid \RR) & = \big(N_{r_j}^{-1}-k_n^{-1}\big)^2  \sum_{i\in I_{n,r_j}} Var\big(\Indd{R_{i/n}>u_{n,r_j},\Theta_{i/n}\in A}\mid\RR\big)\\
    & = \big(N_{r_j}^{-1}-k_n^{-1}\big)^2  \sum_{i\in I_{n,r_j}}\Indd{R_{i/n}>u_{n,r_j}}\St_{n,i}(A)(1-\St_{n,i}(A))\\
    & \le (N_{r_j})^{-1}\big(1-N_{r_j}/k_n\big)^2.
  \end{align*}
  According to Lemma \ref{lemma:bernstein}, to each $\tau>0$ there exists a constant $c$ such that the set $B_n:=\big\{|N_{r_j}/k_n-1|\le c k_n^{-1/2}(\log(k_n/h_n))^{1/2} \text{ for all } 1\le j\le J_n\big\}$ has probability greater than $1-(h_n/k_n)^\tau$. On this set, we may conclude using (R) that, for all $x>0$, eventually
  \begin{align*}
  P\Big( & (k_nh_n)^{1/2}\Big|\sum_{j=1}^J \Dt_{n,r_j}(A)\Big|>x/2\,\Big|\,\RR\Big)\\
  & \le P\Big( (k_nh_n)^{1/2}\Big|\sum_{j=1}^J (\Dt_{n,r_j}(A)-E(\Dt_{n,r_j}(A)|\RR))\Big|>x/4\,\Big|\,\RR\Big)\\
  & \le k_nh_n \sum_{j=1}^J Var(\Dt_{n,r_j}(A)\mid \RR)(4/x)^2 = O\big(k_n^{-1}\log(k_n/h_n)\big)\to 0,
  \end{align*}
  uniformly for all $1\le j\le J_n$ and $A\in\AA$. Hence, Lemma 2.3.7 of \cite{vdVW96} yields
  \begin{align*}
    & P\Big(\sup_{J,A} (k_nh_n)^{1/2} \Big|\sum_{j=1}^J \Dt_{n,r_j}(A)\Big|>x \,\Big|\,\RR\Big)\\
    & \le 4 P\Big(\sup_{J,A} (k_nh_n)^{1/2} \Big|\sum_{j=1}^J \big(N_{r_j}^{-1}-k_n^{-1}\big) \sum_{i\in I_{n,r_j}}\xi_i \Indd{R_{i/n}>u_{n,r_j},\Theta_{i/n}\in A}\Big|>x/2 \,\Big|\,\RR\Big)
  \end{align*}
  with $\xi_i$ denoting iid Rademacher random variables independent of $\XX$. By Sauer's lemma (\cite{vdVW96}, Corollary 2.6.3), for fixed values of $\XX$, the random vectors \linebreak $(\Indd{R_{i/n}>u_{n,r(i)},\Theta_{i/n}\in A})_{1\le i\le n}$, $A\in\AA$, attain at most $O((\sum_{j=1}^{J_n}N_{r_j})^{\VC-1})$ different values. Hence, on the set $B_n$, using Bernstein's inequality we obtain
  \begin{align}  \label{eq:probgivenX}
    P&\Big(\sup_{J,A} (k_nh_n)^{1/2} \Big|\sum_{j=1}^J \big(N_{r_j}^{-1}-k_n^{-1}\big) \sum_{i\in I_{n,r_j}}\xi_i \Indd{R_{i/n}>u_{n,r_j},\Theta_{i/n}\in A}\Big|>x/2 \,\Big|\,\XX\Big)\\
    & = O\big((k_n/h_n)^{\VC-1}J_n)\sup_{J,A} P\Big( (k_nh_n)^{1/2} \Big|\sum_{j=1}^J \big(N_{r_j}^{-1}-k_n^{-1}\big) \times \nonumber\\
     & \hspace*{5cm} \times \sum_{i\in I_{n,r_j}}\xi_i \Indd{R_{i/n}>u_{n,r_j},\Theta_{i/n}\in A}\Big|>x/2 \,\Big|\,\XX\Big) \nonumber \\
    & = O\bigg(\Big(\frac{k_n}{h_n}\Big)^{\VC}\times\nonumber\\
    & \hspace{1cm} \times\sup_{J,A}\exp\Big(-\frac{x^2}{8k_nh_n\big[V_{n,J}(A)+(x/6)(k_nh_n)^{-1/2}\max_{1\le j\le J_n}|N_{r_j}^{-1}-k_n^{-1}|\big]}\Big)\bigg), \nonumber
  \end{align}
  with
  \begin{align*}
   V_{n,J}(A) & =\sum_{j=1}^J (N_{r_j}^{-1}-k_n^{-1})^2 \sum_{i\in I_{n,r_j}}\Indd{R_{i/n}>u_{n,r_j},\Theta_{i/n}\in A} \\
   & \le \sum_{j=1}^{J_n} (1-N_{r_j}/k_n)^2 N_{r_j}^{-1} \le 2c^2 h_n^{-1}k_n^{-2}\log(k_n/h_n)
  \end{align*}
  and
  $$ \max_{1\le j\le J_n}|N_{r_j}^{-1}-k_n^{-1}| \le 2c k_n^{-3/2}(\log(k_n/h_n))^{1/2}. $$
  Therefore, the conditional probability \eqref{eq:probgivenX} can be bounded on $B_n$ by
  \begin{align*}
  O\bigg( & \Big(\frac{k_n}{h_n}\Big)^{\VC}\exp\Big(-\frac{x^2}{18c^2\big[k_n^{-1}\log(k_n/h_n) +(x/3)ck_n^{-1}h_n^{1/2}(\log(k_n/h_n))^{1/2}\big]}\Big)\bigg)\\
  & = O\Big(\exp\Big(-\frac{x^2k_n}{20c^2 \log(k_n/h_n)}+\log(k_n/h_n)\VC\Big)\Big)
  \end{align*}
  which tends to 0 under Condition (R). Since $P(B_n^c)\to 0$, this concludes the proof.
\end{proof}

Next we prove two technical lemmas which were used to verify that the pseudo-estimator $\widetilde{IS}$ and the estimator $\widehat{IZ}$ exhibit the same asymptotic behavior.
\setcounter{section}{7}\setcounter{theorem}{2}  \setcounter{equation}{4}
\begin{lemma} \label{lemma:orderstatapprox}
  For all $i,k\in I_{n,r_j}$, $i\ne k$, denote the $k_n$th largest order statistic among $R_{l/n}$, $l\in I_{n,r_j}\setminus\{i\}$, by $u_{n,j}^{(i)}$, and the $(k_n-1)$th largest order statistic among $R_{l/n}$, $l\in I_{n,r_j}\setminus\{i,k\}$, by $u_{n,j}^{(i,k)}$. Then, under Condition (L), there exists a constant $c$, depending only on $\eta$, such that
  \begin{align}
    \sup_{1\le j\le J_n} P\Big\{ \Big|\frac{u_{n,j}^{(i)}}{u_{n,r_j}}-1\Big|>\eps \text{ for some } i\in I_{n,r_j}\Big\} & = o\big( \exp(-ck_n)\big) \label{eq:orderstatapprox1}\\
    \sup_{1\le j\le J_n} P\Big\{ \Big|\frac{u_{n,j}^{(i,k)}}{u_{n,r_j}}-1\Big|>\eps \text{ for some } i,k\in I_{n,r_j}\Big\} & = o\big( \exp(-ck_n)\big). \label{eq:orderstatapprox2}
  \end{align}
\end{lemma}
\setcounter{section}{8}\setcounter{equation}{10}
\begin{proof}
   We only prove the second assertion, because the first can be verified by analogous arguments.
   The definition of $u_{n,j}^{(i,k)}$ implies
   \begin{align*}
     P & \big\{u_{n,j}^{(i,k)}>(1+\eps)u_{n,r_j} \text{ for some } i,k\in I_{n,r_j}\big\}\\
     & = P\Big\{ \sum_{l\in I_{n,r_j}\setminus\{i,k\}} \Indd{R_{l/n}>(1+\eps)u_{n,r_j}} \ge k_n-1\text{ for some } i,k\in I_{n,r_j}\Big\}\\
     & \le P\Big\{ \sum_{l\in I_{n,r_j}} \big(\Indd{R_{l/n}>(1+\eps)u_{n,r_j}}-P\{R_{l/n}>(1+\eps)u_{n,r_j}\}\big)\\
     & \hspace{5cm} \ge k_n-1-\nu_{n,r_j}\big((1+\eps)u_{n,r_j},\infty\big)\Big\}.
   \end{align*}
   Choose some $\tilde\eta\in (0,\eta\wedge 1)$. Since, by Condition (L), eventually for all $t\in[0,1]$
   $$ k_n-1 -\nu_{n,t}\big((1+\eps)u_{n,t},\infty\big)\ge k_n-(1-\tilde\eta)k_n = \tilde \eta k_n
   $$
   and $y\mapsto y^2/(v+y/3)$ is increasing, Bernstein's inequality yields
   \begin{align*}
     P & \big\{u_{n,j}^{(i,k)}>(1+\eps)u_{n,r_j} \text{ for some } i,k\in I_{n,r_j}\big\}\\
     & \le \exp\Big(-\frac{\big(k_n-1-\nu_{n,r_j}\big((1+\eps)u_{n,r_j},\infty\big)\big)^2}{
     2\big(\nu_{n,r_j}\big((1+\eps)u_{n,r_j},\infty\big) +\big(k_n-1-\nu_{n,r_j}\big((1+\eps)u_{n,r_j},\infty\big)\big)/3\big)}\Big)\\
     & \le \exp\Big(-\frac{\tilde\eta^2 k_n^2}{2\big(k_n(1-\tilde \eta)+ \tilde\eta k_n/3\big)}\Big)\\
     &\le  \exp\Big(-\frac{\tilde\eta^2}2 k_n\Big).
   \end{align*}
   Similarly, eventually for all $j\in\{1,\ldots,J_n\}$,
  \begin{align*}
     P & \big\{u_{n,j}^{(i,k)}<(1-\eps)u_{n,r_j} \text{ for some } i,k\in I_{n,r_j}\big\}\\
     & \le P\Big\{ \sum_{l\in I_{n,r_j}} \Indd{R_{l/n}>(1-\eps)u_{n,r_j}}< k_n+1\Big\}\\
     & \le \exp\Big(-\frac{\big(\nu_{n,r_j}\big((1-\eps)u_{n,r_j},\infty\big)-k_n-1\big)^2}{
     2\big(k_n+1 +4\big(\nu_{n,r_j}\big((1-\eps)u_{n,r_j} ,\infty\big)-k_n-1\big)/3\big)}\Big)\\
     & \le \exp\Big(-\frac{\tilde\eta^2 k_n^2}{2k_n(2\eta+1)}\Big)\\
     & = \exp\Big(-\frac{\tilde\eta^2}{2(2\eta+1)}k_n \Big).
   \end{align*}
  Now the assertion is obvious.
\end{proof}
\setcounter{section}{7}\setcounter{theorem}{4}
\begin{lemma}
  If the Conditions (US), (US$^*$) and (L) are met and $\log h_n=o(k_n)$, then
   \begin{equation*}
   \sup_{1\le j\le J_n,A\in\AA} E \Delta_{n,j}^2(A) = O\big(q_n+q_n^*+\big(k_n^{-3}\log(k_n/h_n)\big)^{1/2}\big).
   \end{equation*}
\end{lemma}
\setcounter{section}{8}\setcounter{equation}{9}
\begin{proof} 
   For $i,k\in I_{n,r_j}$ with $i\ne k$, the condition $\min(R_{i/n},R_{k/n})>\hat u_{n,r_j}$ is equivalent to $\min(R_{i/n},R_{k/n})>\unjik$ (with $\unjik$ defined in Lemma 7.3). Hence
   \begin{align*}
     & k_n^2 E(\hat S_{n,r_j}^2(A))\\
      & = \sum_{i,k\in I_{n,r_j}, i\ne k} P\{R_{i/n}>\unjik,\Theta_{i/n}\in A,R_{k/n}>\unjik,\Theta_{k/n}\in A\} + k_n E(\hat S_{n,r_j}(A)).
   \end{align*}
   By basically the same arguments as in the proof of (7.8) 
   (using \eqref{eq:orderstatapprox2} instead of \eqref{eq:orderstatapprox1}), one concludes
   \begin{align*}
     \sum_{i,k\in I_{n,r_j}, i\ne k} & P\{R_{i/n}>\unjik,\Theta_{i/n}\in A,R_{k/n}>\unjik,\Theta_{k/n}\in A\}\\
     & = k_n(k_n-1) \big(\Snrj(A)+ O(q_n+q_n^*)\big)^2 + o\big(k_n^{2}  \exp(-ck_n)\big)
   \end{align*}
   and thus, in view of (7.8), 
   \begin{align}
     & E(\hat S_{n,r_j}^2(A))\nonumber \\
      & = (\Snrj(A))^2+ k_n^{-1}\big(E(S_{n,r_j}(A))-(\Snrj(A))^2\big)+O(q_n+q_n^*) + o\big(\exp(-ck_n)\big) \nonumber\\
     & = (\Snrj(A))^2+ k_n^{-1}\Snrj(A)\big(1-\Snrj(A)\big)+O(q_n+q_n^*) + o\big(\exp(-ck_n)\big). \label{eq:T1barbound}
   \end{align}

   Next, we conclude similarly as in the proof of (7.7) 
   that uniformly for all $t\in[0,1]$
   \begin{align*}
     & E(\tilde S_{n,t}^2(A))\\
      & = \sum_{i,k\in I_{n,t},i\ne k} E\Big[N_t^{-2} P(\Theta_{i/n}\in A\mid R_{i/n}>u_{n,t})P(\Theta_{k/n}\in A\mid R_{k/n}>u_{n,t})\Indd{R_{i/n}\wedge R_{k/n}>u_{n,t}}\Big] \\
     & \hspace{1cm}+ \sum_{i\in I_{n,t}} E\Bigg[ N_t^{-2} P(\Theta_{i/n}\in A\mid R_{i/n}>u_{n,t})\Indd{R_{i/n}>u_{n,t}} \Big]\\
     & = \big(\Snt(A)\big)^2 \sum_{i,k\in I_{n,t},i\ne k} E\Big[N_t^{-2} \Indd{R_{i/n}>u_{n,t},R_{k/n}>u_{n,t}} \Big] \\
      & \hspace{1cm} + \Snt(A)\sum_{i\in I_{n,t}} E\Big[N_t^{-2} \Indd{R_{i/n}>u_{n,t}} \Big] + O(q_n)\\
     & = \big(\Snt(A)\big)^2 +  \Snt(A)\big(1-\Snt(A)\big)E\big[N_t^{-1}\big] +O(q_n),
   \end{align*}
   where for notational simplicity we use the convention $1/0:=0$.
   Recall from \eqref{eq:bern}
   that for all $d>0$
   \begin{align*}
     & P\Big\{\Big| \sum_{i\in I_{n,t}} \nonumber \Indd{R_{i/n}>u_{n,t}}-k_n\Big|>d(k_n\log(k_n/h_n))^{1/2}\Big\}\\
     & = O\bigg(\exp\Big(-\frac{d^2k_n\log(k_n/h_n)}{2(k_n+d(k_n\log(k_n/h_n))^{1/2}/3)}\Big) \bigg) \nonumber\\
     & = o\big( (k_n/h_n)^{-d^2/3}\big). 
   \end{align*}
   Hence
   \begin{align*}
     \Big|E\big[ N_t^{-1} -{k_n}^{-1}\big]\Big| & = \bigg|E\Big[ \frac{N_t-k_n}{k_nN_t}\Big]\bigg|\\
     & \le 2k_n^{-2} E|  N_t-k_n| + o\big(k_n^{-2}\big)\\
     & = O\big(\big(k_n^{-3}\log(k_n/h_n)\big)^{1/2}\big),
   \end{align*}
   and we obtain the approximation
   \begin{align}
     E(\tilde S_{n,t}(A))^2 & = \big(\Snt(A)\big)^2 + k_n^{-1}\Snt(A)\big(1-\Snt(A)\big) + O\big(q_n+\big(k_n^{-3}\log(k_n/h_n)\big)^{1/2}\big)  \label{eq:T2barbound}
   \end{align}
   uniformly for all $t\in[0,1]$ and $A\in\AA$.

   Finally, we consider
   \begin{align}
     k_n E\big(\hat S_{n,r_j}(A)\tilde S_{n,r_j}(A)\big)
     & = \sum_{i,k\in I_{n,r_j},i\ne k} E\Big[N_{r_j}^{-1} \Indd{R_{i/n}>\hat u_{n,r_j},\Theta_{i/n}\in A, R_{k/n}>u_{n,r_j}, \Theta_{k/n}\in A}\Big] \nonumber\\
     & \hspace{1cm} + \sum_{i\in I_{n,r_j}} E\Big[N_{r_j}^{-1}  \Indd{R_{i/n}>\hat u_{n,r_j}\vee u_{n,r_j},\Theta_{i/n}\in A}\Big]\nonumber\\
     & =: T_1+T_2. \label{eq:mixbound1}
   \end{align}
   By conditioning on $\FF_{n,i,k}:=\big(\big(\Indd{R_{l/n}>u_{n,r_j}}\big)_{l\in I_{n,r_j}\setminus\{i\}},\Indd{\Theta_{k/n}\in A}\big)$, the term $T_1$ can be bounded below by
   \begin{align*}
     & \sum_{i,k\in I_{n,r_j},i\ne k} E\bigg[ \frac{\Indd{R_{i/n}>\unji,\Theta_{i/n}\in A, R_{k/n}>u_{n,r_j}, \Theta_{k/n}\in A}\Indd{|\unji/u_{n,r_j}-1|\le\eps}}{1+\sum_{l\in I_{n,r_j}\setminus\{i\}} \Indd{R_{l/n}>u_{n,r_j}}}\bigg]\\
     & = \sum_{i,k\in I_{n,r_j},i\ne k}\\
      & \hspace{1cm} E\bigg[ \frac{P\big(R_{i/n}>\unji,\Theta_{i/n}\in A,|\unji/u_{n,r_j}-1|\le\eps \;\big|\; \FF_{n,i,k}\big) \Indd{ R_{k/n}>u_{n,r_j}, \Theta_{k/n}\in A}}{1+\sum_{l\in I_{n,r_j}\setminus\{i\}} \Indd{R_{l/n}>u_{n,r_j}}}\bigg].
   \end{align*}
   Here, using the Conditions (US) and (US$^*$), the conditional probability in the numerator on the right hand side can be approximated as follows:
   \begin{align*}
     \int & P\{R_{i/n}>u,\Theta_{i/n}\in A\} \ind{[(1-\eps)u_{n,r_j},(1+\eps)u_{n,r_j}]}(u)\, P^{\unji\mid \FF_{n,i,k}}(du) \\
     & = \big(\Snrj(A)+O(q_n+q_n^*)\big)P\big(R_{i/n}>\unji, |\unji/u_{n,r_j}-1|\le\eps \;\big|\; \FF_{n,i,k}\big).
   \end{align*}
   Thus, using Lemma \ref{lemma:orderstatapprox}, we obtain
   \begin{align}
     T_1 \ge \big(\Snrj(A)+O(q_n+q_n^*)\big)  \sum_{i,k\in I_{n,r_j},i\ne k} & E\bigg[ \frac{\Indd{R_{i/n}>\hat u_{n,r_j}, R_{k/n}>u_{n,r_j}, \Theta_{k/n}\in A}}{1+\sum_{l\in I_{n,r_j}\setminus\{i\}} \Indd{R_{l/n}>u_{n,r_j}}}\bigg] \nonumber\\
       & + o\big(k_n\exp(-ck_n)\big).
     \label{eq:mixbound2}
   \end{align}
    The sum on the right hand side can be  expressed as
   \begin{align*}
       & \sum_{i,k\in I_{n,r_j},i\ne k}  E\bigg[ \frac{\Indd{R_{i/n}>\hat u_{n,r_j}, R_{k/n}>u_{n,r_j}, \Theta_{k/n}\in A}}{N_{r_j}}\bigg]\\
       & \qquad - \sum_{i,k\in I_{n,r_j},i\ne k} E\bigg[ \frac{\Indd{u_{n,r_j}\ge R_{i/n}>\hat u_{n,r_j}, R_{k/n}>u_{n,r_j}, \Theta_{k/n}\in A}}{N_{r_j}(N_{r_j}+1)}\bigg]\\
       & \ge  E\bigg[  \sum_{i\in I_{n,r_j}}\Ind{R_{i/n}>\hat u_{n,r_j}}\cdot \sum_{k\in I_{n,r_j}}\frac{\Indd{R_{k/n}>u_{n,r_j}, \Theta_{k/n}\in A}}{N_{r_j}}\bigg]\\
       & \qquad - E\bigg[ \sum_{i\in I_{n,r_j}}\frac{\Indd{R_{i/n}>\hat u_{n,r_j}\vee u_{n,r_j}, \Theta_{i/n}\in A}}{N_{r_j}}\bigg] - E\bigg[ \frac{(k_n-N_{r_j})^+}{N_{r_j}+1}\bigg]\\
       & = k_n E\big(\tilde S_{n,r_j}(A)\big)-T_2+O\big((\log(k_n/h_n)/k_n)^{1/2}\big),
   \end{align*}
   where the last step follows from Lemma \ref{lemma:bernstein}.
   Combine this with \eqref{eq:mixbound1}, \eqref{eq:mixbound2} and (7.7) 
    to conclude
   \begin{align*}
     & k_n  E\big(\hat S_{n,r_j}(A)\tilde S_{n,r_j}(A)\big)\\
      & \ge \Snrj(A)k_n E\big(\tilde S_{n,r_j}(A)\big) + (1-\Snrj(A))T_2 +O\big(k_n(q_n+q_n^*)+(\log(k_n/h_n)/k_n)^{1/2}\big)\\
      & = k_n (\Snrj(A))^2+ (1-\Snrj(A))T_2 +O\big(k_n(q_n+q_n^*)+(\log(k_n/h_n)/k_n)^{1/2}\big).
   \end{align*}

   Similarly as before, by conditioning on $(\Indd{R_{l/n}>u_{n,r_j}})_{l\in I_{n,r_j}\setminus\{i\}}$, one obtains
   \begin{align*}
     T_2 & = \sum_{i\in I_{n,r_j}} E\bigg[\frac{\Indd{R_{i/n}>\unji\vee u_{n,r_j}, \Theta_{i/n}\in A,|\unji/u_{n,r_j}-1|\le\eps}}{1+\sum_{l\in I_{n,r_j}\setminus\{i\}}\Indd{R_{l/n}>u_{n,r_j}}}\bigg] + o\big(\exp(-ck_n)\big)\\
     & = \big(\Snrj(A)+O(q_n+q_n^*)\big)E\bigg[\frac{\sum_{i\in I_{n,r_j}}\Indd{R_{i/n}>\unji\vee u_{n,r_j}}}{N_{r_j}}\bigg] + o\big(\exp(-ck_n)\big)\\
     & = \Snrj(A)\Big(1-E\Big[\frac{(N_{r_j}-k_n)^+}{N_{r_j}}\Big]\Big) +O(q_n+q_n^*) + o\big(\exp(-ck_n)\big)\\
     & = \Snrj(A) + O\big((\log(k_n/h_n)/k_n)^{1/2}\big) +O(q_n+q_n^*).
   \end{align*}
   To sum up, we have shown that
   \begin{align}
      &E\big(\hat S_{n,r_j}(A)\tilde S_{n,r_j}(A)\big) \nonumber\\
       & \ge (\Snrj(A))^2 + k_n^{-1} \Snrj(A)(1-\Snrj(A)) +  O\big(q_n+q_n^*+(\log(k_n/h_n)/k_n^3)^{1/2}\big).
      \label{eq:T12barbound}
   \end{align}
   Therefore, in view of \eqref{eq:T1barbound}, \eqref{eq:T2barbound} and \eqref{eq:T12barbound},
   \begin{align*}
   E\Delta_{n,j}^2(A) & = E\big(\hat S_{n,r_j}(A)\big)^2 + E\big(\tilde S_{n,r_j}(A)\big)^2 - 2 E\big(\hat S_{n,r_j}(A)\tilde S_{n,r_j}(A)\big)\\
   & = O\big(q_n+q_n^*+\big(k_n^{-3}\log(k_n/h_n)\big)^{1/2}\big).
   \end{align*}
 \end{proof}

\section{Further simulation results}
We present results of simulations for further parameter settings of the model classes described in Section 4 of the accompanying paper.

Table \ref{table:critval_large} lists further critical values for both tests in dimension 2, obtained from 10,000 simulations of a Brownian pillow.
\setcounter{table}{4}
\begin{table}[b]
 \caption{Asymptotic critical values of tests based on (3.1) and (3.2) in dimension $d=2$ \label{table:critval_large}}
  \centerline{\begin{tabular}{l|cccccc}
    & \multicolumn{6}{c}{nominal size} \\
    & 0.005 & 0.01 & 0.025 & 0.05 & 0.10 & 0.20\\ \hline
    $T_n^{(KS)}$ & 0.9660 & 0.9222& 0.8634 & 0.8135 & 0.7626 & 0.6990\\
    $T_n^{(CM)}$ & 0.3021 & 0.2683 & 0.2242 &  0.1939 & 0.1621 & 0.1289
  \end{tabular}}
\end{table}

Table \ref{table:empsizes2d_part2} contains the empirical probability of a type 1 error of both tests with nominal size $0.05$ in dimension $d=2$ for different combinations of  copulas and marginal distributions. (See the main paper for explanations of the models and parameters.) Again we observe that both tests keep their nominal levels well. The Kolmogorov-Smirnov type test is quite conservative, in particular for large block lengths.
\begin{table}
  \caption{Empirical probability of a type 1 error of test based on $T_n^{(KS)}$ (upper value) and $T_n^{CM}$ (lower value) for bivariate observations and nominal level 0.05. \label{table:empsizes2d_part2}}
  \centerline{\small
  \begin{tabular}{r|rrr|rrr|rrr}
    $b$ & \multicolumn{3}{c|}{50} & \multicolumn{3}{c|}{100} & \multicolumn{3}{c}{200}\\
    $k$ & 5 & 10 & 20 & 5 & 10 & 20 & 5 & 10 & 20 \\ \hline
    Gumbel copula $\lambda=2$,  & 0.03	&0.04	&0.04	&	0.02	&0.03	& 0.03	&	0.02	&0.01	&0.02\\
    Fr\'{e}chet $\alpha=4$ & 0.04&	0.06&	0.05	&	0.04&	0.05&	0.06&		0.04	&0.04	&0.05\\ \hline
    Gumbel copula $\lambda=6$, & 0.04	&0.05&	0.04	&	0.02&	0.03&	0.03	&	0.01&	0.02&	0.02\\
    Fr\'{e}chet $\alpha=4$ & 0.04&	0.07&	0.05	&	0.04	&0.06	&0.05	&	0.04	&0.05	&0.05\\ \hline
    Gumbel copula $\lambda=6$,  & 0.04	&0.04	&0.04	&	0.02	&0.01	& 0.03	&	0.01&	0.02&	0.02\\
    Fr\'{e}chet $\alpha=4$, with sine-factor & 0.05	& 0.06	& 0.07	&	0.04	& 0.03	&0.04	&	0.03	&0.04	&0.04\\ \hline
    $t_2$-copula, $\rho=0.5$& 0.04	&0.03&	0.03	&	0.03&	0.03&	0.03&		0.02	&0.02	&0.02\\
    Fr\'{e}chet $\alpha=4$  & 0.05	&0.04	&0.05	&	0.04	&0.05	&0.05	& 	0.04	&0.05	&0.05\\ \hline
    $t_2$-copula,  $\rho=0.9$ & 0.04&	0.04&	0.04	&	0.03&	0.03&	0.03	&	0.02	&0.02	&0.02 \\
    Fr\'{e}chet $\alpha=4$& 0.06	&0.06&	0.05	&	0.05&	0.05&	0.06	&	0.05&	0.04	&0.05
   \end{tabular}}
\end{table}

Next we examine the power of our tests against different types of changing extreme value dependence. Since we have seen that both tests do not perform well if the block length is large, here we only consider block lengths $b\in\{50,100\}$. We define the alternative models under consideration by specifying the copula as a function of the time parameter $t$ and the marginal distributions (constant over time):
\begin{itemize}
  \item {\bf Model I}: Gumbel copula with parameter $\lambda=2$ equal to 0 for $t\in[0, 1/2]$ and equal to $\lambda_1\in\{2.5, 3,\ldots, 6\}$ for $t\in(1/2,1]$; Fr\'{e}chet marginals with tail index $\alpha=2$.
  \item {\bf Model II}: Gumbel copula with $\lambda=2$ for $t\in[0,1/3]\cup(2/3,1]$ and $\lambda^*\in\{3, 3.5, \ldots, 7\}$ for $t\in(1/3,2/3]$; Fr\'{e}chet marginals with $\alpha=2$.
  \item {\bf Model III}: $t_2$-copula with correlation parameter $\rho=0$ for $t\in [0,1/4]\cup(3/4,1]$ and $\rho=\rho^*\in\{0.5, 0.6, 0.7, 0.75, 0.8, 0.9, 0.95\}$ for $t\in (1/4,3/4]$; Fr\'{e}chet marginals with $\alpha=4$.
\end{itemize}
In all three models, the parameter that controls the strength of dependence in the tail changes discontinuously. While in Model I there is just one change in the middle of the time intervals, in the other two models there are two structural breaks and the parameter is the same in the beginning and in the end of the time interval.

The empirical power of both tests for different numbers $k\in\{5,10,20\}$ of observations with largest Euclidean norm in each block is shown in Figures \ref{fig:Model_I}--\ref{fig:Model_III}. A comparison of Figure \ref{fig:Model_I} with Figure 1 of the accompanying paper shows that a clear structural break is more easily detected than a smooth change of the dependence structure. This is of course not surprising, because in Model I with a structural break, for all $t\le 1/2$ the spectral measure strongly deviates from the average spectral measure over the whole interval. In contrast, in the model ``Gumbel linear'' examined in Figure 1 the difference between the average of the spectral measures up to time $t$ and the average spectral measure over the whole interval becomes smaller from the start as $t$ increases.

\setcounter{figure}{3}
\begin{figure}[tb]
\includegraphics[width=13cm]{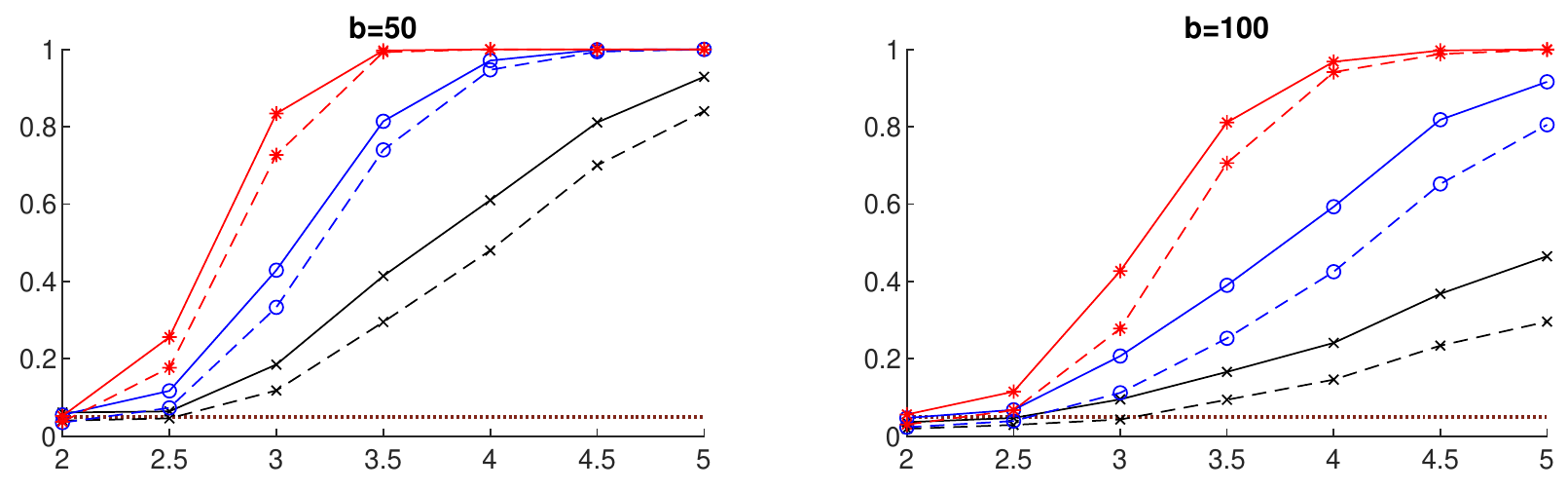}
\caption{Empirical power of tests based on $T_n^{(CM)}$ (solid line) and $T_n^{(KS)}$ (dashed line) in Model I versus $\lambda_1$  for $k=5$ (black $\times$), $k=10$ (blue $\circ$) and $k=20$ (red $*$) largest observations in each block and different block lengths $b$; the nominal size is indicated by the brown dotted horizontal line.
\label{fig:Model_I}}
\end{figure}

The plots in Figure \ref{fig:Model_II} and Figure \ref{fig:Model_III} reveal that both tests are substantially less powerful against changes in the dependence structure that only occur in the middle of the time interval. Note that Model III equals Model ``$t_2$ jump'' examined in the main paper except for the time interval in which the correlation parameter is increased, which has been moved from $(1/2,1]$ to the mid interval $(1/4,3/4]$. Due to the integration over time, though, this shift has a profound influence on the difference between the integrated spectral measure $IS_t$ and the weighted average spectral measure $t\cdot IS_1$, which for most $t$ is much smaller in Model III. It is thus not surprising that consequently both tests are less capable to detect the deviation from the null hypothesis. It is interesting to note, though, that the Kolmogorov-Smirnov type test now performs better than the Cram\'{e}r-von Mises type test, while in the models with a monotone trend considered in the main paper, the Cram\'{e}r-von Mises type test outperforms the Kolmogorov-Smirnov type test. This may perhaps be explained by the opportunity that, for instance in Model III, the latter picks up the quite pronounced difference between $IS_t$ and $t\cdot IS_1$ for $t$ near $1/4$, while the former averages the differences over the whole time interval and for most values of $t$ these differences are quite small.

\begin{figure}[tb]
\includegraphics[width=13cm]{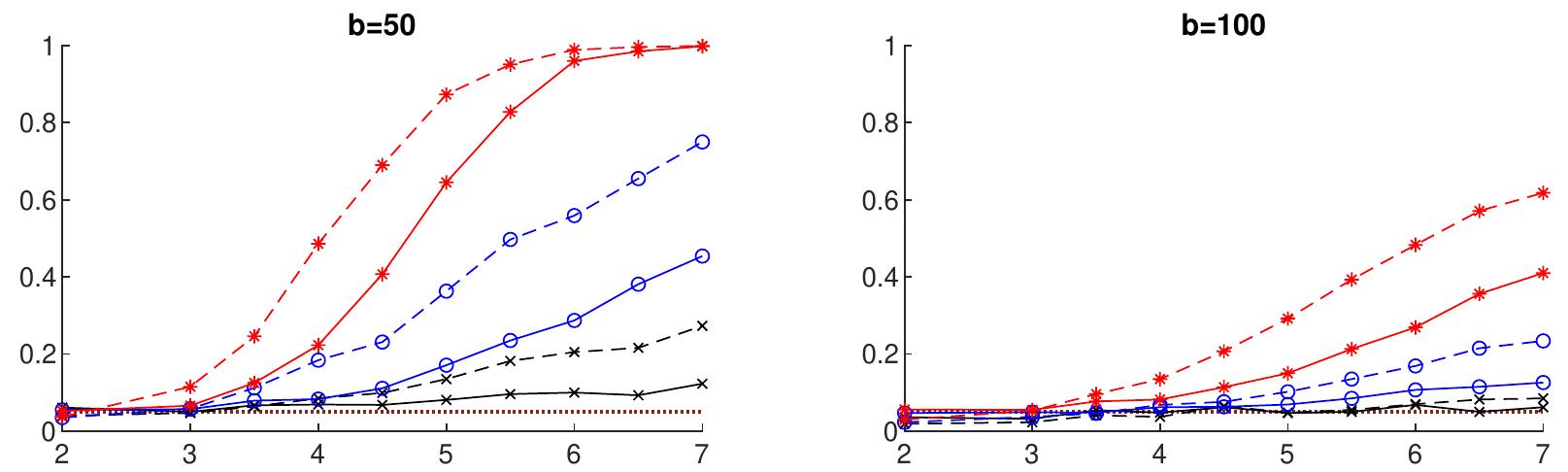}
\caption{Empirical power of tests based on $T_n^{(CM)}$ (solid line) and $T_n^{(KS)}$ (dashed line) in Model II versus $\lambda^*$ for $k=5$ (black $\times$), $k=10$ (blue $\circ$) and $k=20$ (red $*$) largest observations in each block and different block lengths $b$; the nominal size is indicated by the brown dotted horizontal line.
\label{fig:Model_II}}
\end{figure}

\begin{figure}[tb]
\includegraphics[width=13cm]{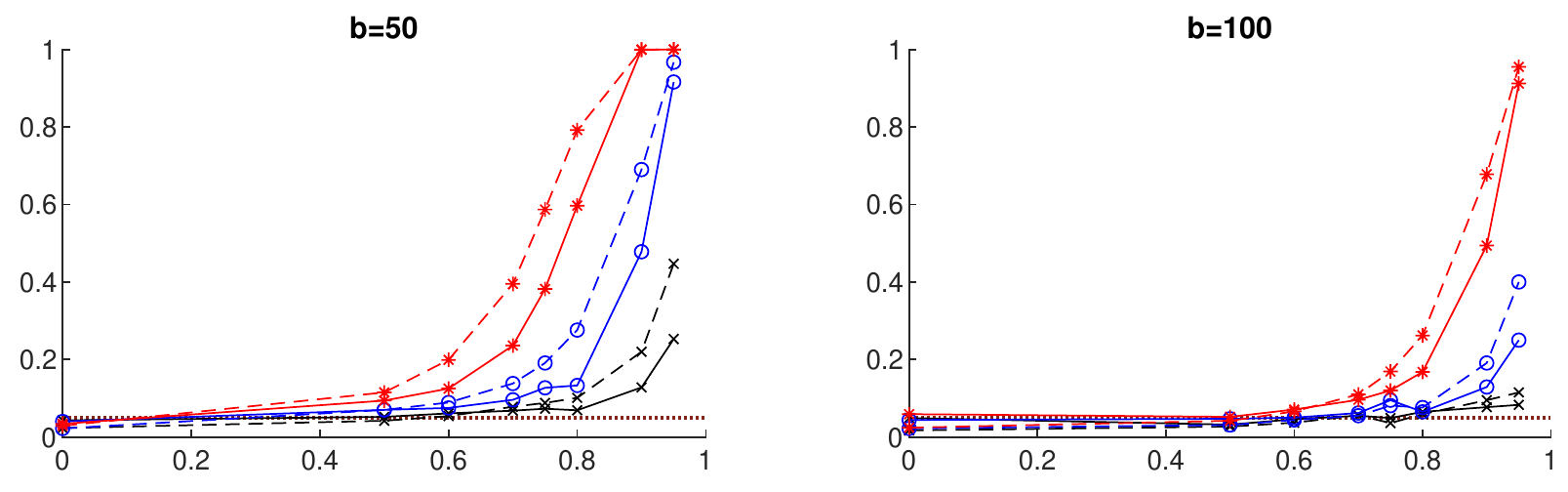}
\caption{Empirical power of tests based on $T_n^{(CM)}$ (solid line) and $T_n^{(KS)}$ (dashed line) in Model III versus $\rho^*$  for $k=5$ (black $\times$), $k=10$ (blue $\circ$) and $k=20$ (red $*$) largest observations in each block and different block lengths $b$; the nominal size is indicated by the brown dotted horizontal line.
\label{fig:Model_III}}
\end{figure}

One may try to counter this weakness of the proposed tests by applying them separately to different subintervals of $[0,1]$, or by even considering sequential versions that are based on the double indexed stochastic processes $(IS_s(A)-(s/t)IS_t(A))_{0\le s\le t\le 1}$, but a detailed analysis of such tests is beyond the scope of this paper.

Finally, we briefly discuss the power function of our tests for different models of three-dimensional observations. We consider the copula described in Models I--III, but in all cases the margins are Fr\'{e}chet distributed with tail index $\alpha=4$. In addition, we investigate a model analogous to Model III, but with the intervals swapped when the dependence is low and high, respectively:
\begin{itemize}
    \item {\bf Model III inv}: $t_2$-copula with correlation parameter $\rho^*\in\{0.5, 0.75, 0.9\}$ for $t\in [0,1/4]\cup(3/4,1]$ and $\rho=0$ for $t\in (1/4,3/4]$; Fr\'{e}chet marginals with tail index $\alpha=4$.
\end{itemize}
The resulting empirical powers are given in Table \ref{table:emppower3d_supp}.

\begin{table}[bt]
  \caption{Empirical power of tests based on $T_n^{(KS)}$ (upper value) and $T_n^{(CM)}$ (lower value), for three-dimensional observations and nominal level 0.05.\label{table:emppower3d_supp}}
  \centerline{\small
  \begin{tabular}{r|rrr|rrr}
    $b$ & \multicolumn{3}{c|}{50} & \multicolumn{3}{c}{100}\\
    $k$ & 5 & 10 & 20 & 5 & 10 & 20  \\ \hline
        Model I, $\lambda_1=2.5$ & 0.16	&0.28	&0.53	&	0.10	&0.16&	0.26\\
     & 0.15	&0.26&	0.52	&	0.09&	0.14&	0.25\\[1ex]
    $\lambda_1=3$ & 0.43&	0.76&	0.99	&	0.22&	0.43&	0.76\\
     & 0.41	&0.76	&0.99	&	0.20	&0.38	&0.74\\[1ex]
    $\lambda_1=3.5$ & 0.76	&0.98&	1.00	&	0.38	&0.74	&0.99\\
    &  0.74	&0.97&	1.00	&	0.37	&0.70	&0.98\\ \hline
    Model II, $\lambda^*=3$ & 0.12	&0.17	&0.39	&	0.08&	0.14&	0.18\\
      & 0.06&	0.08&	0.10	&	0.05&	0.07	&0.08\\[1ex]
    $\lambda^*=4$ & 0.25&	0.57	&0.96	&	0.16&	0.26&	0.57\\
     & 0.09	& 0.16&	0.60	&	0.09	&0.09	&0.15\\[1ex]
    $\lambda^*=5$ & 0.47&	0.90&	1.00	&	0.20	&0.46	&0.89\\
    & 0.12&	0.43&	0.98	&	0.08&	0.13	&0.41 \\[1ex]
    $\lambda^*=6$ &  0.66&	0.99&	1.00	&	0.26	&0.59	&0.98  \\
     & 0.20	&0.73	&1.00	&	0.09&	0.19	&0.71\\    \hline
    Model III, $\rho^*=0.5$ & 0.10	&0.16	&0.33	&	0.09&	0.10&	0.15\\
    & 0.06	&0.08&	0.11	&	0.07	&0.07	&0.08\\[1ex]
    $\rho^*=0.75$ &0.27	&0.70&	0.99	&	0.14	&0.28&	0.69\\
    & 0.09	&0.27&	0.88	&	0.07	&0.09	&0.25\\[1ex]
    $\rho^*=0.9$ & 0.77	&1.00	&1.00	&	0.33	&0.79	&1.00\\
    & 0.36	&0.94&	1.00	&	0.12	&0.33	&0.94\\ \hline
    Model III inv, $\rho^* = 0.5$ & 0.09&	0.16&	0.31	&	0.08&	0.08&	0.15\\
        & 0.07&	0.10&	0.10&		0.06&	0.05&	0.07\\[1ex]
        $\rho^*=0.75$ & 0.22&	0.60&	0.98&		0.11&	0.22&	0.60\\
        & 0.09&	0.22&	0.84	&	0.06&	0.08&	0.21\\[1ex]
       $\rho^*=0.9$ & 0.71	& 1.00&	1.00	&	0.22	&0.69&	1.00\\
    & 0.27	&0.94	&1.00	&	0.06&	0.24	&0.94
  \end{tabular}}
\end{table}

While for Model I both tests show a very similar behavior, as in dimension 2 the Kolomogorov-Smirnov type test has a much greater power than the Cram\'{e}r-von Mises type test against alternatives where the dependence structure changes in the middle of the interval and returns to the previous state at the end. A comparison of Model III and Model III inv reveals that it hardly matters whether the dependence is stronger or weaker in the middle of the interval.

So far we have examined the performance of our tests if the nominal level equals 0.05. The left plot of Figure \ref{fig:pvalues} shows the empirical quantile function of the estimated $p$-values if all observations have a Gumbel-copula with $\lambda=2$ and the marginal distributions are Fr\'{e}chet(4) with a scale $c(t)=1+\sin(2\pi t)/2$ depending on $t\in[0,1]$, both for the Kolomogorov-Smirnov type test statistic and the Cram\'{e}r-von Mises type test statistic based on blocks of length $b=50$ and $k=20$ largest observations in each block. Obviously, both empirical distributions of the estimated $p$-values are very close to the uniform distributions, which means that for all nominal sizes the corresponding tests will approximately keep the level.

The right plot of Figure \ref{fig:pvalues} shows the analogous plots in a situation where the null hypothesis is violated, because the correlation parameter of the $t_2$-copula changes linearly from 0 to 0.5. The graphs for both tests clearly lie below the main diagonal, meaning that for all nominal sizes of the tests, the probability of a rejection of the null is much larger than the level. The superiority of the Cram\'{e}r-von Mises type test  for this setting can be seen from the fact that the pertaining graph is lower than that of the Kolmogorov-Smirnov type test.

\begin{figure}[tb]
\includegraphics[width=13cm]{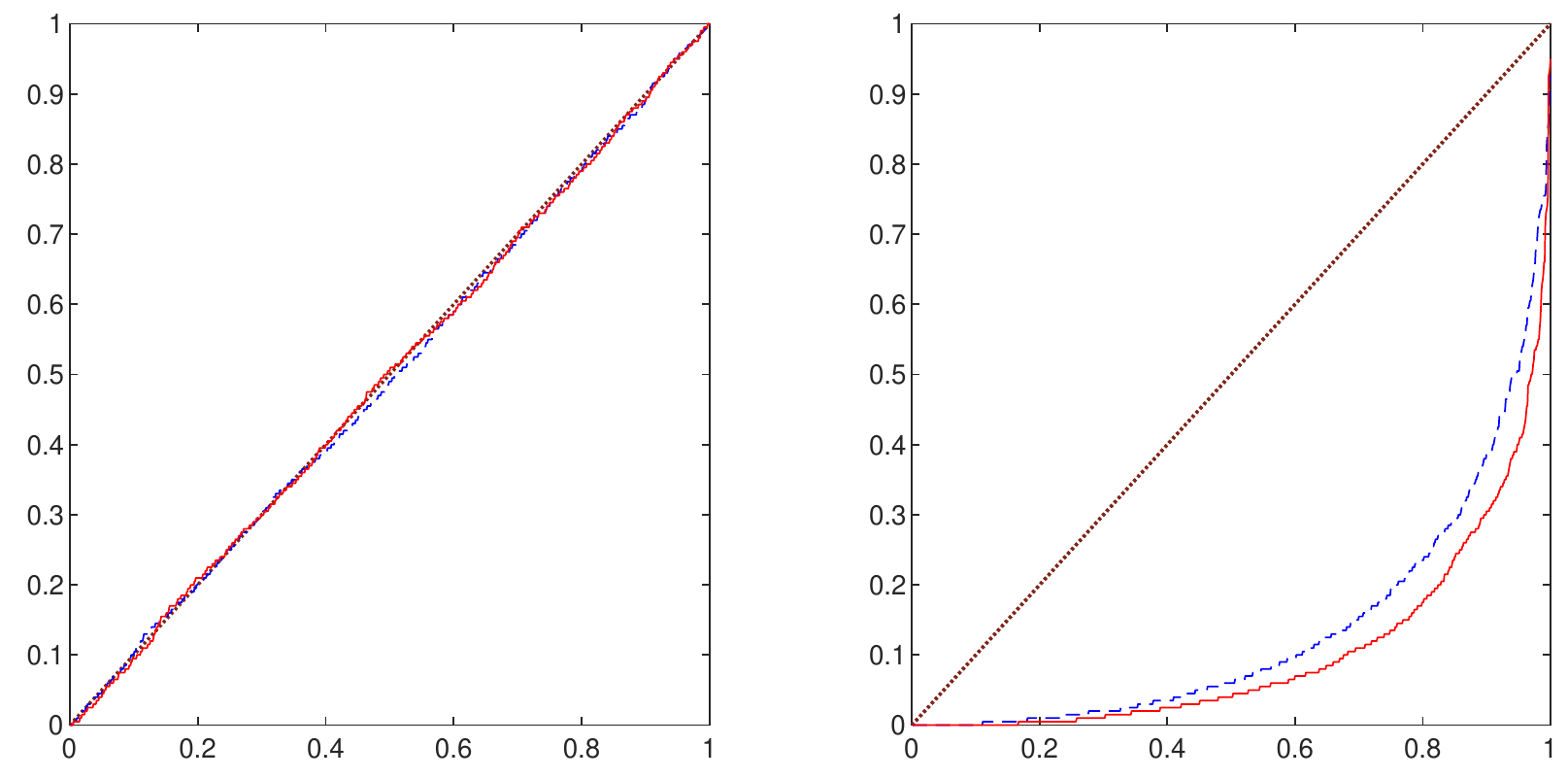}
\caption{Empirical quantile function of the simulated $p$-values of the tests based on $T_n^{(CM)}$ (solid red line) and $T_n^{(KS)}$ (dashed blue line) for three-dimensional observations with Gumbel-copula and Fr\'{e}chet(4) margins with sine-factor (left) and with $t_2$-copula when the correlation parameter  $\rho$ increases linearly from 0 to 0.5 and Fr\'{e}chet(4) margins (right), respectively, for $b=50$ and $k=20$; the main diagonal is indicated by the brown dotted line.
\label{fig:pvalues}}
\end{figure}

\section{Verification of the conditions for a semiparametric model}

Here we check the conditions used in our main result for a concrete bivariate model.
Assume that the $j$th marginal distribution of $X_t$ has a Fr\'{e}chet distribution with cdf $$ F_{t,j}(x_j)= \exp\Big(-\Big(\frac{x_j}{c_{t,j}}\Big)^{-\alpha_t}\Big), \quad x>0, $$
for some Lipschitz continuous functions $t\mapsto c_{t,j}>0$ and $t\mapsto \alpha_t>0$. Moreover, we assume that $X-t$ has a Gumbel copula mit parameter $\lambda_t>1$, so that the cdf of $X_t$ is given by
$$ F_t(x_1,x_2) = \exp\bigg(-\Big(\Big(\frac{x_1}{c_{t,1}}\Big)^{-\lambda_t\alpha_t} + \Big(\frac{x_2}{c_{t,2}}\Big)^{-\lambda_t\alpha_t}\Big)^{1/\lambda_t}\bigg), \quad x_1,x_2>0,
$$
where the function $t\mapsto\lambda_t>1$ is Lipschitz continuous, too.
The bivariate density at time $t\in[0,1]$ is then given by
\begin{align*}
 f_t(x_1,x_2)  = & (\lambda_t-1)\Big(\Big(\frac{x_1}{c_{t,1}}\Big)^{-\lambda_t\alpha_t} + \Big(\frac{x_2}{c_{t,2}}\Big)^{-\lambda_t\alpha_t}\Big)^{1/\lambda_t-2}\\
   & \times \bigg[ 1+\frac 1{\lambda_t-1} \Big(\Big(\frac{x_1}{c_{t,1}}\Big)^{-\lambda_t\alpha_t} + \Big(\frac{x_2}{c_{t,2}}\Big)^{-\lambda_t\alpha_t}\Big)^{1/\lambda_t}\bigg]\\
   & \times \frac{\alpha_t^2}{c_{t,1} c_{t,2}} \Big(\frac{x_1 x_2}{c_{t,1}c_{t,2}}\Big)^{-\lambda_t\alpha_t-1}F_t(x_1,x_2).
\end{align*}
By the continuity assumptions, we have $\underline\alpha:=\min_t \alpha_t>0$, $\bar\alpha:=\max_t \alpha_t<\infty$, \linebreak $\underline\lambda_t :=\min_t \lambda_t>1$, $\bar\lambda := \max_t \lambda_t<\infty$, $\underline c:=\min_t \min(c_{t,1},c_{t,2}, c_{t,1}/c_{t,2})>0$ and $\bar c :=\max_t \max(c_{t,1},c_{t,2}, c_{t,1}/c_{t,2})<\infty$, where the minima and maxima are taken over all $t\in[0,1]$.

We use the sum norm to define the spectral measure and consider the family $\AA:=\big(\{(x,1-x)|x\le s\}\big)_{s\in [0,1]}$ of test subsets of the unit sphere, which form a VC class with VC-index 2. Since the total boundedness can be proved by the arguments sketched in Subsection 2.1, condition (A$^*$) is fulfilled.

To check the conditions (US), (US$^*$) and (B), we must calculate probabilities of the form
$$ P\{X_{t,1}+X_{t,2}>u, X_{t,1}\le s(X_{t,1}+X_{t,2})\}
   = P\Big\{ \frac{u}{X_{t,2}}<\frac 1{1-s}, \frac{u}{x_2}-1<\frac{X_{t,1}}{X_{t,2}}\le \frac{s}{1-s}\Big\}.
$$
In the following calculation, we will often omit the subindex $t$ in order to ease the notation.

Using the change of variables
\begin{align*}
  w & = \frac{u}{x_2},\\
  z & =  \Big(\Big(\frac{x_1}{c_{1}}\Big)^{-\lambda\alpha} + \Big(\frac{x_2}{c_{2}}\Big)^{-\lambda\alpha}\Big)^{1/\lambda}
   = \Big( \frac{c_2w}{u}\Big)^\alpha \Big( 1+\Big( c\frac{x_1}{x_2}\Big)^{-\lambda\alpha}\Big)^{1/\lambda}
\end{align*}
with $c:=c_2/c_1$, one obtains
\begin{align}
  P\Big\{ & \frac{u}{X_{t,2}}<\frac 1{1-s}, \frac{u}{x_2}-1<\frac{X_{t,1}}{X_{t,2}}\le \frac{s}{1-s}\Big\}  \nonumber\\
  & =\int_0^\infty \int_0^\infty f(x_1,x_2) \Indd{u/x_2<1/(1-s),u/x_2-1<x_1/x_2\le s/(1-s)}\, dx_1\, dx_2 \nonumber\\
  & = \int_0^\infty \int_0^\infty (\lambda-1) z^{-\lambda} \Big[1+\frac{1}{\lambda-1}z\Big] \frac{\alpha}{c_2u} \Big(\frac{u}{c_2 w}\Big)^{-\lambda\alpha-1} \Big(\frac{u}{w}\Big)^2 \e^{-z} \indd{(0,1/(1-s))}(w) \nonumber\\
  & \hspace{3cm}\indd{[(c_2w/u)^\alpha (1+r^{-\lambda\alpha})^{1/\lambda}, (c_2w/u)^\alpha (1+(c(w-1)^+)^{-\lambda\alpha})^{1/\lambda})}(z) \, dz\, dw, \nonumber\\
  & = \alpha  \Big(\frac{u}{c_2}\Big)^{-\lambda\alpha}\int_0^{1/(1-s)} \int_{(c_2w/u)^\alpha (1+r^{-\lambda\alpha})^{1/\lambda}}^{(c_2w/u)^\alpha (1+(c(w-1)^+)^{-\lambda\alpha})^{1/\lambda}} \big[(\lambda-1)z^{-\lambda}+z^{1-\lambda}\big]\e^{-z}\, dz \nonumber\\
  & \hspace{11cm} w^{\lambda\alpha-1}\, dw  \label{eq:probabint}
\end{align}
where $r:= cs/(1-s)$, and $(1+(c(w-1)^+)^{-\lambda\alpha})^{1/\lambda}$ is interpreted as $\infty$ if $w\le 1$. In what follows, we split the double integral into several parts, which are then approximated or bounded separately.

Choose some $w_0\in (2,u/(2c_2))$ (which may depend on $u$). Then
\begin{align*}
 \bigg| & P\Big\{  \frac{u}{X_{t,2}}<\frac 1{1-s}, \frac{u}{x_2}-1<\frac{X_{t,1}}{X_{t,2}}\le \frac{s}{1-s}\Big\}\\
   &\quad - \alpha  \Big(\frac{u}{c_2}\Big)^{-\lambda\alpha}\int_0^{1/(1-s)} \int_{(c_2w/u)^\alpha (1+r^{-\lambda\alpha})^{1/\lambda}}^{(c_2w/u)^\alpha (1+(c(w-1)^+)^{-\lambda\alpha})^{1/\lambda}}(\lambda-1)z^{-\lambda}\, dz \, w^{\lambda\alpha-1}\, dw\bigg|\\
  & \le  \alpha  \Big(\frac{u}{c_2}\Big)^{-\lambda\alpha}\int_{w_0}^{\infty} \int_{(c_2w/u)^\alpha}^{(c_2w/u)^\alpha (1+(c(w-1))^{-\lambda\alpha})^{1/\lambda}} z^{1-\lambda}\e^{-z}\, dz  \, w^{\lambda\alpha-1}\, dw\\
  & \quad + \alpha  \Big(\frac{u}{c_2}\Big)^{-\lambda\alpha}\int_{w_0}^{\infty} \int_{(c_2w/u)^\alpha }^{(c_2w/u)^\alpha (1+(c(w-1))^{-\lambda\alpha})^{1/\lambda}} (\lambda-1)z^{-\lambda}\, dz  \, w^{\lambda\alpha-1}\, dw\\
  & \quad + \alpha  \Big(\frac{u}{c_2}\Big)^{-\lambda\alpha}\int_0^{w_0} \int_{(c_2w/u)^\alpha (1+r^{-\lambda\alpha})^{1/\lambda}}^{(c_2w/u)^\alpha (1+(c(w-1)^+)^{-\lambda\alpha})^{1/\lambda}} (\lambda-1)z^{-\lambda}(1-\e^{-z})\, dz  \, w^{\lambda\alpha-1}\, dw\\
   & \quad +  \alpha  \Big(\frac{u}{c_2}\Big)^{-\lambda\alpha}\int_0^{w_0} \int_{(c_2w/u)^\alpha }^\infty z^{1-\lambda}\e^{-z}\, dz  \, w^{\lambda\alpha-1}\, dw\\
   & =: D_1+D_2+D_3+D_4.
  \end{align*}

    Note that (due to a Taylor expansion of $x\mapsto x^{1/\lambda}$ at 1) the length of the interval over which the inner integral of $D_1$ is taken is bounded by
$$ \frac 1\lambda \big((c(w-1))^{-\lambda\alpha}-1\big) \Big(\frac{c_2w}{u}\Big)^\alpha \le \frac{(2/c)^{\lambda\alpha}}{\lambda} w^{-\lambda\alpha} \Big(\frac{c_2w}{u}\Big)^\alpha
$$
for $w\ge 2$. Therefore, using the change of variables $v=(c_2w/u)^\alpha$, we obtain
\begin{align}
  D_1
  & \le \frac{\alpha}{\lambda}  \Big(\frac{cu}{2c_2}\Big)^{-\lambda\alpha}\int_{w_0}^\infty
    \Big(\frac{c_2w}{u}\Big)^{(1-\lambda)\alpha}
    \Big(\frac{c_2}{u}\Big)^\alpha w^{\alpha-1} \exp\Big(-\Big(\frac{c_2w}{u}\Big)^\alpha\Big) \, dw \nonumber\\
  &  = \frac{1}{\lambda} \Big(\frac{u}{2c_1}\Big)^{-\lambda\alpha}\int_{(c_2w_0/u)^\alpha}^\infty  v^{1-\lambda}\e^{-v}\, dv \nonumber\\
  & \le \frac{1}{\lambda} \Big(\frac{u}{2c_1}\Big)^{-\lambda\alpha} \bigg[ \int_{(c_2w_0/u)^\alpha}^1  v^{1-\lambda}\, dv \label{eq:int1} +\int_{1}^\infty  \e^{-v}\, dv\bigg] \nonumber\\
  & = \frac{1}{\lambda} \Big(\frac{u}{2c_1}\Big)^{-\lambda\alpha} \Big[ \alpha h_{(2-\lambda)\alpha}\Big(\frac{c_2w_0}{u}\Big)+1\Big],
\end{align}
with
\begin{equation}
  h_\tau(x) := \int_x^1 z^{\tau-1}\, dz
    = \begin{cases}
        (1-x^{\tau})/\tau, & \tau\ne 0,\\
        -\log x, & \tau=0.
      \end{cases}
\end{equation}

Similarly, we may bound $D_2$:
\begin{align}
  D_2 & = \alpha \Big(\frac{u}{c_2}\Big)^{-\lambda\alpha}\int_{w_0}^\infty
    \Big(\frac{c_2w}{u}\Big)^{(1-\lambda)\alpha} \Big[1- (1+(c(w-1))^{-\lambda\alpha})^{1/\lambda-1}\Big] w^{\lambda\alpha-1} \, dw \nonumber\\
    & \le  \alpha\Big(1-\frac 1{\lambda}\Big)
    \Big(\frac{cu}{2c_2}\Big)^{-\lambda\alpha} \Big(\frac{c_2}{u}\Big)^{(1-\lambda)\alpha}  \int_{w_0}^\infty
    w^{(1-\lambda)\alpha-1} \, dw \nonumber\\
    & = \frac{c_2^\alpha}{\lambda} \Big(\frac{2}{c}\Big)^{\lambda\alpha}  w_0^{(1-\lambda)\alpha} u^{-\alpha}.
\end{align}

To bound $D_3$, note that $0\le 1-\e^{-z}\le \min(z,1)$ for all $z\ge 0$  and that $(c_2w/u)^\alpha (1+r^{-\lambda\alpha})^{1/\lambda}<1$ holds if and only if $w<(1+r^{-\lambda\alpha})^{-1/(\lambda\alpha)}u/c_2=:w_{r,u}$.
Thus direct calculations show that
\begin{align}
D_3
& \le \alpha  \Big(\frac{u}{c_2}\Big)^{-\lambda\alpha}\int_0^{w_0} \int_1^\infty  (\lambda-1)z^{-\lambda}\, dz w^{\lambda\alpha-1}\, dw  \nonumber\\
& \hspace{1cm} +  \alpha (\lambda-1) \Big(\frac{u}{c_2}\Big)^{-\lambda\alpha}\int_0^{w_{r,u}\wedge w_0} \int_{(c_2w/u)^\alpha (1+r^{-\lambda\alpha})^{1/\lambda}}^1 z^{1-\lambda}\, dz\, w^{\lambda\alpha-1}\, dw   \nonumber\\
& = \frac{1}{\lambda} \Big(\frac{c_2w_0}{u}\Big)^{\lambda\alpha}  \nonumber\\
& \hspace{1cm} + \frac{\lambda-1}2 \Big(\frac{c_2(w_{r,u}\wedge w_0)}{u}\Big)^{\lambda\alpha} \Big[ h_{2-\lambda}\Big( \Big(\frac{c_2(w_{r,u}\wedge w_0)}{u}\Big)^{\alpha}(1+r^{-\lambda\alpha})^{1/\lambda}\Big)+\frac{1}{\lambda}\Big] \nonumber\\
& \le \Big(\frac{c_2w_0}{u}\Big)^{\lambda\alpha} \Big[\frac{\alpha(\lambda-1)}{2} h_{(2-\lambda)\alpha}\Big(\frac{c_2w_0}{u}\Big) +\frac{\lambda+1}{2\lambda}\Big],
\end{align}
where in the last step we have used that $(c_2w_{r,u}/u)^\alpha(1+r^{-\lambda\alpha})^{1/\lambda}=1$ and that the function $h_{2-\lambda}$ is decreasing.

To bound $D_4$, we again split the inner integral into two parts at the point $z=1$:
\begin{align}
  D_4
  &\le \alpha  \Big(\frac{u}{c_2}\Big)^{-\lambda\alpha} \int_0^{w_0}
     \int_{(c_2w/u)^\alpha}^{1} z^{1-\lambda}\, dz \, w^{\lambda\alpha-1}\, dw +
     \alpha  \Big(\frac{u}{c_2}\Big)^{-\lambda\alpha} \int_0^{w_0} w^{\lambda\alpha-1}\, dw   \nonumber\\
  & = \alpha  \Big(\frac{u}{c_2}\Big)^{-\lambda\alpha} \int_0^{w_0} h_{2-\lambda}\bigg(\Big(\frac{c_2w}{u}\Big)^\alpha\bigg) w^{\lambda\alpha-1}\, dw + \frac{1}{\lambda} \Big(\frac{u}{c_2w_0}\Big)^{-\lambda\alpha} \nonumber\\
  & =  \Big(\frac{c_2w_0}{u}\Big)^{\lambda\alpha} \Big[\frac{3}{2\lambda} +\frac{\alpha}{2} h_{(2-\lambda)\alpha}\Big(\frac{c_2w_0}{u}\Big)\Big],
\end{align}
where the last step follows by direct calculations.

Let
\begin{align*}
 S^{(0)}_t(s) & := \int_0^{1/(1-s)} \int_{(c_{t,2}w/u)^{\alpha_t} (1+r^{-\lambda_t\alpha_t})^{1/\lambda_t}}^{(c_{t,2}w/u)^{\alpha_t} (1+(c_t(w-1)^+)^{-\lambda_t\alpha_t})^{1/\lambda_t}}(\lambda_t-1)z^{-\lambda_t}\, dz \, w^{\lambda_t\alpha_t-1}\, dw \nonumber\\
 & = \int_0^{1/(1-s)} \bigg[\Big(1+\Big(\frac{c_{t,2}}{c_{t,1}}\frac{s}{1-s}\Big)^{-\lambda_t\alpha_t}\Big)^{1/\lambda_t-1}
  \nonumber\\
  & \hspace{2.5cm} -\Big(1+\Big(\frac{c_{t,2}}{c_{t,1}}(w-1)^+\Big)^{-\lambda_t\alpha_t}\Big)^{1/\lambda_t-1}\bigg] w^{\alpha_t-1}\, dw.
\end{align*}

Since all parameters are suitably bounded and bounded  from below for all $t\in[0,1]$, we have shown that there exists constants $C,\tilde C$ (not depending on $t$ or $s$) such that
\begin{align*}
   \bigg| &  P\{X_{t,1}+X_{t,2}>u, X_{t,1}\le s(X_{t,1}+X_{t,2})\}  - \alpha_t  \Big(\frac{u}{c_{t,2}}\Big)^{-\alpha_t}S^{(0)}_t(s)\bigg|\\
& \le D_1+D_2+D_3+D_4\\
& \le \tilde C\Big[ \Big(\frac{w_0}{u}\Big)^{\lambda_t\alpha_t} \Big(h_{(2-\lambda_t)\alpha_t}\Big(\frac{c_{t,2}w_0}{u}\Big)+1\Big)+ w_0^{(1-\lambda_t)\alpha_t} u^{-\alpha_t}\Big]\\
& \le C \Big[\log\Big(\frac{w_0}{u}\Big)\Big(\frac{w_0}{u}\Big)^{(2\wedge\lambda_t)\alpha_t} + w_0^{(1-\lambda_t)\alpha_t} u^{-\alpha_t}\Big],
\end{align*}
where in the last step we have used that the function $(z,\tau)\mapsto h_\tau(z)/(z^{\min(\tau,0)}\log z)$ is bounded for $z\in[0,z_0]$ (with $z_0\in(0,1)$) and $\tau$ belonging to some compact interval.

Choose $w_0=u^{1-\lambda_t/((2\wedge\lambda_t)+\lambda_t-1)}(\log u)^{-1/(\alpha_t((2\wedge\lambda_t)+\lambda_t-1))}$ (to minimize the order of the right-hand side) to obtain the approximation
\begin{align}  \label{eq:finalapprox}
  P\{R_t>u,\Theta_{t,1}\le s\} = \alpha_t  \Big(\frac{u}{c_{t,2}}\Big)^{-\alpha_t}S^{(0)}_t(s) + O\big(u^{-\tau_t\alpha_t}(\log u)^{\kappa_t}\big)
\end{align}
 with
 $$\tau_t:=\tau(\lambda_t):=\frac{\lambda_t(2\wedge\lambda_t)}{(2\wedge\lambda_t)+\lambda_t-1}>1, \quad \kappa_t:=\kappa(\lambda_t):=\frac{\lambda_t-1}{(2\wedge\lambda_t)+\lambda_t-1},
 $$
  which holds uniformly for all $s,t\in[0,1]$.

In particular, we obtain an approximation for $P\{R_t>u\}$ (for $s=1$) and may thus conclude that
\begin{align}
  P(\Theta_{t,1}\le s\mid R_t>u)
  &= \frac{S_t^{(0)}(s)+ O(u^{-(\tau_t-1)\alpha_t}(\log u)^{\kappa_t})}{S_t^{(0)}(1)+ O(u^{-(\tau_t-1)\alpha_t}(\log u)^{\kappa_t})} \nonumber\\
  & = \frac{S_t^{(0)}(s)}{S_t^{(0)}(1)}+ O(u^{-(\tau_t-1)\alpha_t}(\log u)^{\kappa_t}). \label{eq:angmeasapprox}
\end{align}
Hence, the spectral measure at time $t$ is given by
$$ S_t (A_s) = \frac{S_t^{(0)}(s)}{S_t^{(0)}(1)} $$
with $A_s:=\{\theta\in\mathcal{S}^1 \mid \theta_1\le s\}$ for $s\in[0,1]$.
It is not difficult to check that $S_t(A_s)$ is a continuously differentiable function of $\alpha_t,\lambda_t$ and $c_{t,2}/c_{t,1}$ and that the partial derivatives are bounded for $\alpha_t\in[\underline\alpha,\bar\alpha]$, $\lambda_t\in[\underline\lambda,\bar\lambda]$, $c_{t,1}/c_{t,2}\in[\underline c,\bar c]$ and $s\in[0,1]$. Hence, the Lipschitz continuity of the parameter functions implies that $t\mapsto S_t(A_s)$ is uniformly Lipschitz continuous. Consequently, condition (IS) is met if $k_nh_n\to 0$.

Moreover, for $|r-t|\le h_n$, \eqref{eq:finalapprox} and  the assumed Lipschitz continuity of the parameter functions imply
$ P\{R_r>u_n\} = P\{R_t>u_n\}(1+o(1))$, provided $h_n\log u_n\to 0$. Hence, if we assume $h_n\log n\to 0$, then $\nu_{n,t}(u_n,\infty)=2nh_n\alpha_t(u_n/c_{t,2})^{-\alpha_t} S_t^{(0)}(1)(1+o(1))$,
$$ u_{n,t} = c_{t,2} \bigg(\frac{k_n}{2nh_n\alpha_tS_t^{(0)}(1)}\bigg)^{-1/\alpha_t} (1+o(1)),
$$
and condition (L) is fulfilled. Combine this with \eqref{eq:angmeasapprox} to conclude
\begin{equation*}
  \sup_{A\in\AA^*}\big| P(\Theta_t\in A_s\mid R_t>vu_{n,t}) - S_t(s)\big| = O\Big(\Big(\frac{k_n}{nh_n}\Big)^{\tau_t-1}\Big(\log \frac{nh_n}{k_n}\Big)^{\kappa_t}\Big)
\end{equation*}
uniformly for all $t\in[0,1]$ and $v\in[1/2,3/2]$. The assumed Lipschitz continuity of $t\mapsto\lambda_t$ implies that $t\mapsto \tau_t$ and $t\mapsto\kappa_t$ are Lipschitz continuous, too. Thus the right hand side is of the same order for all $t$ in an interval of length $2h_n$, provided $h_n\log(k_n/(nh_n))\to 0$, which is satisfied if $h_n\log n\to 0$. Therefore, Conditions (US), (US$^*$) and (B) are met with
$$ q_n=q_n^*=q_n'= \max\bigg( \sup_{t\in [0,1]}\Big(\frac{k_n}{nh_n}\Big)^{\tau(\lambda_t)-1}\Big(\log \frac{nh_n}{k_n}\Big)^{\kappa(\lambda_t)}, h_n\bigg).
$$
Since the function $\tau$ is strictly increasing on $(1,\infty)$, the supremum on the right hand side is attained at the point of minimum of the function $t\mapsto\lambda_t$.

To sum up, the conditions of Theorem 2.3 are met if
$ h_n\log n\to 0$, $\log^3 h_n=o(k_n)$, $k_nh_n\log^2(k_n/h_n)\to 0$ and
$$  \Big(\frac{k_n}{h_n}\Big)^{1/2} \Big(\frac{k_n}{nh_n}\Big)^{\tau(\underline\lambda)-1}\Big(\log \frac{nh_n}{k_n}\Big)^{\kappa(\underline\lambda)} \to 0.
$$
For example, one may choose $h_n\sim C n^{-\xi}$ and $k_n\sim \tilde Cn^\zeta$ for some $C,\tilde C>0$, $\xi\in\big(0,(\tau(\underline\lambda)-1)/(2\tau(\underline\lambda)-1)\big)$ and $\zeta\in(0,\xi)$.

\section{Tail bounds on the limit distributions of $T_n^{(KS)}$}

As explained in Section 3, in general the limit distribution of the test statistics $T_n^{(KS)}$ and $T_n^{(CM)}$ under the null hypothesis will depend on the underlying spectral measure, if the family $\AA$ of test sets is not linearly ordered, which will typically be the case in dimensions greater than 2. Here we outline how to obtain critical values for the Kolomogorv-Smirnov type test using general tail bounds for Gaussian processes. Since these bounds are usually quite crude, in most instances the resulting tests will be very conservative and the approach described in Section 3 of the accompanying paper will be preferable.

A set-indexed Brownian motion pertaining to a probability measure $Q$ is defined as a centered Gaussian process $(W(A))_{A\in\AA}$ with  $Cov(W(A),W(B))=Q(A\cap B)$. Covariance calculations show that the limit process $\ZZ$ defined in Corollary 3.2 has the same distribution as $W([0,t]\times A)-tW([0,1]\times A)-S_1(A)W([0,t]\times \sphere)+tS_1(A)W([0,1]\times\sphere)$, $t\in[0,1], A\in\AA$, where $W$ is a set-index Brownian motion pertaining to the product of the uniform distribution on $[0,1]$ and the spectral measure $S_1$. Therefore upper bounds on probabilities that the supremum of a set-indexed Brownian exceeds a critical value imply corresponding bounds on $P\{\sup_{t\in [0,1], A\in\AA} |\ZZ_t(A)|>c\}$.

In what follows, we give more details of this approach in the case $d=3$ when the left tail of all coordinates $X_t$ are lighter than the right tails such that the spectral measure is concentrated on ${\mathcal S}^2_+:={\mathcal S}^2\cap [0,\infty)^3$ and we denote the spectral measure $S_1$ by $S$, to keep the notational burden as light as possible. We use the sum norm (although this choice is not essential) so that ${\mathcal S}^2_+=\{(x_1,x_2,1-x_1-x_2)\mid x_1,x_2\ge 0, x_1+x_2\le 1\}$.
Natural choices for $\AA$ are, e.g.,  $\AA_1 := \big\{ A(a,b):=\{\bfx\in  {\mathcal S}^2_+| x_1+x_2\le a, x_2/x_1\le b/(1-b)\}\mid a,b\in[0,1]\big\}$ and $\AA_2 := \big\{ \{\bfx\in  {\mathcal S}^2_+| x_1\le a_1, x_2\le a_2\}\mid a_1,a_2\in[0,1]\big\}$. Here we focus on $\AA_1$, which consists of triangles of points below a ray from the origin into the first quadrant with a sum norm less than or equal to a given value $a$.

The intrinsic semi-metric $\rho_{\ZZ}$ of the Gaussian process $\ZZ$ satisfies
\begin{align*}
 \rho_{\ZZ}^2 & \big((s,A),(t,B)\big)\\
  & := E\big((\ZZ_s(A)-\ZZ_t(B))^2\big)\\
  & = (t-s)\big(s S(A)(1-S(A))+(1-t)S(B)(1-S(B))\big)\\
  & \hspace*{1cm}
  +s(1-t)S(A\triangle B)(1-S(A\triangle B))\\
  & \le \frac 14 (|t-s|+S(A\triangle B)),
\end{align*}
where in the intermediate step w.l.o.g.\ we have assumed  $s\le t$; in particular, $\sup_{t\in[0,1],A\in\AA_1}$ $Var(\ZZ_{A})$ $\le 1/16$ (and equality holds if $S(A)=1/2$ for some $A\in\AA$). Using this result, it is not difficult to bound the covering number of $[0,1]\times\AA_1$ by balls of radius $\eps$ w.r.t.\ $\rho_{\ZZ}$ by a multiple of $\eps^{-6}$ that does not depend on $S$. Hence, Proposition A.2.7 of \cite{vdVW96} yields a bound of the form
$$ P\Big\{\sup_{t\in[0,1],A\in\AA_1} |\ZZ_t(A)|>\lambda\Big\}\le C \lambda^6 \bar\Phi(4\lambda) \sim \frac{C}4 \lambda^5 \e^{-8\lambda^2}
$$
for some universal constant $C$ and sufficiently large $\lambda$. Using the aforementioned relationship to set-indexed Brownian motions and results from \cite{AdlerBrown86}, one may establish bounds of the order $\lambda^4 \e^{-8\lambda^2}$, provided $S$ is sufficiently smooth; see also \cite{Samo91} for similar results on richer families $\AA$ of sets.

\bibliographystyle{imsart-number}
\bibliography{ProcLit}